\documentclass[11pt,oneside,english,reqno]{amsart}
\usepackage[T1]{fontenc}
\usepackage[utf8]{inputenc}
\usepackage[a4paper]{geometry}
\usepackage{mathrsfs}
\usepackage{amstext}
\usepackage{amsthm}
\usepackage{amssymb}
\usepackage{bbm}
\usepackage{shuffle}
\usepackage{color}
\usepackage{hyperref}
\usepackage{enumerate}
\usepackage{verbatim} 
\usepackage{stmaryrd}
\usepackage{enumitem}
\usepackage{graphicx}
\usepackage{todonotes}

\makeatletter
\numberwithin{equation}{section}
\numberwithin{figure}{section}
\theoremstyle{plain}
\newtheorem{thm}{\protect\theoremname}
\numberwithin{thm}{section}
\theoremstyle{definition}
\newtheorem{defn}[thm]{\protect\definitionname}
\theoremstyle{remark}
\newtheorem{rem}[thm]{\protect\remarkname}
\theoremstyle{plain}
\newtheorem{lem}[thm]{\protect\lemmaname}
\theoremstyle{plain}
\newtheorem{prop}[thm]{\protect\propositionname}
\theoremstyle{plain}
\newtheorem{cor}[thm]{\protect\corollaryname}
\theoremstyle{plain}
\newtheorem{example}[thm]{\protect\examplename}

\makeatother

\usepackage{babel}
\providecommand{\corollaryname}{Corollary}
\providecommand{\definitionname}{Definition}
\providecommand{\lemmaname}{Lemma}
\providecommand{\propositionname}{Proposition}
\providecommand{\remarkname}{Remark}
\providecommand{\theoremname}{Theorem}
\providecommand{\examplename}{Example}

\newcommand{\cA}{\mathcal{A}}

\newcommand{\cD}{\mathcal{D}}

\newcommand{\cF}{\mathcal{F}}
\newcommand{\cG}{\mathcal{G}}
\newcommand{\cH}{\mathcal{H}}
\newcommand{\cI}{\mathcal{I}}

\newcommand{\cO}{\mathcal{O}}
\newcommand{\cP}{\mathcal{P}}

\newcommand{\cT}{\mathcal{T}}
\newcommand{\cU}{\mathcal{U}}
\newcommand{\cV}{\mathcal{V}}

\newcommand{\cOT}{\cO\cT}
\newcommand{\cOF}{\cO\cF}
\newcommand{\cUT}{\cU\cT}
\newcommand{\cUF}{\cU\cF}

\newcommand{\LL}{\mathbb{L}}

\newcommand{\NN}{\mathbb{N}}

\newcommand{\RR}{\mathbb{R}}

\newcommand{\XX}{\mathbb{X}}
\newcommand{\YY}{\mathbb{Y}}
\newcommand{\ZZ}{\mathbb{Z}}
\newcommand{\one}{{\mathbbm 1}}

\newcommand{\myspan}{\mbox{span}}

\newcommand{\la}{\langle}
\newcommand{\ra}{\rangle}
\newcommand{\scalar}[1]{\left\la #1 \right\ra}
\newcommand{\brackets}[1]{\left\llbracket #1 \right\rrbracket}

\newcommand{\Id}{\mathrm{Id}}
\newcommand{\diag}{\mathrm{diag}}
\newcommand{\Lip}{\mbox{Lip}}
\newcommand{\dist}{\mbox{dist}}

\newcommand{\norm}[1]{\left\lVert #1\right\rVert}
\newcommand{\abs}[1]{\left\lvert #1\right\rvert}

\newenvironment{tree}
               {\begin{tikzpicture}[every node/.style={circle,draw,fill,minimum size=3pt,inner sep=0pt,outer sep=0pt},line cap=round,baseline = .1pt]}
               {\end{tikzpicture}}

\newcommand{\mynode}[1]{node[label=right:{\tiny #1}]{}}

\newcommand{\tomato}[1]{
    \begin{tree}
        \draw (0,.1)\mynode{{#1}};
    \end{tree}}

\newcommand{\ladderTwo}[2]{\begin{tree}
                            \node[label=right:{\tiny #1}]{}[grow'=up, level distance = 2ex, sibling distance = 2.5ex]
                                child{node [label=right:{\tiny #2}] {}};
                        \end{tree}}

\newcommand{\ladderThree}[3]{\begin{tree}
                            \node[label=right:{\tiny #1}]{}[grow'=up, level distance = 2ex, sibling distance = 2.5ex]
                                child{node [label=right:{\tiny #2}] {}
                                    child{node [label=right:{\tiny #3}]{}}
                                };
                        \end{tree}}

\newcommand{\cherry}[3]{\begin{tree}
                            \node[label=right:{\tiny #1}]{}[grow'=up, level distance = 2ex, sibling distance = 2.5ex]
                                child{node [label=right:{\tiny #2}] {}}
                                child{node [label=right:{\tiny #3}] {}};
                        \end{tree}}

\newcommand{\righttree}[4]{\begin{tree}
                            \node[label=right:{\tiny #1}]{}[grow'=up, level distance = 2ex, sibling distance = 2.5ex]
                                child{node [label=right:{\tiny #2}] {}}
                                child{node [label=right:{\tiny #3}] {}
                                    child{node [label=right:{\tiny #4}] {}}
                                    };
                        \end{tree}}

\newcommand{\lefttree}[4]{\begin{tree}
                            \node[label=right:{\tiny #1}]{}[grow'=up, level distance = 2ex, sibling distance = 2.5ex]
                                child{node [label=right:{\tiny #2}] {}
                                    child{node [label=right:{\tiny #4}] {}}}
                                child{node [label=right:{\tiny #3}] {}};
                        \end{tree}}

\newcommand{\Ytree}[4]{\begin{tree}
                            \node[label=right:{\tiny #1}]{}[grow'=up, level distance = 2ex, sibling distance = 2.5ex]
                                child{node [label=right:{\tiny #2}] {}
                                    child{node [label=right:{\tiny #3}] {}}
                                    child{node [label=right:{\tiny #4}] {}}};
                        \end{tree}}

\newcommand{\bushThree}[4]{\begin{tree}
                            \node[label=right:{\tiny #1}]{}[grow'=up, level distance = 2ex, sibling distance = 2.5ex]
                                child{node [label=right:{\tiny #2}] {}}
                                child{node [label=right:{\tiny #3}] {}}
                                child{node [label=right:{\tiny #4}] {}};
                        \end{tree}}

\title{Flow techniques for non-geometric RDEs on manifolds}
\author{Hannes Kern and Terry Lyons}
\date{\today}

\begin{document}

\maketitle

\begin{abstract}
    In 2015, Bailleul presented a mechanism to solve rough differential equations by constructing flows, using the log-ODE method. We extend this notion in two ways: On the one hand, we localize Bailleul's notion of an almost-flow to solve RDEs on manifolds. On the other hand, we extend his results to non-geometric rough paths, living in any connected, cocommutative, graded Hopf algebra. This requires a new concept, which we call a pseudo bialgebra map. We further connect our results to Curry et al (2020), who solved planarly branched RDEs on homogeneous spaces.
\end{abstract}

\tableofcontents

\section{Introduction}

\subsection{The story in a nutshell}

\noindent We are interested in analyzing the flow operation associated with a rough differential equation: Consider a (for now geometric) rough path $\XX$ and a set of vector fields $V_i$ on $\RR^d$. Then, the solution to the RDE
\begin{equation}\label{eq:RDE}
    dY_t = \sum_{i=1}^d V_i(Y_t)d\XX^i_t
\end{equation}
follows the Davie's series \cite{davie},\cite{FV10}:
\begin{equation}\label{eq:expansion}
    Y_t = Y_s + V_i(Y_s) X^i_{s,t} + V_iV_j \Id(Y_s) \XX^{i,j}_{s,t} + o(\abs{t-s}^{3\alpha})\,,
\end{equation}
if $\XX$ is $\alpha$-Hölder continuous. Thus, the flow operator of the solution $Y_t$, given by $\mu_{s,t}(y) = Y_t$ for $y= Y_s$ has the local expansion
\begin{equation}\label{eq:operatorY}
\mu_{s,t} = \Id + X^i_{s,t} V_i + \XX^{i,j}_{s,t} V_i V_j + \dots\,,
\end{equation}
where one recovers \eqref{eq:expansion} by applying $\mu_{s,t}$ to the identity map $\Id$. We want to analyze this perspective towards solving RDEs: We show that Davie's formula is a direct result of the interaction between a rough path and what we call a \emph{pseudo bialgebra map} (see Def \ref{def:pseudoBi}). For geometric RDEs, this directly leads to the perspective of \cite{bailleul15}. For non-geometric rough paths \cite{gubinelli04},\cite{HK15} one only needs to identify a new pseudo bialgebra map and apply the same machinery to get a Davie's formula. Furthermore, this approach easily extends to RDEs on manifolds, which gives us a new framing for non-geometric rough paths on manifolds. We especially show that our techniques generalize the results of \cite{emilio22} and \cite{curry20}. To summarize, the goals of this paper can be seen as the following:
\begin{itemize}
    \item We extend the flow techniques presented in \cite{bailleul15} to non-geometric rough paths. This requires us to replace the higher order differential operator $\XX^{ij}_{s,t}V_i V_j$ with appropriate differential operators $\XX^\tau_{s,t} D^\tau$, where $\tau$ are a basis of the more general Hopf algebra.
    \item By using vector fields $V_i, i=1,\dots, n$ living on a manifold $M$, we can extend this machinery to solve RDEs on manifolds. If $\XX$ is geometric, this approach works for all manifolds $M$, whereas a branched rough path requires a connection $\nabla$ on $M$ (see \cite{geometricViewRP}). We also analyze the case in which $\XX$ lives in any connected, graded, cocommutative Hopf algebra, which leads us to the previously mentioned pseudo bialgebra maps.
    \item This perspective allows us to reframe existing results, focusing on non-geometric RDEs on manifolds. We explain this new approach and show that it generalizes \cite{emilio22}, \cite{geometricViewRP} and \cite{curry20}.
\end{itemize}
Before we explain our approach, let us take a look at already existing results: While flow techniques for rough paths have been known since the nineties \cite{LYONS1997135}, they got their most concise presentation somewhat recently in \cite{bailleul15}. In this paper, Bailleul presents a simple machinery to derive the operator \eqref{eq:operatorY} with the use of a log-ODE method \cite{logODECastell}. If $\LL_{s,t}$ is the logarithm (in the Lie-sense) of $\XX_{s,t}$, it is possible to associate $\LL_{s,t}$ with a vector field $\cF(\LL_{s,t})$ in a canonical way. For small $\abs{t-s}$, one can then find the solution to the starting value problem 
\begin{equation*}
    \begin{cases}
        dZ_r = \cF(\LL_{s,t})(Z_r) \\
        Z_0 = z
    \end{cases}
\end{equation*}
and set $\tilde \mu_{s,t} z = Z_1$. \cite{bailleul15} then introduces a sewing lemma for almost-flows to ``sew'' together $\tilde\mu$ to get the solution flow $\mu$, leading to the actual solution of the RDE \eqref{eq:RDE}. He later extended this technique in the follow-up paper \cite{bailleul19}, where the authors introduced the notion of a \emph{rough flow}.

For non-geometric rough paths, flow techniques are not standard at the moment, although they are tremendously helpful in generalizing \eqref{eq:expansion} to RDEs on manifolds. For branched rough paths, originally introduced in \cite{GUBINELLI2010693}, the Davie's formula was shown in \cite{HK15}:
\begin{equation}\label{eq:branchedDavis}
    Y_t \approx \sum_{\tau}X^\tau_{s,t} (\cF(\tau) \Id) (Y_s)
\end{equation}
where the sum goes over all rooted trees up to a specific number of nodes, and for certain differential operators $\cF(\tau)$ depending on the vector fields $V_1,\dots, V_d$. The differential operators $\cF(\tau)$ first appeared in the theory of Butcher series, where they are called elementary differentials of $V_1,\dots, V_d$ (A term going back to Butcher \cite{butcher_1963} and Merson \cite{Merson}, more recent references include \cite{MKW08}, \cite{Lundervold2009HopfAO}.). This theory has been extended to RDEs on manifolds in \cite{geometricViewRP}, by replacing $\Id$ with some smooth map $\phi:M\to \RR^d$. Here, Weidner makes the observation that \eqref{eq:branchedDavis} does lead to coordinate dependent solutions when one takes the classical elementary differentials. However, they also demonstrate that one can construct different elementary differentials as long as the manifold is equipped with a flat, torsion-free connection or one has non-planarly branched rough paths, see Remark 4.35 of \cite{geometricViewRP}. We will analyze this map in more detail in Section \ref{subsec:MyMap}.

For the level $N=2$ case, one can get rid of the assumption that the connection is flat and torsion-free: In \cite{emilio22}, the authors use the theory of Itô integrals on manifolds \cite{paperEmery}, \cite{StochCalcEmery} as a motivation to add a correction term (similar to the Itô-Stratonovich correction term) to the rough-path integral to construct a coordinate independent integral, as long as the manifold is equipped with a connection. In the paper, they also derive a Davie's formula, allowing us to compare their results with our approach.

Similar results are well known in the Butcher-series theory: If $X$ is simply a smooth path, all of the above still holds true, as long as we replace the rough path with its signature. The theory of Butcher series then tells us how to choose $\cF(\tau)$ to get the correct solution on $M$, as long as it is a homogeneous space.

Curry et. al \cite{curry20} managed to use this inside to generalize this approach to the (to our knowledge) most general setting currently known: If $\XX$ is a planar branched rough path and $M$ is a homogeneous space, we can use the post-Lie algebra structure of the vector fields of $M$ to construct the differential operators $\cF(\tau)$ for each tree $\tau$ and solve the RDE.

The main structure of these approaches can be summarized into two main components: One first needs to identify the differential operators $\cF(\tau)$ and then argue that the Davie's formula \eqref{eq:branchedDavis} indeed generates a flow $\mu_{s,t}$. In this paper, we identify the minimum assumption on a map $\cF$, mapping the underlying Hopf algebra of $\XX$ into differential operators, such that $\mu_{s,t}$ becomes a flow. Interestingly, there seem to be two approaches to construct $\mu_{s,t}$ from its approximation \eqref{eq:branchedDavis}: Curry et. al. prefer to treat $\XX$ as an infinite series (that is, they do not truncate the underlying Hopf algebra), and show that the Taylor expansion converges. On the other hand, Baileul et.al. construct the flow $\mu_{s,t}$ with the use of a sewing lemma and log-Ode techniques, as described above. We will follow the second approach, as it avoids assumptions on the asymptotic dimension of the Hopf algebra.

To do so, we extend the sewing lemma for flows (\cite{bailleul15} Theorem 2.1) to a local setting, which allows us to use it on any finite-dimensional manifold. This allows us to find the minimal assumption on $\cF$ to make this machinery work, which we call a \emph{pseudo bialgebra map}. 

\begin{defn}\label{def:pseudoBi}
$\cF$ is called a pseudo bialgebra map, if $\cF(\one) = \Id$ and
\begin{align*}
    \cF(\tau\star\sigma) &= \cF(\tau)\circ\cF(\sigma)\\
    \cF(\Delta\tau)(\phi\otimes\psi) &= \cF(\tau)(\phi\cdot\psi)\,.
\end{align*}
\end{defn}

\noindent I.e. we want $\cF$ to be an algebra morphism as a map from $(\cH,\star)$ to the space of differential operators equipped with composition as a product, and we want the coproduct on $\cH$ to be "dual" to the product between functions.

Using the log-ODE method gives us two important insides into $\cF$: On the one hand, the log-ODE method only depends on $\cF$ restricted to the \emph{primitive elements} $\cP$ of the Hopf algebra $\cH$. This is not very surprising, as the Theorem of Milnor-Moore gives us a $1-1$ correspondence between the Hopf algebra $\cH$ and the Lie algebra of its primitive elements $\cP$ \cite{Milnor-Moore}. However, this does give us an alternative perspective on the pseudo bialgebra map $\cF$, as they are in a $1-1$ correspondence to Lie algebra maps mapping $\cP$ to vector fields. Furthermore, the log-ODE method does not rely on the post-Lie structure of the vector fields on $M$, which allows us to generalize the results from \cite{curry20} to general manifolds, as long as they are equipped with a connection. In this case, we can write down the map $\cF$ explicitly, leading to the map from \cite{geometricViewRP}: If our Hopf algebra is the Munthe-Kaas-Wright algebra \cite{MKW08} $\cH_{MKW}$ over ordered forests, the map $\cF$ mapping dual forests of $\cH_{MKW}$ to differential operators is given by
\begin{align*}
    \cF(\tomato i)\psi &= V_i\psi \\
    \cF(\tau_1,\dots,\tau_k)\psi &= \nabla^k \psi (\cF(\tau_1),\dots,\cF(\tau_k)) \\
    \cF([\tau_1,\dots,\tau_k]_i) &= \nabla^k V_i (\cF(\tau_1),\dots,\cF(\tau_k))\,,
\end{align*}
where $\tau_1,\dots,\tau_k$ are rooted, ordered trees and $[\tau_1,\dots,\tau_k]_i$ denotes the rooted, ordered tree with root $i$ and children $\tau_1,\dots,\tau_k$. $\nabla^k$ denotes the $k$-th total covariant derivative. We show that the above map is a pseudo bialgebra map generalizing the post-Lie map used in \cite{curry20}.

Last but not least, we briefly discuss if the solution $Y$ can be seen as a rough path itself. The short answer to the question is \emph{only in the geometric case}, which is already well known. However, we show a shorter proof than the classical one using only algebraic identities to show that in the geometric case, $Y$ becomes a rough path on $M$. The non-geometric case remains an interesting field for further research, although we recommend \cite{ferrucci22} for current advances in that direction.

\begin{rem}
    We will use the $\alpha$-Hölder setting for rough paths with $\alpha\in(0,1)$, as it is more common in this area and allows us to link our results to compare our results with \cite{emilio22}, \cite{bailleul15}, \cite{curry20} and more. However, we heavily suspect that our results hold just as well in the $p$-variation setting.
\end{rem}

\subsection{The structure of this paper}

Since the concept of an almost-flow can be constructed independently from rough path theory, we start this paper by introducing the concept of \emph{local flows} in Section \ref{sec:flows}. This construction allows us to extend the results of \cite{bailleul15} from flows on Banach spaces to flows on manifolds.

In Section \ref{sec:prelim}, we recall some facts about Hopf algebras and rough path theory and fix our notation. We also use this chance to briefly discuss the three most commonly used Hopf algebras in rough path theory: The tensor algebra for geometric rough paths, the Grossman-Larson algebra for non-planarly branched rough paths, and the Munthe-Kaas-Wright algebra for planarly branched rough paths. Finally, we recall some notations from differential geometry and linear connections.

Afterward, we introduce the concept of a \emph{pseudo bialgebra map} in Section \ref{sec:FlowsGenByRP}, which is the most general map $\cF$ from a Hopf algebra into the set of differential operators on $M$, such that the log-ODE generates an almost-flow. We further show that pseudo bialgebra maps are in a 1-1 correspondence to Lie algebra maps mapping primitive elements to vector fields. Finally, we show that a rough path and  a pseudo bialgebra map do indeed generate a local almost-flow on $M$, which has the Davie's formula \eqref{eq:branchedDavis}.

In section \ref{sec:elementaryDiff}, we go over the commonly used techniques to construct elementary differentials on the tensor algebra, the GL-algebra as well as the MKW-algebra. We show that all of these techniques generate pseudo bialgebra maps on their respective algebras and further show that on the MKW-algebra, the construction of $\cF$ can be generalized to all smooth manifolds equipped with a connection. This also allows us to relate the results from \cite{curry20} with \cite{emilio22}, who constructed rough path solutions on manifolds without a post-Lie structure.

\begin{rem}
    \cite{emilio22} only solves the RDE \ref{eq:RDE} for level $N=2$. It has been generalized to higher levels in \cite{ferrucci22} for quasi-geometric rough paths. We are convinced that there is another pseudo bialgebra map on the quasi-shuffle algebra hiding in this result, but at this moment, we are not sure what exactly this map looks like.
\end{rem}

\noindent One advantage \cite{emilio22} has over our theory, is that it can make sense of the solution of \eqref{eq:RDE} as a rough path on a manifold. In general Hopf algebras, it is not even clear how to define a rough path on a manifold. However, we do briefly discuss how to construct this for the geometric case in Section \ref{sec:RPonManifolds}, and present a new, mainly algebraic proof that the solution to \eqref{eq:RDE} is a rough path on $M$.

\section{Local flows and almost-flows}\label{sec:flows}

\noindent The goal of this section is to construct flows from \emph{almost-flows}, a concept introduced in \cite{bailleul15}. The main idea is that for a given almost-flow $\mu_{s,t}$, one can sew together a flow-map
\begin{equation}\label{eq:flowSewing}
\eta_{s,t} = \mu_{t_{N-1},t_N}\circ\dots\circ\mu_{t_1, t_2}\,,
\end{equation}
where $s=t_1<\dots< t_N = t$ is a dyadic partition of $[s,t]$. Given a manifold $M$, we can  express any flow in an open subset $U\subset\RR^d$ by using coordinates. So if we can guarantee, that we do not ``fall off'' $U$ during the sewing procedure, we could apply the sewing lemma for flows to get a flow $\eta_{s,t}$ for small enough $\abs{t-s}$. Unfortunately, we can not give this guarantee a priori: If the control $X$ of \eqref{eq:RDE} is $\alpha$-Hölder continuous, it is reasonable to assume that $\abs{\mu_{s,t} x - x}\lesssim \abs{t-s}^\alpha$. Taking the $N$-th dyadic partition of $[s,t]$, it then follows that
\begin{equation*}
    \abs{\bigcirc_{i=1}^{2^N} \mu_{t_{i-1},t_i} x-x} \lesssim \sum_{i=1}^{2^N} 2^{-N\alpha}\abs{t-s}^\alpha \to\infty  
\end{equation*}
as $N\to\infty$, so one really needs to invoke the almost sewing properties of the almost-flow $\eta$. We will show in subsection \ref{subsec:BanachspaceFlows} that these assumptions indeed suffice to do the sewing in any open set $U$, provided that $\abs{t-s}$ is small enough. In subsection \ref{subsec:ManifoldFlows}, we will then use this to conclude that we can sew any flow on the manifold.

We should note that we will be working in a finite-dimensional setting. This has the advantage, that any closed ball is compact, so by showing that $\abs{\bigcirc_{i=1}^{2^N} \mu_{t_{i-1},t_i}x-x}\le C$ for some $C\ge 0$, we get that the sewing procedure takes place in some compact set, allowing us to only assume local bounds. In an infinite-dimensional setting, we would need to assume global bounds on the flows.

Throughout this section, we will use the following notation:

\begin{itemize}
    \item The simplex is denoted by $\Delta_T := \{(s,t)\in [0,T]^2~\vert~ s\le t\}$.
    \item The diagonal is denoted by $\diag_T := \{(t,t)\in [0,T]^2\}$.
    \item The open ball is denoted by $B_x(r) := \{y\in\RR^d~\vert~\abs{x-y} < r\}$.
\end{itemize}


\subsection{Local flows in Banach spaces}\label{subsec:BanachspaceFlows}

Throughout this section, let us fix an open set $U\subset\RR^d$ and a time $T>0$. It is unreasonable to assume that there is a $T(U) > 0$, such that the almost-flow $\mu_{s,t}(x)$ is well-defined whenever $\abs{t-s}\le T(U)$. Indeed, if $x$ approaches the border of $U$, we would assume $T$ to go to zero. Thus, we will use a setting, where we allow $T$ to continuously depend on $x$. The standard setting for this (e.g. \cite{ApplicationsLieGroups}) is defined as follows:

\begin{defn}
    Let $O$ be an open set, such that $\diag_T\times U\subset O \subset \Delta_T\times U$. We call $O$ an admissible domain, if for any compact set $K\subset U$, there is a $0<T(K)\le T$, such that for each $x\in K$ and $\abs{t-s}\le T(K)$, $(s,t,x)\in O$, and for any $(s,t,x)\in O$ and $s\le u\le t$: $(s,u,x)\in O$.
\end{defn}

\noindent We can now define a local flow to be a map on an admissible domain, such that the classical flow property holds:

\begin{defn}
    Let $O$ be an admissible domain. We say $\eta:O\mapsto U$, $(s,t,x)\mapsto\eta_{s,t} x$ is a local flow, if for all $(s,t,x)\in O$ and $s\le u\le t$ such that $(u,t,\eta_{s,u}x)\in O$, it holds that
    \begin{equation*}
        \eta_{s,t}x = \eta_{u,t}\circ\eta_{s,u} x\,.
    \end{equation*}
\end{defn}

\noindent We want to find a notation of \emph{almost-flows}, such that \eqref{eq:flowSewing} converges for $N\to\infty$. For the classical sewing techniques (originally introduced in \cite{gubinelli04}, \cite{Feyel2006}), we can sew together two-parameter processes $A_{s,t}$, as long as the following coherence property holds:
\begin{equation*}
    \abs{A_{s,t}-(A_{s,u}+A_{u,t})} \lesssim \abs{t-s}^{1+\epsilon}
\end{equation*}
for some $\epsilon>0$ and all $s\le u\le t$. In a more general setting, one can replace the plus with more general products, to arrive at Terry Lyons's notion of almost-multiplicative functionals \cite{originalArticleRP}:
\begin{equation*}
    \norm{X_{s,t} - X_{s,u}\star X_{u,t}} \le \abs{t-s}^{1+\epsilon}\,,
\end{equation*}
for some appropriate norm. On the surface, this looks precisely like the result we need, if one replaces $\star$ with the composition product $\circ$. However, as discovered in \cite{bailleul15}, one needs to be a bit more careful with flows than with multiplicative functionals: If we are given a composition of flows $\mu_{u,t}\circ\mu_{s,u}$ and replace $\mu_{s,u}$ with $\mu_{u,\tilde u}\circ\mu_{s,\tilde u}$, it in general does not hold true that $\abs{(\mu_{s,u} x - \mu_{\tilde u,u}\circ\mu_{s,\tilde u} x)}$ being small implies that $\abs{\mu_{u,t}\circ(\mu_{s,u} x - \mu_{\tilde u,u}\circ\mu_{s,\tilde u} x)}$ is small. Thus, we need additional Lipschitz assumptions on our almost-flow $\mu$:

\begin{defn}
    We say $\mu:O\mapsto U$, $(s,t,x)\mapsto\mu_{s,t} x$ is a local almost-flow, if for all compact sets $K\subset U$ we have

    \begin{itemize}
        \item Lipschitz continuity: There exists constants $L(K,s,t)$ with $L(K,s,s) = \lim_{\abs{t-s}\to 0}L(K,s,t) = 1$ such that for all $x,y\in K$, $\abs{t-s}\le T(K)$:
        \begin{equation*}
            \abs{\mu_{s,t} x-\mu_{s,t} y}\le L(K,s,t)\abs{x-y}
        \end{equation*}

        \item Hölder continuity: There is an $\alpha\in(0,1)$ (independent of $K$) and a constant $B(K)$ such that for all $x\in K$, $\abs{t-s}\le T(K)$:
        \begin{equation*}
            \abs{\mu_{s,t} x-x} \le B(K)\abs{t-s}^\alpha\,.
        \end{equation*}
        \item almost-flow property: There is an $\epsilon>0$ (independent of $K$) and a constant $C(K)$ such that for all $x, y\in K$, $\abs{t-s}\le T(K)$ and $(u,t,\mu_{u,s}x), (u,t,\mu_{u,s} y)\in O$:
        \begin{equation*}
            \abs{\mu_{s,t}x-\mu_{u,t}\circ\mu_{s,u} x}\le C(K)\abs{t-s}^{1+\epsilon}\,.
        \end{equation*}
    \noindent as well as the joint Lipschitz-almost-flow property holds:
        \begin{equation*}
            \abs{(\mu_{s,t}-\mu_{u,t}\circ\mu_{s,u})(x-y)} \le C(K) \abs{t-s}^{1+\epsilon}\abs{x-y}
        \end{equation*}
    \end{itemize}
\end{defn}

\noindent We fix a local almost-flow $\mu$. Let $P_k([s,t]) := \left\{[t_i,t_{i+1}]~\middle\vert~ t_i = s + 2^{-k} i (t-s), i= 0,\dots,2^{k}-1\right\}$ be the $k$-th dyadic partition of $[s,t]$ and let
\begin{equation*}
    \mu_{s,t}^k := \bigcirc_{[t_i,t_{i+1}]\in P_k([s,t])}~\mu_{t_i,t_{i+1}}
\end{equation*}
be the sewn-together flow under $P_k([s,t])$. As mentioned above, it is a priori not clear that $\abs{\mu^k_{s,t} x-x}$ is even bounded. The next lemma shows, that the almost-flow property guarantees that:

\begin{lem}\label{lem:flowTechnical1}
    Let $K\subset O$ be a compact set, and $x\in K$ be an inner point. Chose $R>0$ in such a way, that $B_x(R)\subset K$. Further let $0<r<R, \hat c>0$ and set $T(r,\hat c) = \min\left(\left(\frac{R-r}{\hat c}\right)^{\frac 1\alpha},T(K)\right)$. Then, there is $\hat c, c_1,c_2, c_3 > 0$ such that for all $y\in B_x(r)$, $\abs{t-s}\le T(r,\hat c)$ we have that $\mu^k_{s,t} y$ is well-defined and that
    \begin{align}
        \abs{\mu^k_{s,t} y-\mu_{s,t} y} &\le c_1\abs{t-s}^{1+\epsilon}\label{ineq:mu_k1} \\
        \abs{\mu^k_{s,t} y-y} &\le c_2\abs{t-s}^{\alpha}\label{ineq:mu_k2}\\
        \abs{(\mu^k_{s,t} -\mu_{s,t})(x-y)}&\le c_3\abs{t-s}^{ 1+\epsilon}\abs{x-y}\,.\label{ineq:mu_k3}
    \end{align}
\end{lem}

\begin{proof}
    By the definition of an almost-flow, all of these properties are obvious for $k=1$. We show $k>1$ with induction: Set $u=\frac{t+s}{2}$ and decompose
    \begin{equation*}
        \mu^k_{s,t} y -\mu_{s,t} y = \underbrace{(\mu_{u,t}^{k-1}\mu_{s,u}^{k-1} y - \mu_{u,t}\mu_{s,u}^{k-1} y)}_{(I)} + \underbrace{(\mu_{u,t}\mu_{s,u}^{k-1} y- \mu_{u,t} \mu_{s,u} y)}_{(II)} + \underbrace{(\mu_{u,t}\mu_{s,u} y- \mu_{s,t} y)}_{(III)}\,.
    \end{equation*}
    By induction hypothesis, $\mu^{k-1}_{s,u} y$ is well-defined and \eqref{ineq:mu_k2} implies $\mu^{k-1}_{s,u} y\in B_x(r+c_2 (T(r,\hat c)/2)^\alpha)$. For large enough $\hat c$, it holds that \[
        r+c_2(T(r,\hat c)/2)^\alpha+\hat c(T(r,\hat c)/2)^\alpha = r+ \frac{c_2+\hat c}{2^\alpha}T(r,\hat c)^\alpha < R\,,
    \]
    which implies that $\abs{t-u}\le T(r+c_2(T(r,\hat c)/2)^\alpha,\hat c)$. It follows that $\mu^{k-1}_{u,t} (\mu^{k-1}_{s,u} y)$ is well-defined by induction hypothesis, and \eqref{ineq:mu_k1} gives us
    \begin{equation*}
        \abs{(I)} \le c_1\abs{t-s}^{1+\epsilon}2^{-1-\epsilon}\,.
    \end{equation*}
    Furthermore, $\mu_{s,u}^{k-1} y,\mu_{s,u}y\in B_x(r+c_2(T/2)^\alpha)\subset K$, so we can use the Lipschitz continuity of $\mu_{u,t}$ as well as $\abs{t-u}\le T(K)$ to show
    \begin{equation*}
        \abs{(II)} \le L(K,u,t)c_1\abs{t-s}^{1+\epsilon}2^{-1-\epsilon}\,,
    \end{equation*}
    where we used $\abs{\mu^{k-1} y_{s,u} -\mu_{s,u} y}\le c_1\abs{t-s}^{1+\epsilon}2^{-1-\epsilon}$ by \eqref{ineq:mu_k1}. Finally, the third term is simply bounded by the almost-flow property by
    \begin{equation*}
        \abs{(III)}\le C(K)\abs{t-s}^{1+\epsilon}\,.
    \end{equation*}
    Adding all three terms together, we get that
    \begin{equation*}
        \abs{\mu^k_{s,t} y -\mu_{s,t} y} \le (c_12^{-1-\epsilon} +L(K,u,t)c_1 2^{-1-\epsilon}+ C(K))\abs{t-s}^{1+\epsilon}\,.
    \end{equation*}
    As long as $\abs{t-s}$ is small enough, such that $L(K,u,t)$ is close enough to $1$, and for $c_1$ big enough (depending on $C(K)$), we get that $(c_12^{-1-\epsilon} +L(K,u,t)c_1 2^{-1-\epsilon} +C(K))\le c_1$, which shows that \eqref{ineq:mu_k1} holds for $k$. \eqref{ineq:mu_k3} follows analogously if one replaces the almost-flow property with the joint Lipschitz-almost-flow property and uses the Lipschitz property from Corollary \ref{cor:flowTechnical2} for $\mu^{k-1}_{s,u}$. \eqref{ineq:mu_k2} follows directly from 
    \begin{equation*}
        \abs{\mu_{s,t}^k y- y}\le \abs{\mu_{s,t}^k y-\mu_{s,t} y} + \abs{\mu_{s,t} y- y}\le c_2\abs{t-s}^\alpha 
    \end{equation*}
    for $c_2\ge c_1+B(K)$.
\end{proof}

\noindent It should be noted that our $T(r, \hat c)$ only really depends on the distance between $B_x(r)$ and $K^c$. So we can get \eqref{ineq:mu_k1}-\eqref{ineq:mu_k3} for all of $y\in K$ as well as all $\abs{t-s}\le \hat T(K)$ for a uniformly chosen $\hat T(K)$, by choosing a $\delta>0$ and a $\delta$-fattening of $K$, that is, a compact set $\tilde K\supset \{x\in U~\vert~ dist(x,K)\le\delta\}$, such that $\tilde K\subset U$ and applying the above lemma to $\tilde K$.

A simple corollary of Lemma \ref{lem:flowTechnical1} is that $\mu^k_{s,t}$ is Lipschitz continuous with a Lipschitz constant independent of $k$:

\begin{cor}\label{cor:flowTechnical2}
    For all $k\ge 1$ and $y, z\in B_x(r),\abs{t-s}\le T(r,\hat c)$ as above, we have that $\mu_{s,t}^k$ is Lipschitz continuous (on $B_x(r)$) with
    \begin{equation*}
        \abs{\mu_{s,t}^k(z-y)}\le (L(K,s,t)+c_3\abs{t-s}^{1+\epsilon})\abs{z-y}\,.
    \end{equation*}
\end{cor}

\begin{proof}
    This follows directly by using $\mu^k_{s,t} = (\mu^k_{s,t}-\mu_{s,t}) +\mu_{s,t}$.
\end{proof}
\noindent Let $\Delta_n := \{kT/2^N ~\vert~ k=0,\dots,2^N\}$ (where we identify $\{s= t_1\le\dots\le t_N=t\}$ with $\{[t_i,t_{i+1}]~\vert~s=t_1\le\dots\le t_N=t, i=1,\dots, N-1\}$) be the n-th dyadic partition of $[0,T]$ and let $\Delta= \bigcup_n \Delta_n$ be the set of dyadic numbers. For $s,t\in\Delta_n$, let 
\[
\mu^{\Delta_n}_{s,t} = \bigcirc_{\substack{[u,v]\in P_n([0,T]) \\ ~[u,v]\subset[s,t]}} ~\mu_{u,v}
\]
be the sewn-together flow over $P_n([0,T])$ restricted to $[s,t]$. We claim that the statements of Lemma \ref{lem:flowTechnical1} and Corollary \ref{cor:flowTechnical2} still hold for $\mu_{s,t}^{\Delta_n}$:
\begin{lem}\label{lem:flowTechnical3}
    In the above setting, we have that $\mu^{\Delta_n}_{s,t} y$ is well-defined for all $n\ge 0$, $y\in B_x(r)$ and $\abs{t-s}\le T(r,\tilde c)$. For these $s,t,y$ as well as $z\in B_x(r)$, it further holds that
    \begin{align}
        \abs{\mu^{\Delta_n}_{s,t} y-\mu_{s,t} y}&\le \tilde c_1\abs{t-s}^{1+\epsilon} \label{ineq:muDeltaN1}\\
        \abs{\mu^{\Delta_n}_{s,t} y-y} &\le \tilde c_2\abs{t-s}^\alpha \label{ineq:muDeltaN2}\\
        \abs{\mu^{\Delta_n}_{s,t}(z-y)} &\le (L(K,s,t)+ \tilde c_3\abs{t-s}^{1+\epsilon})\abs{x-y}\,,\label{ineq:muDeltaN3}
    \end{align}
    for a new set of constants $\tilde c_1,\tilde c_2,\tilde c_3,\tilde c>0$.
\end{lem}

\begin{proof}
    It holds that either $\mu_{s,t}^{\Delta_n}$ is a dyadic partition of $[s,t]$, or we can write it as $\mu^{\Delta_n}_{s,t} = \bigcirc_{k = 1}^K \tilde\mu_{u_k,v_k}$, where $\tilde\mu_{u_k,v_k}$ is a dyadic partition of $[u_k,v_k]$ and $\abs{u_k-v_k}\le \abs{t-s}2^{-(K-k)}$. If $\mu_{s,t}^{\Delta_n}$ is already the sewn-together flow over a dyadic partition, there is nothing to prove, so we assume it is not. It then holds that
    \begin{align*}
        \bigcirc_{k = 1}^K \tilde \mu_{u_k,v_k}y -\mu_{s,t}y = \underbrace{(\tilde\mu_{u_K,t}-\mu_{u_K,t})(\bigcirc_{k=1}^{K-1}\tilde \mu_{u_k,v_k} y)}_{(I)} &+ \underbrace{\mu_{u_K,t}(\bigcirc_{k=1}^{K-1}\tilde \mu_{u_k,v_k}-\mu_{s,v_{K-1}})}_{(II)}y \\
        &+ \underbrace{(\mu_{u_K,t}\circ\mu_{s,v_{K-1}} - \mu_{s,t})}_{(III)}\,.
    \end{align*}   
    We use induction over $K$. By induction hypothesis and \eqref{ineq:muDeltaN2}, we have that $\bigcirc_{k=1}^{K-1}\tilde \mu_{u_k,v_k} y \in B_x\left(r+\tilde c_2 \left(\frac{T(r,\tilde c)}2\right)^\alpha\right)$. If $\tilde c$ is large enough, such that
    \[
        r+ \hat c T(r,\tilde c)^\alpha + \tilde c_2\left(\frac{T(r,\tilde c)}{2}\right)^\alpha < R\,.
    \]
    Lemma \ref{lem:flowTechnical1} (recall that $\tilde \mu_{u_K,t}$ is dyadic) gives us that $\bigcirc_{k = 1}^K \tilde \mu_{u_k,v_k}y$ is well-defined, and that
    \[
        \abs{(I)} \le c_1 \abs{t-s}^{1+\epsilon}\,.
    \]
    For $(II)$, we note that one can use the induction hypothesis and the Lipschitz continuity of $\mu$ to get
    \begin{equation*}
        \abs{(II)}\le L(K,u_K,t)\tilde c_1\left(\frac{\abs{t-s}}{2}\right)^{1+\epsilon}\,.
    \end{equation*}
    We further get
    \[
        \abs{(III)}\le C(K)\abs{t-s}^{1+\epsilon}\,,
    \]
    as before. Gathering all three of the inequalities, we get
    \[
    \abs{\bigcirc_{k=1}^K\tilde\mu_{u_k,v_k} y-\mu_{s,t} y}\le \left(c_1 + L(K,s,t)\tilde c_12^{-1-\epsilon} + C(K)\right)\abs{t-s}^{1+\epsilon}
    \]
    For large enough $\tilde c_1$, it follows that $\left(c_1 + L(K,s,t)\tilde c_12^{-1-\epsilon} + C(K)\right)\le \tilde c_1$, showing \eqref{ineq:muDeltaN1}. \eqref{ineq:muDeltaN2} and \eqref{ineq:muDeltaN3} follow as before.
\end{proof}

\noindent With these technical lemmas, we can now show that local almost-flows generate flows:

\begin{lem}[sewing lemma for local almost-flows]\label{lem:sewing}
    Let $\mu_{s,t}$ be a local almost-flow. Then there exists an admissible domain $\hat O\subset O$ and a unique flow $\eta$ on $\hat O$, such that for all $x\in K$, $\abs{t-s}\le \hat T(K)$, we have
    \begin{equation}\label{ineq:uniqueness1}
        \abs{\mu_{s,t} x -\eta_{s,t} x}\le \hat C(K)\abs{t-s}^{1+\epsilon}\,,
    \end{equation}
    as well as
    \begin{align}
        \abs{\eta_{s,t} x- x}&\le \hat B(K) \abs{t-s}^\alpha \label{ineq:uniqueness2}\\
        \abs{\eta_{s,t}(x-y)}&\le \hat L(K,s,t)\abs{x-y}\label{ineq:uniqueness3}\,.
    \end{align}
\end{lem}

\begin{proof}
    {\bf Existence:} Let $K$ be some compact set with an inner point $x$ and let $y\in B_x(r)$. For $\abs{t-s}\le T(r,\tilde c)$, we claim that $\mu^{\Delta n}_{s,t} y$ is a Cauchy sequence in $n$ as long as $s,t\in \Delta_m$ for some $m\ge 0$. Once we have shown this, we will set $\eta_{s,t} y = \lim_{n\to\infty} \mu^n_{s,t} y$.
    
    It holds that $\mu^{\Delta_n}_{s,t} = \bigcirc_{k=0}^K \mu_{u_k^n,v_k^n}$. Set $f_{u_k^n,v_k^n} := \mu_{u_{2k+1}^{n+1},v_{2k+1}^{n+1}}\mu_{u_{2k}^{n+1},v_{2k}^{n+1}}$. It follows that $\mu^{\Delta_{n+1}}_{s,t} = \bigcirc_{k=0}^K f_{u_k^n,v_k^n}$. We write
    \begin{equation*}
        \mu^{\Delta_{n+1}}_{s,t}-\mu^{\Delta_n}_{s,t} = \sum_{k=0}^K \bigcirc_{l=k+1}^{K} f_{u_l^n,v_l^n} \circ(f_{u_k^n,v_k^n}-\mu_{u_k^n,v_k^n}) \circ\bigcirc_{l=0}^{k-1} \mu_{u_l^n,v_l^n}\,.
    \end{equation*}
    We chose $m$ in such a way, that $2^{-m-1} T < \abs{t-s}\le 2^{-m}T$. It follows that we have $\le 2^{n-m}$ summands in the above sum. By enlargening $\tilde c$ (or shrinking $T(r,\tilde c)$) a final time, we conclude that the above term is well-defined and the almost-flow property together with \eqref{ineq:muDeltaN3} gives us
    \begin{align*}
        \abs{\mu^{\Delta_{n+1}}_{s,t}y -\mu^{\Delta_n}_{s,t} y} &\le 2^{n-m}(\tilde L(K,s,t)C(K) 2^{-n(1+\epsilon)}\abs{T}^{1+\epsilon})\\
        &\le 2^{-(n-m)\epsilon} 2^{-m(1+\epsilon)}\abs{T}^{1+\epsilon}\tilde L(K,s,t)C(K)\,,
    \end{align*}
    where $\tilde L(K,s,t) = (L(K,s,t) + \tilde c_3\abs{t-s}^{1+\epsilon})$ is the Lipschitz constant from \eqref{ineq:muDeltaN3}. It follows that the sequence is Cauchy and thus convergent in $\RR^d$. Further, we immediately get from $\eta-\mu = \sum_{n=m}^{\infty} (\mu^{\Delta_{n+1}}-\mu^{\Delta_n})+ (\mu^{\Delta_m}-\mu)$, that
    \begin{align}
        \abs{\eta_{s,t} x - x}\le \hat B(K) \abs{t-s}^\alpha \label{ineq:eta1}\\
        \abs{\eta_{s,t} (x-y)} \le \hat L(s,t,K)\abs{x-y} \label{ineq:eta2}\\
        \abs{\eta_{s,t}x -\mu_{s,t} x} \le \hat C(K) \abs{t-s}^{1+\epsilon}\label{ineq:eta3}
    \end{align}
    for all $x,y \in K\subset U$ compact, $\abs{t-s}\le \hat T(K)$ and $\hat C(K), \hat L(s,t,K), \hat B(K)$ a new set of constants. Let us now check that $\eta$ is indeed a flow on $\Delta$. For all $s,u,t\in \Delta_n$, $s\le u\le t$, it holds that $\mu^{\Delta_n}_{s,t} =\mu^{\Delta_n}_{u,t}\circ\mu^{\Delta_n}_{s,u}$ by construction. Thus, we can use the continuity of $\eta$ to show that for all $s,u,t\in\Delta$, $s\le u\le t$ with $\abs{t-s}\le\hat T(K)$:
    \begin{align*}
        \abs{\eta_{s,t}-\eta_{u,t}\eta_{s,u}} &\le \abs{\eta_{s,t} -\mu^{\Delta_n}_{s,t}} + \abs{\eta_{u,t}\eta_{s,u}-\mu^{\Delta_n}_{u,t}\mu^{\Delta_n}_{s,u}} + \abs{\mu_{s,t}^{\Delta_n}-\mu^{\Delta_n}_{u,t}\mu^{\Delta_n}_{s,u}}\\
        &\to 0
    \end{align*}
    as $n\to \infty$, where we used that
    \begin{equation*}
        \abs{\eta_{u,t}\eta_{s,u}-\mu^n_{u,t}\mu^n_{s,u}} \le \abs{\eta_{u,t}(\eta_{s,u}-\mu^n_{s,u})} + \abs{(\eta_{u,t}-\mu^n_{u,t})\mu^n_{s,u}}\to 0
    \end{equation*}
    as $n\to\infty$. It follows that $\eta$ is a flow on $\Delta$, and by continuity we can extend $\eta$ uniquely to a flow on $\Delta_T$.\\[\baselineskip]
    {\bf Uniqueness:} Assume $\eta$ is a flow on $\hat O$ fulfilling \eqref{ineq:uniqueness1}-\eqref{ineq:uniqueness3}. Then we can calculate for $s,t\in\Delta$:
    \begin{align*}
        \abs{\eta_{s,t} - \mu_{s,t}^{\Delta_n}} &\le \sum_{k=0}^K \abs{\mu^{\Delta_n}_{u_{k+1},t}\circ(\eta_{u_k,v_k}-\mu_{u_k,v_k}) \circ\eta_{s,v_{k-1}}} \\
        &\le 2^{-n\epsilon-m}\tilde L(s,t,K)\hat C(K) T^{1+\epsilon}\to 0
    \end{align*}
    as $n\to\infty$. This finishes the proof.
\end{proof}

\begin{rem}\label{rem:UNiquenessAddition}
    It should be noted, that we never used the Lipschitz property \eqref{ineq:uniqueness3} in the proof of uniqueness. \eqref{ineq:uniqueness1} and \eqref{ineq:uniqueness2} are enough to uniquely characterize $\eta$.
\end{rem}

\subsection{Almost-flows on manifolds}\label{subsec:ManifoldFlows}

We want to use our notation of local almost-flows to generate flows on a manifold $M$. To this end, we say that an almost-flow on $M$ is a map $\mu$ mapping some admissible domain in $\Delta_T\times M$ to $M$, such that for each coordinate chart $\phi:M\supset V\to U\subset\RR^d$, $\mu^\phi_{s,t} := \phi\circ\mu_{s,t}\circ\phi^{-1}$ is a local almost-flow on some admissible domain over $U$. The following proposition shows that this indeed generates a flow:

\begin{prop}\label{prop:sewingOnM}
    Let $\mu_{s,t} x$ defined on an admissible domain $\diag_T\times M\subset O \subset \Delta_T\times M$. Assume that for each coordinate chart $\phi:M\supset V \to U\subset \RR^d$, $\mu_{s,t}^\phi$ is an almost-flow on some open set $\diag_T\times U\subset O^\phi\subset \Delta_T\times U$. Let $\eta_{s,t}^\phi$ be the flow associated with $\mu^\phi_{s,t}$. Then $\eta_{s,t} := \phi^{-1}\eta_{s,t}$ is coordinate independent. (Up to the maximal time $\abs{t-s}\le T^\phi$, which might depend on the coordinate chart)
\end{prop}

\begin{proof}
    This follows immediately from the fact that $\mu_{s,t}^n x\in M$ is coordinate independent, as long as it is well-defined. Given two coordinate charts $\phi,\tilde\phi$, and a point $x\in V\cap\tilde V$, we can always choose a small enough $T$ such that Lemma \ref{lem:sewing} can be applied both to $(\mu^{\phi,\Delta_k}_{s,t}x)_{k\ge 0}$ and $(\mu^{\tilde\phi,\Delta_k}_{s,t}x)_{k\ge 0}$ for all $\abs{s-t}\le T$. By construction, $\mu^{\phi,\Delta_k}_{s,t}x = \mu^{\tilde\phi,\Delta_k}_{s,t}x$ for all $k\ge 0$. Since this sequence converges in the topology of $M$ we get auniqe limit and conclude that $\eta^\phi_{s,t} x = \eta^{\tilde\phi}_{s,t} x$. Thus, $\eta$ is a coordinate independent flow.
\end{proof}

\section{Preliminaries}\label{sec:prelim}

\subsection{Hopf and Lie algebras}

We use this section to recall the notion of Hopf algebras and the most commonly used Hopf algebras in rough path theory. We mainly follow the setting of \cite{curry20} and \cite{ferrucci22}, but also recommend \cite{GUBINELLI2010693},\cite{HK15} as well as classical references like \cite{sweedler69}, \cite{manchon2006hopf} or \cite{HopfAlgebrasAbe}.

In a nutshell, Rough paths live in the dual of a graded, connected, and commutative Hopf algebra, making the space the rough paths live in a graded, connected, cocommutative Hopf algebra. We start by recalling the notion of a graded vector space:

\begin{defn}
    A vector space $V$ is called a graded vector space, if there are finite-dimensional vector spaces $V^{(0)},V^{(1)},\dots$, such that $V = \bigoplus_{i=0}^\infty V^{(i)}$ is the set of finite linear combinations of vectors in $V^{(i)}$. For all $v\in V^{(i)}$, we write $\abs v = i$ and call $i$ the grade of $v$. 
\end{defn}

\noindent Note that we especially assume that all $V^{(i)}$ are finite-dimensional. This will immediately give us, that all truncated spaces $\bigoplus_{i\le N} V^{(i)}$ are finite-dimensional.

We can now define a Hopf algebra. The following definition comes from \cite{ferrucci22}:

\begin{defn}
    A graded connected Hopf algebra over $\RR$ is a graded vector space $\cH = \bigoplus_{i=1}^\infty \cH^{(i)}$ equipped with linear maps $m:\cH\otimes\cH\to\cH, m(x,y) =: x\cdot y$ (product), $\Delta:\cH\to\cH\otimes\cH$ (coproduct), $\one:\RR\to\cH$ (unit), $\epsilon:\cH\to\RR$ (counit) and $S:\cH\to\cH$ (antipode), such that the following holds for all $x,y,z\in\cH$:

    \begin{itemize}
        \item {\bf Connectedness:} $\cH^{(0)} = \RR$.
        \item {\bf Unit:} $\one:\RR\to\cH$ has the properties $\one(\lambda)\cdot x = \lambda x = x\cdot\one(\lambda)$, as well as $\Delta\circ\one = \one\otimes\one$. With a slight abuse of notation, we also write $\one = \one(1)\in \cH$. 

        \item {\bf Counit:} $m\circ(\epsilon\otimes\Id)\circ\Delta = \Id = m\circ(\Id\otimes\epsilon)\circ\Delta$ and $\epsilon(x\cdot y) = \epsilon(x)\epsilon(y)$.

        \item {\bf Associativity:} $(x\cdot y)\cdot z = x\cdot(y\cdot z)$.
        \item {\bf Coassociativity:} $(\Delta\otimes\Id)\Delta x = (\Id\otimes\Delta)\Delta x$.

        \item {\bf Compatibility:} $\Delta(x\cdot y) = \Delta x\cdot \Delta y$ and $\epsilon\circ\one = \Id_{\RR}$.

        \item   {\bf Antipode:} $m \circ(S\otimes \Id) \Delta = m \circ(\Id\otimes S) \Delta = \one\circ\epsilon$.

        \item {\bf Grading:} $\cH^{(i)}\cdot\cH^{(j)}\subset\cH^{(i+j)}$ and $\Delta\cH^{(i)}\subset\bigoplus_{(k+l = i)}\cH^{(k)}\otimes\cH^{(l)}$\,. The antipode fulfills $S(\cH^{(i)}) \subset\cH^{(i)}$.
    \end{itemize}
\end{defn}

\noindent Let us recall some basic results about Hopf algebras, which can all be found in \cite{sweedler69}: It always holds that $S$ is an anti-homomorphism, that is that for all $u,v\in\cH$, $S(u\cdot v) = S(v)\cdot S(u)$. If $\cH$ is further graded, we know that $\one\in\cH^{(0)} = \RR$ and can thus identify it with $1$. Furthermore, the grading implies that $\epsilon(v) = 0$ for all $v\in\cH^{(i)}$ for $i\ge 1$. Thus, the compatibility condition implies that $\epsilon$ has precisely the following form:
\begin{equation*}
    \epsilon(\lambda \one + R) = \lambda\,,   
\end{equation*}
where $R\in \bigoplus_{i\ge 1} \cH^{(i)}$.

The next concept we need is commutativity and its dual counterpart, cocommutativity. To this end, we define the switch operator $\tau:\cH\otimes\cH\to\cH\otimes\cH$ to be $v\otimes u\mapsto u\otimes v$ for all $u,v\in\cH$. We then call $\cH$:

\begin{itemize}
    \item {\bf Commutative}, if $m \circ\tau = m$.
    \item {\bf Cocommutative}, if $\tau\circ\Delta = \Delta$.
\end{itemize}

\noindent The level $n$ trunctation of $\cH$ is given by

\begin{equation*}
    \cH^n = \cH/\bigoplus_{k = n+1}^\infty\cH^{(k)}\,,
\end{equation*}

\noindent where we use that $\bigoplus_{k = n+1}^\infty\cH^{(k)}$ is an two-sided ideal of $\cH$. We identify the trunctation with $\cH^n = \bigoplus_{k = 0}^n\cH^{(k)}$.

\subsubsection{Graded dual of a graded Hopf algebra}

Let $\cH$ be a graded Hopf algebra. For each $i\in\NN$, we choose a basis $\{b^i_1,\dots, b^i_{k_i}\}$ of $\cH^{(i)}$. Since $\{b^i_k~\vert~ i\in\NN,k\le k_i\}$ is a countable Hamel-basis of $\cH$, its dual $\cH^*$ is given by all formal series
\[
\sum_{i\in\NN, k\le k_i} \beta(i,k) b_{k}^{i*}\,,
\]
where $b^{i*}_k (b^j_l) = \delta_{(i,k),(j,l)}$. We denote the dual action of $\cH^*$ on $\cH$ by $\scalar{\cdot,\cdot}: \cH^*\times\cH\to \RR$. Since $\cH$ is infinite-dimensional, $\cH^*$ does in general not have a Hopf algebra structure. However, its \emph{graded dual} (\cite{sweedler69}, page 231) of $\cH$, given by $(\cH^g, \star, \delta,\epsilon^*,\one^*,S^*)$:
\[
\cH^g := \bigoplus_{k=0}^\infty \cH^{(i)*} \subset\cH^*
\]
is a graded Hopf algebra with Hamel-basis $\{b^{i*}_k~\vert~ i\in\NN,k\le k_i\}$ and dual operations
\begin{align*}
    \scalar{x\star y, z} &:= \scalar{x\otimes y,\Delta z}\\
    \scalar{\one^*,z} &:= \epsilon(z)\\
    \scalar{\delta x, z\otimes y} &:= \scalar{x,z\cdot y}\\
    \epsilon^*(x) &:= \scalar{x,\one}\\
    \scalar{S^*(x),y} &:=\scalar{x,S(y)}\,.
\end{align*}
If $\cH$ is connected, it immediately follows that $\cH^g$ is connected as well. $\cH^g$ is cocommutative, if $\cH$ is commutative and $\cH^*$ is commutative, if $\cH$ is cocommutative.

\subsubsection{Group-like and primitive elements}

Rough paths live in a Lie group inside $\cH^*$, called the group-like elements or characters. These are given by all non-zero $g\in \cH^*$, such that $\scalar{g,\cdot}: \cH\to\RR$ is a homomorphism. If $g\in\cH^g$, this would imply that $\delta g = g\otimes g$, as $\delta$ is dual to $m$. However, in practice, this is only fulfilled by infinite series $g= \sum_{i,k}\beta(i,k) b_k^{i*}$, whereas $\cH^g$ only contains finite sums. Thus, for the three Hopf algebras introduced in Section \ref{sec:HopfALgebrasInRP}, the only group-like element would be given by $\one$.

We get around this issue by extending $\delta$ to a map
\[
\delta:\overline{\bigoplus_{i = 0}^\infty \cH^{(i)*}} \to \overline{\bigoplus_{i,j=0}^\infty \cH^{(i)*}\otimes\cH^{(j)*}}
\]
by setting
\[
\delta\left(\sum_{i,k} \beta(i,k)b_k^{i*}\right) := \sum_{i,k}\beta(i,k)\delta b_k^{i*}\,.
\]
Here, $\overline{\bigoplus_{i=0}^\infty \cH^{(i)*}}$ denotes the space of infinite series $\sum_{i,k} \beta(i,k)b_k^{i*}$. Thus, we say that the set of group-like elements is defined by
\[
    \cG := \{x\in \cH^* ~\vert~ \delta x= x\otimes x, x\neq 0\}\,.
\]
While $\cG$ is not a subset of the Hopf algebra $\cH^g$, we can always truncate it to an element in $\cH^n$ for some $n\ge 0$ to regain the Hopf algebra structure. This leads to the following definition:
\begin{defn}
    The set of level $n$ group-like elements is given by
    \begin{equation*}
        \cG^n := \{x\in\cH^{n*}~\vert~ \delta x = x\otimes x, x\neq 0\}\,,
    \end{equation*}
    where the identity $\delta x = x\otimes x$ holds in the truncated tensor product $(\cH^g\otimes\cH^g) /(\bigoplus_{i+j\ge n+1} \cH^{(i)*}\otimes\cH^{(j)*})$.
\end{defn}
\noindent $\cG^n$ is a group with the product $\star$ inherited from $\cH^{n*}$ and the inverse given by $x^{-1} = S(x)$. As it turns out, it is a Lie group with its Lie algebra being given by the set of \emph{primitive elements}, which is defined as follows:

\begin{defn}
    The set of primitive elements is given by
    \[
        \cP := \{x\in\cH^*~\vert~ \delta x = x\otimes 1+1\otimes x\}\,.
    \]
    The level $n$ primitives are given by
    \[
        \cP^n := \{x\in\cH^{n*} ~\vert~ \delta x = x\otimes 1+1\otimes x\}\,,
    \]
    where $\delta x = x\otimes 1+1\otimes x$ is again considered in $(\cH^g\otimes\cH^g) /(\bigoplus_{i+j\ge n+1} \cH^{(i)*}\otimes\cH^{(j)*})$.
\end{defn}

\noindent One can easily calculate that $\cP$ as well as $\cP^n$ are Lie algebras if we equip them with the commutator $[x,y] = x\star y-y\star x$ as a Lie bracket, see Section \ref{subsec:LieAlgebras} for details.

We will treat $\cG^n,\cP^n$ as subsets of $\bigoplus_{i=0}^n \cH^{(i)*}$, where they have a unique representative. For $g\in\cG^n, p\in\cP^n$ and an $m\le n$, it holds that $\pi^m g\in\cG^m$ and $\pi^mp \in\cP^m$, where $\pi^m$ denotes the truncation map $\cH^{n*}\to\cH^{m*}$, $\sum_{i=0}^n h_i \mapsto \sum_{i=0}^m h_i$. For the primitive elements, the inverse is also true: Given an $p\in \cH^{m*}$, $p = \sum_{i=0}^m p_i$, we have that $\sum_{i=0}^m p_i \in \cP^n$ is also a level $n$-primitive. This is in general not true for the group-like elements.

In fact, for primitive elements, an even stronger statement holds:

\begin{lem}\label{lem:pkIsPrimitive}
    Let $p\in\cP$ be a primitive element. For any $k\ge 1$, denote by $p^k$ the projection of $p$ onto $\cH^{(k)*}$. Then $p^k$ is a primitive element.
\end{lem}

\begin{proof}
    Fix a $k\ge 1$ and let $R^k$ be defined by
    \[
        \Delta p^k = \one\otimes p^k + p^k\otimes \one + R^k\,.
    \]
    We claim that $R^k = 0$. Since $\cH^g$ is grade, we see that $R^k\in \bigoplus_{i+j = k} \cH^{(i)*}\otimes\cH^{(j)*}$. If we define $R^{\tilde k}$ analogously to $R^k$ for each $\tilde k\neq k$,
    \[
        \Delta p = \sum_{m\ge 1}\Delta p^m = 1\otimes p + p\otimes 1+ R^k + \sum_{\tilde k\neq k}R^{\tilde k}
    \]
    gives us that $R^k = - \sum_{\tilde k\neq k}R^{\tilde k} \in \bigoplus_{i+j\neq k} \cH^{(i)*} \otimes\cH^{(j)*}$. Since \\ $\left(\bigoplus_{i+j\neq k} \cH^{(i)*} \otimes\cH^{(j)*}\right) \cap \left(\bigoplus_{i+j= k} \cH^{(i)*} \otimes\cH^{(j)*}\right) = \{0\}$, we get that $R^k = 0$, showing the claim.
\end{proof}

\subsubsection{Exponential and Logarithmic map}

Since $\cG^n,\cP^n$ is a Lie group and its associated Lie algebra, there exists at least a local diffeomorphism $\exp:\cP^n\to\cG^n$ around the $0$-element in $\cP^n$. As it turns out, this is a global map simply given by the usual formula for the exponential map. In this subsection, we want to construct this map as well as its inverse $\log$.

A standard result gives us that for all $x\in\cG^n, y\in\cP^n$, $\epsilon(x) = 1$ and $\epsilon(y) = 0$. Recalling the form of $\epsilon$, we conclude that $\cG^n\subset\cH^{n*}_1$ and $\cP^n\subset\cH^{n*}_0$, where
\[
\cH^{n*}_a := \left\{a\one + \sum_{i=1}^n h_i ~\middle\vert~ h_i\in\cH^{(i*)}\right\}\,.
\]
Note that thanks to the truncation, $\cH^{(n*)}_0$ is nilpotent under the product $\star$, implying that all power series $\sum_{n\ge 0} \beta(n) h^{\star k}$ converge for $h\in\cH^{n*}_0$. We conclude that the following maps are well-defined:
\begin{align*}
    \exp_n: \cH^{n*}_0 &\to \cH^{n*}_1 \\
    h&\mapsto \sum_{k\ge 0} \frac{h^{\star k}}{k!} \\
    \log_n: \cH^{n*}_1&\to\cH_0^{n*} \\
    \one+h&\mapsto \sum_{k\ge 1} (-1)^{k+1}\frac{h^{\star k}}{k}\,.
\end{align*}
One easily checks that $\exp_n,\log_n$ are inverse to each other, turning them into diffeomorphisms with $\exp(0) = \one$. Further, the following result gives us that the Lie algebra $\cP^n$ is indeed associated with the Lie group $\cG^n$:

\begin{prop}
    It holds that $\exp_n\vert_{\cP^n}: \cP^n\to \cG^n$ is a diffeomorphism with inverse $\log_n\vert_{\cG^n}$.
\end{prop}

\begin{proof}
    The proof is the same as in the tensor algebra case, given in \cite{freeLieAlgebras}, Theorem 3.2.
\end{proof}

\subsection{Lie algebras and universal enveloping algebra}\label{subsec:LieAlgebras}

The goal of this section is to recall the correspondence between Lie algebras and Hopf alebras For more details on this topic, we refer to \cite{surveyLieAlgebras} for a short introduction as well as \cite{freeLieAlgebras} and \cite{universalAlgebras} for a deeper introduction. Let us recall the definition of a Lie algebra:

\begin{defn}
    A Lie algebra is a vector space L equipped with a bilinear \emph{Lie bracket} $[\cdot,\cdot]: L\times L\to L$, such that
    \begin{itemize}
        \item $[\cdot,\cdot]$ is anti-symmetric.
        \item The Jacobi-identity holds: For all $x,y,z\in L$, we ave that
        \[
        [x,[y,z]] + [y,[z,x]] + [z,[x,y]] = 0\,.
        \]
    \end{itemize}
    We say $(L,[\cdot,\cdot])$ is graded if one can decompose it into finite-dimensional vector spaces
    \begin{equation*}
        L = \bigoplus_{n\ge 1} L^{(n)}\,,
    \end{equation*}
    such that for all $x\in L^{(k)}, y\in L^{(l)}$, we have that $[x,y]\in L^{(l+k)}$.
\end{defn}

\noindent Note that, unlike the Hopf-algebra case, we start with $n=1$ and will not have a space $L^{(0)}$! We say that $L$ is nilpotent, if there is an $N\ge 1$ such that for all $k\ge N$, $L^{(k)} = \emptyset$. For any graded $L$, its truncation $L^n = L/\bigoplus_{i\ge n+1} L^{(i)} = \bigoplus_{i=0}^n L^{(i)}$ is automatically nilpotent.

\begin{example}
    Let $(A,\circ)$ be an associate algebra with commutator $[x,y] = x\circ y- y\circ x$. Then $(A,[\cdot,\cdot])$ is a Lie algebra.
\end{example}

\noindent Since any Hopf algebra is an associative algebra, we can always equip it with its commutator to turn it into a Lie algebra. However, we are more interested in the Lie subalgebra given by the primitive elements $\cP$: For two primitive elements $p,q\in \cP$, it holds that
\begin{equation*}
    \Delta[p,q] = \Delta p \circ \Delta q - \Delta q\circ\Delta p = [p,q]\otimes 1+ 1\otimes [p,q]\,.
\end{equation*}
Thus, $\cP$ is closed under $[\cdot,\cdot]$ and thus a Lie subalgebra. It especially follows that the truncation $\cP^n$ is a nilpotent Lie algebra.

If one starts with a Hopf algebra $\cH$, it is in general not possible to recover $\cH$ just from the Lie algebra $\cP$ of primitive elements. However, one can recover a smaller Hopf algebra $U(\cP)$, called the universal enveloping algebra of $\cP$, such that $[x,y] = x\star y- y\star x$ holds in $U(\cP)$ and $\cP\subset U(\cP)$ is the set of primitive elements in $U(\cP)$. It is constructed as follows: Let $T(\cP)$ be the tensor algebra over with grading inherited from $\cP$, i.e. $p_1\otimes p_2\otimes\dots\otimes p_n$ has grading $\sum_{k=1}^n \abs{p_k}$. When $T(\cP)$ is an associative algebra with the product $\otimes$. To make sure that $x\otimes y-y\otimes x$ is the same as the Lie bracket $[x,y]$, we set $\cI$ to be the two-sided ideal generated by $\tilde\cI := \{x\otimes y - y\otimes x-[x,y] \vert x,y \in \cP\}$. By quotioning out this ideal, we guarantee the desired identity.

\begin{defn}
    The space
    \begin{equation*}
        U(P) := T(P)/\cI
    \end{equation*}
    is called the universal enveloping algebra.
\end{defn}

\noindent $U(\cP)$ becomes a graded Hopf algebra by setting

\begin{itemize}
    \item the multiplication to be the natural product in $U(P)$.
    \item For all $x\in P$, $\Delta x = 1\otimes x + x\otimes 1$.
    \item $\one(\alpha) = \alpha \one$.
    \item $\epsilon(\one) = 1$ and $\epsilon (x) = 0$ for all $x$ of degree higher than $0$.
    \item $S(x) = -x$ for all $x\in P$.
\end{itemize}

\noindent Here, the coproduct, $\epsilon$ and $S$ will be extended to all of $U(P)$ by demanding, that $\Delta,\epsilon$ are algebra morphisms and $S$ is an anti-homomorphism. It is a classical result (e.g. \cite{surveyLieAlgebras}, \cite{universalAlgebras}), \cite{Hochschild1981} that this turns $U(\cP)$ into a graded Hopf algebra.

Furthermore, $U(\cP)$ inherits a grading from $\cP$ by setting $\abs{p}$ to be the grade of $p$ as an element of the graded Lie algebra $\cP$. Further, we require $\abs{x* y} =\abs x + \abs y$, which extends the grading to all of $U(\cP)$.

\begin{prop}\label{prop:UisShuffle}
    $U(P) = \cH_{Sh}/\cI$, where $\cH_{Sh}$ denotes the shuffle-Hopf algebra.
\end{prop}

\begin{proof}
    It is clear that $\cH_{Sh}/\cI$ has all the above properties, as long as it is well-defined. One just needs to prove that $\epsilon, S$ and $\Delta$ are well-defined in $\cH_{Sh}/\cI$. Note that for all $z\in\tilde\cI$, it holds that $S(z) = -z\in\tilde\cI$. Thus, it is well-defined as an anti-homomorphism $\cH_{Sh}/\cI\to \cH_{Sh}/\cI$. Furthermore, $\epsilon(\tilde\cI) = 0$. $\Delta z = 1\otimes z + z\otimes 1$ implying $\Delta\tilde\cI = 1\otimes\tilde\cI + \tilde\cI\otimes 1$. Thus, all operations are well-defined showing the claim.
\end{proof}

\noindent The following statement can be found in \cite{Hochschild1981} as Theorem 2.1:

\begin{thm}
    Let $\cP$ be a graded Lie-algebra with universal enveloping algebra $U(\cP)$. Then $\cP$ is the set of primitive elements in $U(\cP)$.
\end{thm}

\noindent While in general $\cH$ can be bigger than $U(\cP)$, it turns out that they are isomorphic to each other if $\cH$ is cocommutative. This is called the Theorem of Milnor-Moore, and gives a 1-1 correspondence between Lie algebras and Hopf algebras in the cocommutative setting:

\begin{thm}[Milnor-Moore, \cite{Milnor-Moore}]\label{MilnoreMoore}
    Let $\cH$ be a connected, graded, cocommutative Hopf algebra. Then $U(\cP)\cong \cH$.
\end{thm}

\noindent This is of special interest in rough path theory, since all of the classically considered Hopf algebras of rough paths are cocommutative, implying that rough paths actually live in universal enveloping algebras. In Section \ref{subsec:BiMapsFromLieMaps}, we will use this observation to construct elementary differentials from Lie algebra maps.

\subsection{Rough paths in general Hopf algebras}

Let us fix the notation $\Delta_T := \{(s,t)\in [0,T]^2\vert s\le t\}$ for the 2 dimensional simplex. A rough path over a general Hopf algebra $\cH$ is a two-parameter process $(\XX_{s,t})_{(s,t)\in\Delta_T}$ in its dual space $\cH^*$, which fulfills three properties, namely a Hölder continuity assumption, a Chen identity and that $\XX$ is a character on $\cH$. Recall that the last assumption translates to $\XX_{s,t}$ being a group-like element. We make a further assumption, namely that for all rough paths, $\cH$ is a commutative Hopf algebra. This especially implies that $\XX$ lives in the cocommutative Hopf algebra $\cH^g$, which can therefore be seen as a universal enveloping algebra by Theorem \ref{MilnoreMoore}. Note that all three Hopf algebras introduced in Section \ref{sec:HopfALgebrasInRP} fulfill this assumption.

Thus, our definition of a rough paths reads as follows:

\begin{defn}
    Let $\cH$ be a graded, connected, commutative Hopf algebra with graded dual $\cH^g$. Let $\alpha\in(0,1)$ and set $N$ such that $N\alpha \le 1$, $(N+1)\alpha > 1$. Then $\XX:\Delta_T\to G^{N}\subset\cH^{N*}$ is called an $\alpha$-rough path, if:

    \begin{itemize}
        \item Chen's identity:
        \begin{equation*}
            \XX_{s,u}\star\XX_{u,t} = \XX_{s,t}\,,
        \end{equation*}
        where $\star$ refers to the truncated product $G^N\times G^N \to G^N$.

        \item Hölder continuity: There exists a constant $C>0$, such that for each $\tau\in\cH^{(i)*}$ for some $0\le i\le N$:
        \begin{equation*}
            \abs{\scalar{\XX_{s,t},\tau}} \le C^{\abs\tau}\abs{t-s}^{\abs{\tau}\gamma}\,.
        \end{equation*}
        We call the infimum over all such constants $C$ the norm $\norm{\XX}$.
    \end{itemize}
\end{defn}

\begin{rem}
    Note that at this point, we made the choice to think of $\XX_{s,t}$ not as an infinite sum in $\cH^*$, but only as a finite sum in $\cH^{N*}$. Thus, this is the point where we choose to follow the sewing approach of \cite{bailleul15} instead of the algebraic approach of \cite{curry20}.
\end{rem}

\noindent Our choice of $\norm{X}$ allows for the following interaction with the multiplication: For any $k\ge 1$ and $\tau\in\cH^{(i)}$ for some $0\le i\le N$, one calculates that
\begin{equation}\label{ineq:normX^k}
    \abs{\scalar{\XX_{s,t}^{\star k},\tau}} = \abs{\sum_{\tau_1,\dots,\tau_k}\scalar{X_{s,t},\tau_1}\dots\scalar{X_{s,t},\tau_k}\scalar{\tau_1\star\dots\star\tau_k,\tau}} \le C\norm{X}^{\abs\tau}\abs{t-s}^{\alpha\abs\tau}\,,
\end{equation}
for some $C>0$ depending only on the $\cH$, and where we used that $\scalar{\tau_1\star\dots\star\tau_k,\tau} \neq 0$ only holds for $\abs{\tau_1}+\dots+\abs{\tau_k} = \abs{\tau}$. This immediately allows us to compare the norm of the logarithm of $\XX$ with $\norm{\XX}$:

\begin{lem}\label{lem:boundNormL}
    Let $\LL_{s,t} = \log(\XX_{s,t})\in \cP^N$ and define $\norm{\LL}$ the same way as $\norm{\XX}$. Then there are constants $c,C>0$ depending only on the space $\cH^N$, such that for all $\tau\in \cH^{(k)}$ for some $1\le k\le N$, we get that

    \begin{equation*}
        c\norm{\XX}\le\norm{\LL}\le C\norm{\XX}\,.
    \end{equation*}
\end{lem}

\begin{proof}
    Let $\tilde\XX_{s,t} = \XX_{s,t} - 1$. Then the formula
    \begin{equation*}
        \LL_{s,t} = \sum_{k=1}^N (-1)^{k+1} \frac{\tilde\XX_{s,t}^k}{k}
    \end{equation*}
    together with \eqref{ineq:normX^k} immediately gives us, that there is a $C>0$ such that
    \begin{equation*}
        \abs{\scalar{\LL_{s,t},\tau}}\le C^{\abs\tau}\norm{\XX}^{\abs\tau}\abs{t-s}^{\alpha\abs\tau}\,,
    \end{equation*}
    for all $(s,t)\in\Delta_T$ and $\tau\in\cH^{(i)}, 0\le i\le N$. This immediately gives us $\norm{\LL}\le C\norm{\XX}$. $\norm{\XX}\le \frac 1c\norm{\LL}$ follows analogously from $\XX_{s,t} = \exp_N(\LL_{s,t})$.
\end{proof}

\subsection{Hopf algebras relevant to Rough path theory}\label{sec:HopfALgebrasInRP}

In this section, we want to review the classical Hopf algebras used in Rough path theory:

\subsubsection{Shuffle algebra and tensor algebra}

The classical Hopf algebra considered in rough path theory (\cite{originalArticleRP}, \cite{FV10}, \cite{FH20}) is the tensor algebra acting on the shuffle algebra. An in-depth look at the algebra itself can be found in \cite{freeLieAlgebras}.

Let $V$ be a (finite-dimensional) vector space and consider the set of tensor polynomials
\begin{equation*}
    T(V) := \bigoplus_{k\ge 0} V^{\otimes k}\,.
\end{equation*}
We will denote the product $m:T(V)\otimes T(V)\to T(V)$ with $\cdot$, so that $\otimes$ is reserved for elements $x\otimes y \in T(V)\otimes T(V)$. Thus, $T(V)$ consists of words $e_{i_1}\dots e_{i_n}$ with $i_1,\dots,i_n\in\{1,\dots,d\}$ and $\{e_1,\dots,e_d\}$ being a basis of $V$, if it is finite-dimensional. $T(V)$ is naturally equipped with the product $\cdot$, which concatenates two words $u\cdot v = uv$. 

We can also equip it with the shuffle product $\shuffle$ as follows: Given two words $w = w_1\otimes\dots\otimes w_n$, $u = u_1\otimes\dots\otimes u_k$, $w\shuffle u$ is the sum over all words with the letters $w_1,\dots, w_n,u_1,\dots,u_k$, such that the original order of letters in $w$ and $u$ gets preserved. It is formally defined recursively by
\begin{align*}
    w\shuffle \one &= w\\
    w\shuffle u &= (w\shuffle (u_1\dots u_{k-1}))\otimes u_k + ((w_1\dots w_{n-1})\shuffle u)\otimes w_n
\end{align*}
where $\one$ is the empty word and $w= (w_1\dots w_n), u= (u_1\dots u_k)$. Note that $\shuffle$ is a commutative product. We denote the dual of the multiplication $\cdot$ with $\Delta$, which is just given by
\begin{equation*}
    \Delta w = \sum_{w_1\cdot w_2 = w} w_1\otimes w_2\,.  
\end{equation*}
On the other hand, we denote the dual of the shuffle product by the coproduct $\Delta_\shuffle$, also called the \emph{deshuffle}. We then set $\cH = (T(V),\shuffle,\Delta)$ to be the shuffle-algebra (where $\one$ is the empty word and $\epsilon(\one) = 1$, $\epsilon(w) = 0$ for all non-empty $w$), turning $\cH$ into a bialgebra.

We can introduce a grading on $\cH$, whenever there is a grading on $V$: Set $\abs{(w_1\dots w_n)} = \abs{w_1}+\dots+\abs{w_n}$ for all words $w$, suc that each letter $w_i \in V^{(j)}$ for some $j\in\NN$ and all $i=1,\dots,n$. If $V = V^{(1)}$, the grade of each word is simply its length. In this case, we say that the grading of $T(V)$ is homogenous. Otherwise, we call the grading inhomogenous. $\cH$ is a connected graded bialgebra and thus a Hopf algebra by \cite{sweedler69}.

Its graded dual is given by the Tensor algebra $\cH^g = (T(V),\cdot, \Delta_\shuffle)$ with the same $\one$ and $\epsilon$. In this case, we can directly identify the primitive elements as the single letter words $\cP = V\subset \cH$.

Rough paths in $(T(V),\cdot,\Delta_\shuffle)$ are called (weakly) geometric rough paths. Since we never use non-weakly geometric rough paths in this paper, we drop the word weakly and simply refer to rough paths in $T(V)$ as geometric rough paths.

\subsubsection{Hopf algebras of rooted trees}\label{sec:HopfAlgebrasTrees}

The classical Hopf algebras for branched rough path theory are considered over rooted trees \cite{foissy13}. In this context, we need to differentiate between \emph{planarly} and \emph{non-planarly} branched rough paths, which are constructed over ordered and unordered rooted trees, respectively.

To get started, we call an acyclic connected graph with finitely many vertices a tree. If there is a preferred vertex, we call that vertex the root and the tree rooted. Then drawing a rooted tree, we will always draw the root as the bottommost vertex. Given an alphabet $A$, any tree is called an $A$-decorated tree, if each vertex is equipped with a decoration $a\in A$. For a tree with root with decoration $i$ and children $\tau_1,\dots,\tau_n$, we use the notation
\begin{equation*}
    [\tau_1,\dots,\tau_n]_i = \begin{tree}
        \draw (0,0)\mynode{i} -- (-2.7ex,2.5ex)\mynode{$\tau_1$};
        \draw (0,0) -- (-.9ex,1.5ex);
        \draw (0,0) -- (0,1.5ex);
        \draw (0,0) -- (.9ex, 1.5ex);
        \draw (0.15,2.5ex) node[draw = none, fill = none]{$\dots$};
        \draw (0,0) -- (2.7ex,2.5ex)\mynode{$\tau_n$};
\end{tree}
\end{equation*}
We also introduce the operators $B^+_a(\tau_1\dots\tau_n) = [\tau_1,\dots,\tau_n]_a$, which maps forests to trees by adding a root decorated with $a$, as well as the root-cutting operator $B^-([\tau_1,\dots,\tau_n]_i) = \tau_1\dots\tau_n$, which maps trees to forests. 

We say a tree is \emph{ordered} (or \emph{planar}) if the children of each vertex are equipped with an ordering. If none of them have an ordering, we call the tree \emph{unordered} (or \emph{non-planar}). We denote the set of unordered trees with $\cUT$ and the set of ordered trees with $\cOT$. For example, in $\cUT$ it holds that
\begin{equation*}
    \cherry 123 = \cherry 132, \qquad \righttree 1234 = \lefttree 1324
\end{equation*}
as these trees only differ in the order of children of the root. On the flip side, in $\cOT$, we have
\begin{equation*}
    \cherry 123 \neq \cherry 132, \qquad \righttree 1234 \neq \lefttree 1324\,.
\end{equation*}
We call words of decorated trees forests, and will again make a distinction between the unordered forests $\cUF$ and the ordered forests $\cOF$: $\cUF$ is given by \emph{unordered monomials} over $\cUT$, making $\myspan(\cUF)$ simply the symmetric algebra $\mbox{Sym}(\cUT)$ over unordered trees (\cite{HopfAlgebrasCombinatorics}, Example 1.1.3). That is, in $\cUF$ we have
\begin{equation*}
    \tau_1 \tau_2 = \tau_2\tau_1
\end{equation*}
for any $\tau_1,\tau_2\in\cUT$. $\cOF$ is given by ordered monomials or words with letters in $\cOT$. Thus, $\myspan(\cOF)$ is simply the tensor algebra $T(\cOT)$ over ordered trees.\\[.5\baselineskip]

\noindent {\bf Non-planar case: Connes-Kreimer and Grossman-Larson Hopf algebras:} Let us first consider the Hopf algebras over $\cUF$, which are necessary to construct non-planar branched rough paths \cite{ferrucci22}, \cite{HK15}: In this case, $\cH$ will be given by the Connes-Kreimer Hopf algebra \cite{Connes1998} and its graded dual $\cH^g$ is given by the Grossman-Larson Hopf algebra \cite{GROSSMAN1989184}. We refer to \cite{Hoffman03} for the relation between those algebras.

As mentioned above, $\myspan(\cUF)$ forms the symmetric algebra with the (commutative) product being given by
\begin{equation*}
    (\tau_1\dots\tau_k)\cdot (\sigma_1\dots\sigma_l) = \tau_1\dots\tau_k\sigma_1\dots\sigma_l
\end{equation*}
as well as unit element $\one$ being the empty forest. The counit is given by $\epsilon(f) = 0$ for every non-empty forest $f$ and $\epsilon(\one) = 1$. The coproduct is called the Connes-Kreimer coproduct and is given by the sum over all \emph{admissible cuts} of a tree $\tau\in\cUT$. Given a tree $\tau$, an admissible cut is a choice of edges such that each path from some vertex to the root contains at most one cut. Any admissible cut of $\tau$ can be associated with an element in $\myspan(\cUF)\otimes\myspan(\cUF)$ given by $f^{(1)}\otimes f^{(2)}$, where $f^{(1)}$ is given by the ``cut off'' trees gathered as a forest, while $f^{(2)}$ is the left over tree $\tau$ after cutting of $f^{(1)}$. For example
\begin{equation*}
    \begin{tree}
        \node[label=right:{\tiny 1}]{}[grow'=up, level distance = 2ex, sibling distance = 2.5ex]
            child{node [label=right:{\tiny 2}] {}}
                child{node [label=right:{\tiny 3}] {}
                child{node [label=right:{\tiny 4}] {}}};
        \draw[red] (-1.3ex, 1ex) -- (0,1ex);
        \draw[red] (0.6ex,3ex) -- (1.9ex,3ex);
    \end{tree}\qquad\text{corresponds to}\qquad
    \tomato{2}\tomato{4}\otimes\ladderTwo{1}{3}
\end{equation*}
We can then set $\Delta \tau := \tau\otimes 1 + \sum_{\text{adm. cuts}} f^{(1)}\otimes f^{(2)}$. For a forest $f=\tau_1\dots\tau_n$, we set $\Delta f := \Delta \tau_1\dots\Delta\tau_n$. With these operations, $\cH_{CK} = (\myspan(\cUF), \cdot, \Delta,\one,\epsilon)$ forms a bialgebra. We can equip it with a scaling by setting $\abs{f}$ to be the number of vertices in the forest $f$. Thus, $\cH_{CK}$ is a graded bialgebra and can thus be equipped with an antipode to become a Hopf algebra. It is called the Connes-Kreimer Hopf algebra.

We denote its graded dual by $\cH_{GL}$ since it forms the Grossman-Larson algebra. We will not associate any basis element in $\cF$ with its dual basis but rather denote with any forest $f\in\cH_{GL}$ the element of $\cH_{CK}^*$, such that
\begin{equation*}
    \scalar{f,g} = sg(f)\delta_{f,g}
\end{equation*}
holds for all $g\in\cUF$, where $sg(f)$ is the \emph{symmetry factor} of $f$, which is recursively given by
\begin{equation*}
    sg(\one) = 1 \qquad sg(\tau_1\dots\tau_n) = n!\prod_{i=1}^n sg(\tau_i) \qquad sg([\tau_1,\dots,\tau_n]_a) = sg(\tau_1\dots,\tau_n)\,,
\end{equation*}
for all trees $\tau_1\dots\tau_n\in \cUT$. This construction causes the dual product $\star$ in $\cH_{GL}$ defined by $\scalar{f\star g, h} = \scalar{f\otimes g,\Delta h}$ to have the following form: For $f= \tau_1\dots\tau_n$, $\tau_i\in\cUT$, $i=1\dots,n$ as well as a $\sigma\in\cUT$, we define the grafting product $f\curvearrowright \sigma$ to be the sum over all possibilities to grow the trees $\tau_1,\dots,\tau_n$ out of vertices of $\sigma$. Then, $\star$ can be defined by adding a root to the forest $g$ and grafting $f$ onto the new forest before removing the root again: $f\star g := B^-(f\curvearrowright B^+_1(g))$. Note that for two trees $\tau,\sigma\in\cUT$, it holds that $\tau\star\sigma = \tau\sigma+\tau\curvearrowright\sigma$. Let us give an example of this operation:
\begin{equation*}
    (\tomato 1\tomato2)\star\ladderTwo 34 = \Ytree 3412 + \righttree 3142 + \righttree 3241 + \bushThree 3124 + \tomato 1 \cherry 324 + \tomato 1 \ladderThree342 + \tomato 2 \cherry 314 + \tomato 2 \ladderThree341 + \tomato 1\tomato 2\ladderTwo34\,.
\end{equation*}
The dual coproduct $\delta$ defined by $\scalar{\delta f, h\otimes g} = \scalar{f,hg}$ is given by the deconcatenation, corrected with the appropriate symmetry factors:
\begin{equation*}
    \delta f = \sum_{gh = f} \frac{sg(f)}{sg(g)sg(h)} g\otimes h\,.
\end{equation*}
$\one,\epsilon$ and the grading is given as in $\cH_{CK}$ and the antipode is simply the dual operator of the antipode $S$ in $\cH_{CK}$. 

In $\cH_{GL}$, the primitive elements are given as the vector space spanned by $\cUT$ of trees, as Proposition 2.11 from \cite{HK15} shows.

A rough path in $\cH_{GL}$ is called a non-planarly branched rough path.\\[.5\baselineskip]

\noindent {\bf Planar case: The Munthe-Kaas-Wright Hopf algebra:} Planar branched rough paths \cite{curry20} need to be constructed over a Hopf algebra constructed from ordered forests in $\cOF$. This Hopf algebra was constructed in \cite{MKW08} and played a vital role in the analysis of Lie-group integrators \cite{Lundervold2009HopfAO}, \cite{Lundervold2013}. We will denote it with $\cH_{MKW}$ and its graded dual with $\cH_{MKW}^g$. Note that in the 1989 paper of Grossman-Larson \cite{GROSSMAN1989184}, they discuss both unordered trees and ordered trees, which makes it sensible to also call $\cH_{MKW}^g$ a Grossman-Larson Hopf algebra. To avoid confusion, we will reserve this name for the algebra $\cH_{GL}$ over unordered trees and strictly speak of the Munthe-Kaas-Wright algebra and its graded dual for ordered trees.

As mentioned before, $\cH_{MKW}$ is given by the tensor algebra (or more precisely, the Shuffle algebra) $T(\cOF)$ over the ordered forests $\cOF$. We equip it with the empty forest as unit element $\one = \emptyset$ and the shuffle product $\shuffle$ to form an associative, commutative algebra. The counit is given by $\epsilon(\one) = 1$ and $\epsilon(f) = 0$ for all non-empty forests $f$, as before. The coproduct, denoted by $\Delta$ can be constructed as the sum over all \emph{full left admissible cuts} (\cite{MKW08}, Def. 6), but it is easier to construct a product $\star$ on $\cH_{MKW}^g$ and define $\Delta$ to be its dual. The grading on $\cH_{MKW}$ is given by the number of vertices, as before.

The graded dual $\cH_{MKW}^g$ can be equipped with the dual operation $\delta$ (making it the deshuffle) and uses the same unit and counit as $\cH_{MKW}$. In the ordered case, we do not need symmetry factors and can simply identify the elements in $\cH_{MKW}^g$ with linear combinations of $\cOF$ by requiring
\begin{equation*}
    \scalar{f,g} = \delta_{f,g}    
\end{equation*}
for all ordered forests $f,g\in\cOF$.The grading on $\cOF$ is again given by the number of vertices in a forest. To make $\cH_{MKW}, \cH_{MKW}^g$ into graded bialgebras and thus Hopf algebras, it remains to construct the product $\star$ on $\cH_{MKW}^g$: For two trees $\tau,\sigma$, we set the \emph{left grafting} $\tau\curvearrowright_l \sigma$ to be the sum over all possibilities, to grow $\tau$ out of a vertex of $\sigma$ \emph{as the left-most child} of said vertex. For a forest $f = \tau_1\dots\tau_n$, we set $f\curvearrowright_l \sigma$ to be the sum over all possibilities to grow the trees $\tau_1,\dots,\tau_n$ out of vertices of $\sigma$ as the left-most child, with the extra condition that if two or more trees $\tau_1,\dots,\tau_k$, $k\ge 2$ grow out of the same vertex $p$ in $\sigma$, then they need to have the same order as children of $\sigma$ as they had as trees in $f$. For example,
\begin{equation*}
    (\tomato 1\tomato 2) \curvearrowright_l \ladderTwo{2}{3} = \Ytree 3412 + \righttree 3142 + \righttree 3241 + \bushThree 3124\,.
\end{equation*}
As before, we can construct $\star$ by adding and removing a root in a smart way: $f\star g := B^-(f\curvearrowright_l B^+_1(g))$. Since $\cH_{MKW}^g$ as a coalgebra is the shuffle coalgebra over $T(\cOT)$, we immediately get from \cite{freeLieAlgebras} that the primitive elements are given by the free Lie algebra over $\cOT$ with respect to the commutator $[f,g]_\otimes = fg-gf$. It should be noted that this is not the commutator of $\cH^g_{MKW}$ as a Hopf algebra, but since $\star$ is associative and $\delta$ is a $\star$-homomorphism, $\cP$ is also a Lie algebra with respect to the $\star$-commutator $[f,g]_\star = f\star g-g\star f$.

The rough paths in $\cH_{MKW}^g$ are called planarly branched rough paths.

\begin{rem}
    With the use of so-called aborification maps, one can show that any geometric rough path gives rise to a canonical non-planarly branched rough path as well as a planarly branched rough path. However, it is in general not true that a non-planarly branched rough path gives rise to a planarly branched rough path or vice versa. For more information, we refer to \cite{curry20}.
\end{rem}

\subsection{Vector fields, differential operators and linear connections}

Let us use this section to recall some notions from differential geometry: Let $M$ be a smooth manifold. A smooth vector field on $M$ is map $V: M\to TM$, such that $m\mapsto V(M)\in T_mM$. In coordinates, every vector field can be expressed in the form
\begin{equation*}
    V^\phi(m) = \sum_{n=1}^d c_n(m) \partial_n\,,
\end{equation*}
where $\partial_n$ is the derivative of $\phi$ evaluated in the $n-th$ unit direction. $V$ acts on $C^\infty(M)$ by
\begin{equation*}
    V\psi(m) = \sum_{n=1}^d c_w(m)\partial_n\psi(m)\,,
\end{equation*}
which gives us a map mapping smooth vector fields into smooth differential operators on $M$. Let us recall the definition of a smooth differential operator:

\begin{defn}
    A smooth map $F:C^\infty(M)\to C^\infty(M)$ is called a smooth differential operator if there is an $N>0$ such that $F$ has the following expression in each coordinate chart:
    \begin{equation*}
        F\psi(x) = \sum_{\abs w \le N}c_w(x)\partial_w\psi(x)\,,
    \end{equation*}
    where $\partial_w \psi = \partial_{w_1}\dots\partial_{w_{\abs w}} \psi$ for each word $w= w_1\dots w_n$ over the alphabet $A = \{1,\dots, d\}$, and $c_w\in C^\infty(U)$ for the open set $U\subset M$ on which the coordinate function lives. If at least one $c_w\neq 0$ for a $\abs w= N$, we call $N$ the order of $F$. If $F$ is of order $1$, we call it a vector field. We denote by $\cD$ the space of differential operators and by $\cV$ the space of vector fields.
\end{defn}

\noindent It should be noted that $c_w$ are coordinate dependent, but the order of $F$ is not. If we have a vector field $F$, one can regain the map $V:M\to TM$ by applying $F$ to the coordinate functions to get $F\phi(m)\in \RR^d \cong T_mM$. It is further well known that the vector fields are uniquely characterized as the differential operators $F\in\cD$ such that the Leipniz rule holds: For all $\phi,\psi\in C^\infty(M)$,
\begin{equation*}
    F(\phi\cdot\psi) = F(\phi)\cdot\psi +\phi\cdot F(\psi)\,.
\end{equation*}
$(\cD,\circ,\Id)$ forms an associative algebra, where $\Id$ is the identity map and $\circ$ is composition. The set of vector fields $\cV$ equipped with the commutator $[V,U] = V\circ U- U\circ V$ form a Lie algebra.

Given a smooth vector field $V$ on $M$, the initial value problem
\begin{equation*}
    \begin{cases}
        dZ_t = V(Z_t) \\
        Z_0 = z \in M
    \end{cases}
\end{equation*}
has a unique solution which we denote by $\mu_t z = Z_t$ for all $t\le T$ up to some explosion time $T(z)$. Since everything is smooth, we can choose $T:M\to \RR_+$ to be a smooth map on $M$. We assume that $M$ does not have a border so that $T(z) > 0$ for all $z\in M$. It follows that we can assign each compact set $K\subset M$ an explosion time $T(K) = \inf_{z\in K}T(z) >0$. Furthermore, it is well known that $\mu_{t}z$ is Lipschitz-continuous in the starting value $z$: If we have $x,y\in K\subset O$ for some compact set $K$ and open neighbourhood $O\subset M$, such that $\mu_{t}x,\mu_{t}y\in O$, and we have a coordinate chart $\phi:O\to U\subset\RR^d$, it holds that
\begin{equation*}
    \abs{\phi(\mu_{t}(x))-\phi(\mu_{t}(y))}\le L(\phi, K, t)\abs{\phi(x)-\phi(y)}\,,
\end{equation*}
for $x,y\in K$, $t\le T(K)$, with $L(\phi, K, t)\to 1$ as $t\to 0$. If $M = \RR^d$ is just the flat space, we can find $T$ explicitly: For any $K\subset \RR^d$ compact, we can equip $V$ with the norm
\[
\norm{V}_K = \norm{V}_{\infty,K}+ \norm{V}_{\Lip,K}\,,
\]
where $\norm{V}_{\infty,K} = \sup_{x\in K}\abs{V(x)}$ and $\norm{V}_{\Lip,K} = \sup_{x\neq y\in K}\frac{\abs{V(x)-V(y)}}{\abs{x-y}}$ are just the usual supremum norm and Lipschitz norm. Let $\mathring K$ be the interior of $K$. Using $Z_t = z+\int_0^t V(Z_s) ds$, we get that $Z_t$ can not leave $K$, as long as $T\norm{V}_K\le \dist(z,\mathring K^C)$. By a standard fixed-point argument, it follows that we get a solution up to time $T(z)$ fulfilling
\begin{equation}\label{eq:T}
    T(z)\norm{V}_K\le \min(1,\dist(z, \mathring K^c))\,.
\end{equation}

\subsection{Linear connections}

For non-geometric rough paths, one requires more structure on a manifold to solve an RDE, as demonstrated by \cite{emilio22}. This extra structure is given by a linear connection (\cite{riemannian_manifolds}, \cite{hsu}, \cite{StochCalcEmery}). Note that from this chapter onwards, we will use Einstein summation notation. 

Let $\cV(M)$ denote the set of smooth vector fields on $M$ (a smooth $d$-dimensional manifold). A linear connection is a form to differentiate a vector field in the direction of another vector field:

\begin{defn}
    A smooth, linear connection (or covariant derivative) is a smooth map $\nabla: \cV(M)\times\cV(M)\to\cV(M)$, which is linear over $C^\infty(M)$ in the first component and fulfills the Leipniz rule in the second component:
    \begin{align*}
        \nabla_{fU} V &= f\nabla_U V \\
        \nabla_U fV &= U(f) V + f\nabla_U V\,.
    \end{align*}
    For all $U,V\in\cV(M)$, $f\in C^\infty(M)$. We call $\nabla_U V$ the  covariant derivative of $V$ in direction $U$. For all fixed $V$, the map $\nabla V:\cV(M)\to\cV(M)$, $U\mapsto \nabla_U V$ is called the total covariant derivative of $V$.
\end{defn}

\noindent One can think of $\nabla_U V$ as a directional derivative of $V$ in direction $U$. In coordinates, the easiest way to deal with connections is to denote by $\partial_i$ the (local) vector field generated by the coordinates and denote the Christoffel symbols as the smooth functions $\Gamma_{i,j}^k$ in $C^\infty(O)$ for some open set $O\subset M$ given by
\begin{equation*}
    \nabla_{\partial_i}\partial_j = \Gamma_{i,j}^k\partial_k\,.
\end{equation*}
Since $\nabla$ is smooth, the Christoffel symbols will be smooth functions. We introduce higher order covariant derivatives as in \cite{riemannian_manifolds}, pages 53-54: We denote the tangent space of $M$ at a point $x\in M$ by $T_xM$ and its dual space by $T_xM^*$. A smooth map from $M$ into the tangent bundle is called a vector field and a map into the cotangent bundle $TM^*$ is called a covector field. We denote $T_nM := TM^{\otimes n}$ and $T^kM := (TM^*)^{\otimes k}$ last but not least, $T_n^k M := T_n M \otimes T^k M$. We can identify $T_n^k M$ with smooth sections of linear maps
\[ F(x): (T_x M^*)^{\otimes n} \otimes (T_x M)^{\otimes k} \to \RR\]
and thus, we can always find smooth maps $F^{i_1,\dots, i_n}_{j_1,\dots, j_k}:M\to \RR$ such that
\begin{equation*}
    F(x) = F^{i_1,\dots,i_n}_{j_1,\dots,j_k}(x) \partial_{i_1}\otimes \dots\otimes\partial_{i_n}\otimes d^{j_1}\otimes\dots\otimes d^{j_k}.
\end{equation*}
With these spaces, we can define the n-th covariant derivative quite easily. Let $V$ be a vector field. Then the covariant derivative in the direction of $V$ is defined by:

\begin{itemize}
    \item For a function $\phi\in T^0_0 M = C^\infty(M)$, we have $\nabla_V \phi = V\phi$. If $U$ is another vector field, we $\nabla_V U$ is given by the covariant derivative.

    \item For an $F\in T^k_n M$, we define

    \begin{align}
    \begin{split}\label{covariant_der_general}
        \nabla_V F(\omega_1,\dots,\omega_n, U_1,\dots, U_k) &= V (F(\omega_1,\dots,\omega_n, U_1,\dots, U_k)) \\
        &\qquad - \sum_{j=1}^n F(\omega_1,\dots, \nabla_V \omega_j,\dots,\omega_k,U_1,\dots, U_k)\\
        &\qquad - \sum_{j=1}^k F(\omega_1,\dots, \omega_k,U_1,\dots, \nabla_V U_j,\dots, U_n)
    \end{split}
    \end{align}
\end{itemize}

\noindent Note that this definition seems to require us to first define $\nabla_V \omega$ for a vector field $V$ and a one-form $\omega$. Since a one form $\omega(U)$ only has a vector field as argument, \eqref{covariant_der_general} gives us a directly $\nabla _V\omega(U) = V(\omega(U))-\omega(\nabla_V U)$, which is well-defined. In coordinates, it reads
\begin{equation}\label{eq:covOneForm}
\nabla_V\omega = (V^\alpha\partial_\alpha\omega_i-V^\alpha\omega_k\Gamma^k_{\alpha,i})d^i\,.
\end{equation}
Thus, \eqref{covariant_der_general} is well-defined for all $F\in T^k_n M$.

We define the total covariant derivative of $F$ by $\nabla F\in T^{k+1}$
\begin{equation*}
    \nabla F(\omega_1\dots,\omega_n, V,U_1,\dots,U_k) := \nabla_V F(\omega_1,\dots,\omega_n, U_1,\dots, U_k).
\end{equation*}
The $m$-th covariant derivative of $F$ is then simply given inductively as $\nabla^m F\in T^{k+m}_n M$, $\nabla^m F := \nabla(\nabla^{m-1} F)$.

In practice, we are only interested in the $m$-th covariant derivative of functions and vector fields, for which the above calculations can be somewhat simplified. Let us pick a coordinate chart, such that we can express all vector fields via $V = V^i\partial_i$ and the one-forms as $\omega = \omega_j d^j$. As discussed above, the covariant derivative of a vector field is then given via
\begin{equation*}
    \nabla_V U = (V^i\partial_i U^k + V^i U^j\Gamma_{i,j}^k )\partial_k,
\end{equation*}
and the covariant derivative of a function is simply $\nabla_V \phi = V^i\partial_i \phi$. Using \eqref{eq:covOneForm} on the derivative of a function $\phi\in C^\infty (M)$ given by $\omega = d\phi = \partial_i\phi d^i$ we get $\nabla d\phi = \nabla^2\phi = (\partial_i\partial_j\phi -\Gamma_{i,j}\partial_k \phi)d^{i}\otimes d^j$. More generally, we get the $n-th$ covariant derivative of $\phi$ inductively via 
\begin{equation*}
    \nabla^n\phi(U_1,\dots, U_n) = U_1(\nabla^{n-1} \phi(U_2,\dots, U_n))- \sum_{k=2}^n\nabla^{n-1}\phi(U_2,\dots,\nabla_{U_1} U_k,\dots, U_n)\,,
\end{equation*}
which is again a function in $C^\infty(M)$. For vector fields, we can use the relation $V\phi = V(d\phi)$ to get a inductive formula with one more term:
\begin{align*}
    \nabla^n V(U_1,\dots,U_n)\phi &= U_1(\nabla^{n-1} V(U_2,\dots, U_n)\phi)\\
    &\qquad - \sum_{k=2}^n\nabla^{n-1} V(U_2,\dots, \nabla_{U_1} U_k,\dots, U_n)(\phi)\\
    &\qquad - \nabla^{n-1} V(\nabla_{U_1} d\phi, U_2,\dots, U_n)\,.
\end{align*}
From an algebraic point of view, a connection gives us a non-commutative, non-associative multiplication $\cV\otimes\cV\to\cV$, $U\otimes V \mapsto U\triangleright V := \nabla_U V$ on the space of vector fields. In this case, it is common to control this multiplication with the use of the \emph{torsion} and \emph{curvature} of the connection, which are given by
\begin{align*}
    T(X,Y) &= \nabla_X Y - \nabla_Y X - [X,Y]\\
    R(X,Y)Z &= \nabla_X\nabla_Y Z -\nabla_Y\nabla_X Z -\nabla_{[X,Y]}Z
\end{align*}
for vector fields $X,Y,Z$, where $[X,Y] = X\circ Y- Y\circ X$ is the commutator of the two vector fields. While $(V,\triangleright)$ is not an associative algebra, one can put extra conditions on $T$ and $R$ to turn it into a pre-Lie or post-Lie algebra, see Section \ref{sec:elementaryDiff} for details.

\section{Pseudo bialgebra maps and constructing local flows from rough paths} \label{sec:FlowsGenByRP}

\subsection{Hopf algebraic structures on $\cD$ and pseudo bialgebra maps}

As mentioned above, $\cD$ forms an associative algebra with composition as a product. If one chooses a reference point $m\in M$, if further acts on the function space $C^\infty M$ via the pairing
\begin{equation*}
    \scalar{F,\phi} = F\phi(m)\,.
\end{equation*}
Since $C^\infty(M)$ is itself an associative algebra with pointwise multiplication as a product, this raises the question if we can turn $\cD$ into a bialgebra or even Hopf algebra by finding a coproduct such that $\scalar{\Delta F,\phi\otimes\psi} = \scalar{F,\phi\cdot\psi}$. The short answer is no: Since $C^\infty(M)$ is infinite-dimensional, the dual operation of $\cdot$ is not a coproduct, and $\cD$ with the pairing described above is clearly not the whole dual space of $C^\infty(M)$.

It is however possible to construct Hopf algebras on subsets of $\cD$: The simplest example is given by Example 2.2 of \cite{HK15}, but more advanced structures have been found for example in \cite{Lundervold2009HopfAO} or \cite{varilly}.

However, in general, it is not necessary to precisely describe the Hopf algebraic structures of subsets of $\cD$: The approach generally taken in Butcher series theory, as well as rough path theory, is to construct abstract Hopf algebras $\cH, \cH^g$ beforehand and find suitable mappings, also called \emph{elementary differentials} \cite{Lundervold2009HopfAO} to map $\cH^g$ into $\cD$. 

In our paper, this class of elementary differentials will be given by what we call \emph{pseudo bialgebra maps}, which are supposed to connect $\cH$ and $\cD$ in such a way, that
\begin{itemize}
    \item[a)] $\star$ in $\cH$ corresponds to the composition $\circ$ in $\cD$ and
    \item[b)] $\Delta$ is ``dual'' to the pointwise product between functions.
\end{itemize}
This leads to the following definition:

\begin{defn}
    Let $\cH$ be a bialgebra. We say that a linear map $\cF:\cH\to\cD$ is a pseudo bialgebra map, if 

    \begin{itemize}
        \item[a)] $\cF:(\cH,\star,\one)\to (\cD,\circ,\Id)$ is an associative algebra homomorphism and
        \item[b)] For all $x\in\cH$ and $\phi,\psi\in C^\infty(M)$, we have
            \begin{equation*}
                \cF(\Delta x)(\psi\otimes\phi) = \cF(x)(\phi\cdot\psi)
            \end{equation*}
    \end{itemize}
\end{defn}

\begin{rem}
    Since $\cdot$ is a commutative product in $C^\infty(M)$, this gives an intuitive meaning to the requirement, that our rough paths need to live in an cocommutative Hopf algebra.
\end{rem}

\subsection{Pseudo bialgebra maps and Lie algebra maps}\label{subsec:BiMapsFromLieMaps}

From the definition of a pseudo bialgebra map, one can easily show the following:

\begin{prop}
    Let $p\in \cP\subset\cH$ be a primitive object. If $\cF$ is a pseudo bialgebra map, then $\cF(p)$ is a vector field.
\end{prop}

\begin{proof}
    Let $\cF$ be a pseudo bialgebra map and $p\in\cP$. Then:
    \begin{equation*}
        \cF(p)(\phi\cdot\psi) = \cF(p\otimes\one+\one\otimes p)(\phi\otimes\psi) = \cF(p)\phi\cdot\psi + \phi\cdot\cF(p)\psi
    \end{equation*}
    holds for all $\phi,\psi\in C^\infty(M)$. Thus, $\cF(p)$ fulfills the Leipniz rule and is a vector field.
\end{proof}

\noindent Recall that $\cP$ equipped with the commutator $[x,y] = x\star y-y\star x$ forms a Lie algebra, and since $\cF$ is an algebra-homomorphism, we immediately get that $\cF\vert_\cP:\cP\to\cV$ is a Lie algebra-morphism. 

An interesting observation from Section \ref{subsec:flowsGeneratedFromRP} is that we only use $\cF\vert_\cP$ to construct the solution flow to a rough differential equation. This raises the question, of whether a Lie algebra map mapping $\cP\to\cV$ already generates a pseudo bialgebra map on all of $\cH$.

The short answer is no, it generates a pseudo bialgebra map on the universal enveloping algebra $U(\cP)$. However, all relevant Hopf algebras for rough paths are cocommutative, so the Theorem of Milnor-Moore gives us, that $\cH\cong U(\cP)$. That implies that any Lie algebra map can be extended to a Hopf algebra map: Let $\tilde \cH$ be another Hopf algebra and let $\cF:\cP(\cH)\to\cP(\tilde\cH)$ be a Lie map. $\cF$ extends to an algebra map $\cF:\cH\to \tilde\cH$ by the universal property of $\cU(\cP(\cH))$. With a bit of work, one can see that it suffices that $\cF$ maps primitive elements into primitive elements to show that it is a Hopf-algebra map (here we just use that $\Delta,\epsilon$ are homomorphism and $S$ is an anti-homomorphism). As it turns out, this result extends to pseudo bialgebra maps:

\begin{thm}
    Let $\cF:\cP\to \cV$ be a Lie map and let $\tilde\cF:\cU(\cP)\to \cD$ be the algebra map generated by $\cF$. Then $\tilde\cF$ is a pseudo bialgebra map.
\end{thm}

\begin{proof}
    We already know that $\tilde\cF$ is an algebra map, so it suffices to show that
    \begin{equation*}
        \cF(\Delta x)(\phi\otimes\psi) = \cF(x)(\phi\cdot\psi)\,.
    \end{equation*}
    For a primitive element $x\in\cP$, this holds since $\cF(x)$ is a vector field and thus obeys the Leipniz rule. We use that every element $w\in\cU(\cP)$ can be written as a word $w= (w_1\dots w_n)$ with letters being primitive elements $w_i\in\cP$, $i=1,\dots n$. For a general word $w\in\cU(\cP)$, it holds that
    \begin{align*}
        \cF(w)(\phi\cdot\psi) &= \cF(w_1)\circ\dots\circ\cF(w_{\abs w})(\phi\cdot\psi)\\
        &= (\Id\otimes\cF(w_1)+\cF(w_1)\otimes\Id)\circ\dots\circ(\Id\otimes\cF(w_{\abs{w}})+\cF(w_{\abs w})\otimes\Id)(\phi\cdot\psi)\\
        &=\cF(\Delta x)(\phi\otimes\psi)\,,
    \end{align*}
    where we use that $\Delta x$ is just the deshuffle of $w$ in $\cU(\cP)$.
\end{proof}

\subsection{Constructing almost-flows from rough paths}\label{subsec:flowsGeneratedFromRP}


The goal of this subsection is to show that any rough path $\XX$ together with a pseudo bialgebra map generates an almost-flow on $M$, and thus a flow on $M$ by Section \ref{sec:flows}. Thus, the only missing ingredient to solve
\[
dY_t = V_i(Y_t) d\XX_t^i\,.
\]
is the construction of pseudo bialgebra maps $\cF(V_1,\dots, V_n)$ for the vector fields $V_1,\dots V_n$. This will be the topic of Section\ref{sec:elementaryDiff}.

Let us start by fixing a connected, graded, commutative Hopf algebra $\cH$ such that its graded dual $\cH^g$ is a connected, graded, cocommutative Hopf algebra, and a rough path $\XX$ in $\cH^{N*}$. Let $\alpha\in (0,1)$ be the Hölder continuity of $\XX$ and $N$ be such that $N\alpha\le 1<(N+1)\alpha$. Further, let $\cF:\cH^g\to \cD$ be a pseudo bialgebra map. Note that we consider $\cF$ and $\XX$ fixed, and will not point out if constants explicitly depend on $\cF$ or $\norm{\XX}$.

We construct the almost-flow $\mu_{s,t}$ with the \emph{log-ODE} method, also used in \cite{bailleul15}, \cite{bailleul19} to construct rough flows. For more general information about the log-ODE method, we recommend \cite{logODECastell}. Since $\LL_{s,t} := \log_N(\XX_{s,t})\in\cP^N\subset\cP$ is a primitive element of $\cH^g$ for all $(s,t)\in\Delta_T$, and thus $\cF(\LL_{s,t})\in \cV$ is a vector field. Thus, for fixed $(s,t)\in\Delta_T$, we can consider the initial value problem
\begin{equation}\label{eq:log-ODE}
    \begin{cases}
        dZ_u = \cF(\LL_{s,t})(Z_u) \\
        Z_0 = z\in M\,.
    \end{cases}
\end{equation}
Since $\cF$ maps into smooth vector fields, this has a unique solution up to some explosion time $T(z)$. We claim that if $z$ is an inner point of $M$ and for small enough $\abs{t-s}$, $T(z)\ge 1$. In this case, we set $\mu_{s,t}(z) := Z_1$. To see that the claim holds, let $z$ be an inner point of $M$ and fix some coordinate function $\phi:M\supset O \to U\subset\RR^d$ such that $z\in O$. We denote by $\cF(\LL_{s,t})^\phi$ the vector field over $U$ given by $x\mapsto \cF(\LL_{s,t})\phi(x)$. We further assume that $U\subset K \subset\RR^d$ is a subset of some compact set, otherwise, we restrict it to $U\cap \bar B_1(\phi(z))$, where $\bar B_1(\phi(z))$ is the closed ball of radius $1$ around $\phi(z)$. It follows that $\norm{\cF(\LL_{s,t})^\phi}_U$ is finite. Furthermore, we can use Lemma \ref{lem:boundNormL} together with $\cF(\LL_{s,t}) = \sum_{\abs{\tau}\le N} \scalar{\LL_{s,t},\tau}\cF(\tau)$ to bound it with
\begin{align*}
    \norm{\cF(\LL_{s,t})}_U &\le \sum_{\abs{\tau}\le N} \abs{\scalar{\LL_{s,t},\tau}}\norm{\cF(\tau)}_U\\
    &\le C\sup_{k = 1,\dots,N}\norm{X}^k \sum_{\abs{\tau}\le N}\norm{\cF(\tau)}_U \abs{t-s}^\alpha
\end{align*}
for all $\abs{t-s}\le 1$, and where the $\tau$ form a basis of $\cH^{N*} = \bigoplus_{n\le N}\cH^{(n)*}$. Thus, by choosing $T$ small enough such that $\norm{\cF(\LL_{s,t})}_U\le \min(1,\dist(z,U^c))$, \eqref{eq:T} gives us that the explosion time is larger or equal to $1$. Since the $z$-dependence of $T(z)$ only depends on $\dist(z,U^c)$, we further see that we can choose $T$ continuously over $z$. Thus, as long as $M$ does not have a boundary, $(s,t,z)\mapsto \mu_{s,t}(z)$ exists on some admissible domain $\diag_T\times M \subset O \subset \Delta_T\times M$.

We can now present the main result of this section:

\begin{thm}\label{theo:main1}
    Let $\mu: O\to M$ be the solution to \eqref{eq:log-ODE} evaluated at time $1$. Then $\mu$ is an almost-flow on $M$. 
\end{thm}

\noindent While it is rather straightforward to show the Lipschitz and Hölder conditions of $\mu$, the almost-flow property needs a bit more work. Our main strategy to show the almost-flow property is to show a Taylor formula for $\mu_{s,t} z$, and show that the Taylor approximation fulfills the almost-flow property. To get started, let us show a general result for smooth differential operators, which we need for the joined Lipschitz-almost-flow property:

\begin{lem}\label{lem:DiffOpLip}
    Let $F\in\cD$ be a smooth differentiable operator of order $m$ and $\phi\in C^{m+1}(M,\RR^d)$. If $\phi$ is additionally a diffeomorphism, we have that for each $x,y\in K$ for some compact set $K\subset M$:
    \begin{equation*}
        \abs{F\phi(x)-F\phi(y)} \le C\abs{\phi(x)-\phi(y)}\,,
    \end{equation*}
    where $C$ only depends on $F$ and $K$ and $\phi$.
\end{lem}

\begin{proof}
    Without loss of generality, assume that $x,y$ are in some coordinate chart $\psi$, with $F = \sum_{\abs w\le m} f^w\partial_w$ in the said chart. It follows that
    \[
        F\phi(x) = \sum_{\abs w\le m} (f^w\partial_w\phi)(\phi^{-1} \circ\phi(x))\,.
    \]
    Since $(f^w\partial_w\phi)\circ\phi^{-1}$ is continuously differentiable for each word $\abs w\le m$, we get
    \[
    \abs{F\phi(x)-F\phi(y)} \le C \norm{\phi}_{C^{m+1}}\norm{\phi^{-1}}_{C^1}\abs{\phi(x)-\phi(y)}\,,
    \]
    where $\norm{\phi}_{C^{m+1}},\norm{\phi}_{C^1}$ are the respective norms with respect to the coordinate chart $\psi$. 
\end{proof}

\noindent Any map $\nu:O \to M$ has an associated operator $\nu_*$ mapping any function $\phi\in C^\infty(M)$ into a function $\nu_*\phi: O \to \RR$ via $\nu_{s,t*} \phi(x) = \phi(\nu_{s,t} x)$. We call $\nu_*$ the push-forward of $\nu$. Note that one can easily recover $\nu$ from its push-forward by applying $\nu_*$ to some coordinate function $\phi:M\supset V\to U\subset\RR^d$. Furthermore, it holds that the map $\nu\mapsto \nu_*$ is anti-homomorph in the following sense: For two maps $\nu,\eta:M\to M$, we have $(\nu\circ \eta)_*\phi = (\eta_*\circ\nu_*)\phi$.

With this notation, we can show that the push-forward of the map $\mu$ from Theorem \ref{theo:main1} has the following Taylor decomposition:

\begin{lem}[Taylor decomposition]
    Let $\phi\in C^\infty(M)$. Let $K\subset M$ be a compact set and $\abs{t-s}\le T(K)$, such that $(s,t,x)\in O$ for all $x\in K$. then there is a constant $C$ only depending on $K$ and $\phi$ and $\norm{X}$, such that
    \begin{equation}\label{eq:TaylorDecomp}
        \abs{(\mu_{s,t*}-\cF(\XX_{s,t}))\phi(x)}\le C\abs{t-s}^{(N+1)\alpha}\,.
    \end{equation}
    If additionally, $\phi:\tilde U\to U$ is a diffeomorphism from some open set $\tilde U\subset M$ to $U\subset\RR^d$, we have that for all $x,y\in K\cap \tilde U$:
    \begin{equation}\label{eq:TaylorDecompJoinedProp}
        \abs{(\mu_{s,t*}-\cF(\XX_{s,t}))(\phi(x)-\phi(y))}\le \tilde C\abs{t-s}^{(N+1)\alpha}\abs{\phi(x)-\phi(y)}\,.
    \end{equation}
    for another constant $\tilde C$ only depending on $K$, $\phi$ and $\norm{\XX}$.
\end{lem}

\begin{proof}
    We start by showing \eqref{eq:TaylorDecomp}, using the same proof as in \cite{bailleul15}: We write $\LL_{s,t} = \sum_{k=1}^n\LL_{s,t}^k$, there $\LL_{s,t}^k$ is the projection of $\LL_{s,t}$ onto $\cH^{(k)}$. Note that by Lemma \ref{lem:pkIsPrimitive}, $\LL_{s,t}^k$ is a primitive element, which implies that $\cF(\LL_{s,t}^k)$ is a vector field. Since $\phi(Z_u)$ solves the initial value problem
    \begin{equation*}
        \begin{cases}
            \frac{\partial}{\partial u}\phi(Z_u) = \cF(\LL_{s,t})\phi(Z_u) \\
            \phi(Z_0) = \phi(x)\,,
        \end{cases}
    \end{equation*}
    it follows that we get the integral identity
    \begin{equation*}
        \phi(Z_u) = \phi(x) + \int_0^u \cF(\LL_{s,t})\phi(Z_r) dr\,.
    \end{equation*}
    By iterating that identity and decomposing $\LL_{s,t}$, we get
    \begin{align*}
        \phi(\mu_{s,t} x) & = \sum_{k=0}^N \frac{\cF(\LL_{s,t})^k}{k!}\phi(x) + \underbrace{\int_{0}^1\dots\int_0^{r_{N-1}}\cF(\LL_{s,t})^{N+1}\phi(Z_{r_{N}}) dr_{N}\dots dr_{1}}_{R^1}\\
        & = \cF(\exp_N(\LL_{s,t}))\phi(x) +\underbrace{\sum_{m=2}^{N} \sum_{k_1+\dots+k_m\ge N} \prod_{j=1}^m\frac{\cF(\LL_{s,t}^{k_j})}{k_j!} \phi(x)}_{R_2} +R_1\\
        &= \cF(\XX_{s,t})\phi(x) + R_1+R_2\,.
    \end{align*}
    It remains to show that $\abs{R_i}\lesssim \abs{t-s}^{(N+1)\alpha}$ holds for $i=1,2$. To do so, let $(p_1,\dots, p_{M_k})$ be a basis for the primitive elements of order $k$ for each $k=1,\dots m$. It then holds that
    \begin{align}
    \begin{split}\label{ineq:LLk}
        \abs{\cF(\LL_{s,t}^k)\phi(x)} &\le C\norm{X}^k \sup_{i=1,\dots M_k}\abs{\cF(p_i)\phi(x)} \abs{t-s}^{k\alpha}\\
        &\le C \norm{X}^k \abs{t-s}^{k\alpha}\,,
    \end{split}
    \end{align}
    where we allow the constant $C$ to change between the lines and does depend on $\cF$ and $\phi$. By iterating this, we can calculate 
    \begin{equation}\label{ineq:prodFL}
    \abs{\prod_{j=1}^m\frac{\cF(\LL_{s,t}^{k_j})}{k_j!} \phi(x)}\le C \norm{X}^{k_1+\dots+k_m}\abs{t-s}^{\alpha(k_1+\dots+k_m)}\,.
    \end{equation}
    This shows the desired inequality for $R_2$. The one from $R_1$ follows analogously by decomposing $\LL_{s,t} = \sum_{k=1}^N\LL_{s,t}^k$.\\[.5\baselineskip]

    \noindent The proof of \eqref{eq:TaylorDecompJoinedProp} is similar. Note that by Lemma \ref{lem:DiffOpLip}, \eqref{ineq:LLk} becomes
    \begin{equation*}
        \abs{\cF(\LL_{s,t}^k)(\phi(x)-\phi(y)))} \le C \norm{X}^k\abs{t-s}^{k\alpha}\abs{\phi(x)-\phi(y)}\,.
    \end{equation*}
    Plugging this into \eqref{ineq:prodFL} gives the desired result for $R_2$. $R_1$ follows as before.
\end{proof}

\noindent We now know that we can approximate $\mu_{s,t*}$ with its Taylor decomposition $\cF(\XX_{s,t})$. To show that $\mu_{s,t}$ is an almost-flow, it is easier to show that $\cF(\XX_{s,t})$ has the almost-flow property and use the above lemma to extend this result to $\mu_{s,t}$. The next lemma shows the almost-flow property of $\cF(\XX_{s,t})$:

\begin{lem}\label{lem:technical4}
    Let $\phi\in C^\infty(M)$ and $x\in M$. Let $K\subset M$ be a compact set and $\abs{t-s}\le T(K)$. Further let $s\le u\le t$ and assume that $\mu_{s,u} x\in K$. It then holds that
    \begin{equation}\label{ineq:almostFlowFX}
        \abs{(\cF(\XX_{s,u})\circ\cF(\XX_{u,t}) - \cF(\XX_{s,t}))\phi(x)} \le C(\abs{t-s}^{(N+1)\alpha}\,.
    \end{equation}
    for some constant $C$ depending on $K,\phi,\cF$ and $\norm{\XX}$. If $\phi$ is additionally an diffeomorphism mapping some open set $\tilde U\subset M$ into $U\subset \RR^d$, we have that
    \begin{equation}\label{ineq:joinedalmostFlowFX}
        \abs{(\cF(\XX_{s,u})\circ\cF(\XX_{u,t}) - \cF(\XX_{s,t})) (\phi(x) - \phi(y))} \le C(K,\phi)\abs{t-s}^{(N+1)\alpha}\abs{\phi(x)-\phi(y)}\,.
    \end{equation}
\end{lem}

\begin{proof}
    We again begin by proving \eqref{ineq:almostFlowFX} first. Let $\star_N$ be the truncated product in $\cH^{N*}$. It holds that
    \begin{align*}
        \XX_{s,u}\star\XX_{u,t} &= \XX_{s,u} \star_N \XX_{u,t} + \underbrace{\sum_{k+l > N} \XX_{s,u}^k\star \XX_{u,t}^l}_{R_N}\\
        &= \XX_{s,t}+R_N\,.
    \end{align*}
    By choosing a basis $\{h_{m}^k~\vert~m=1\dots, M_k\}$ for each $\cH^{(k)}$, we get that 
    \begin{equation*}
        R_N = \sum_{h_m^k, h_q^p, k+p>N} \scalar{\XX_{s,u},h_m^k}\scalar{\XX_{u,t},h_q^p} h_m^k\star h_q^p\,.
    \end{equation*}
    It follows, that
    \begin{align}
    \begin{split}\label{ineq:FRN}
        \abs{\cF(R_N)\phi(x)} &\le \sum_{h_m^k, h_q^p,k+p>N} \abs{\scalar{\XX_{s,u},h_m^k}\scalar{\XX_{u,t},h_q^p}} \abs{\cF(h_m^k\star h_q^p)\phi(x)}\\
        &\le C \sum_{N< k+p\le 2N}\abs{u-s}^{k\alpha}\abs{t-u}^{p\alpha}\\
        &\le C \abs{t-s}^{(N+1)\alpha}\,,
    \end{split}
    \end{align}
    where we again allow $C$ to change between lines. Thus, we use the fact that $\cF$ is a pseudo bialgebra map to calculate
    \begin{align*}
        \abs{\cF(\XX_{s,u})\circ\cF(\XX_{u,t})\phi(x) - \cF(\XX_{s,t}) \phi(x))} &= \abs{\cF(\XX_{s,u}\star\XX_{u,t})\phi(x)- \cF(\XX_{s,t}) \phi(x)} \\
        &= \abs{\cF(R_N)\phi(x)}\le C \abs{t-s}^{(N+1)\alpha}\,.
    \end{align*}
    \eqref{ineq:joinedalmostFlowFX} follows analogously, by using Lemma \ref{lem:DiffOpLip} in \eqref{ineq:FRN} to get $\abs{\cF(h_m^k\star h_q^p)(\phi(x)-\phi(y))}\le C\abs{\phi(x)-\phi(y)}$.
\end{proof}

\noindent With these tools at hand, we can now prove our main result:

\begin{proof}[Proof of Theorem \ref{theo:main1}]
    Fix  a smooth coordinate function $\phi:M\supset\tilde U\to U\subset \RR^d$ as well as a compact set $K\subset U$. We need to show the Lipschitz and Hölder continuity and the almost-flow property of $\mu$. Let us start with the Lipschitz property: We set $\cF(\LL_{s,t})^\phi$ to be the vector field $\cF(\LL_{s,t})$ on $U$ given by $\cF(\LL_{s,t})\phi$. We get the integral identity
    \begin{equation*}
        Z_u = Z_0 + \int_0^u \cF(\LL_{s,t})^\phi(Z_r)dr\,.
    \end{equation*}
    Thus, it follows that if $Z$ has starting value $x\in K$ and $Z'$ has starting value $y\in K$:
    \begin{equation*}
        \abs{Z_u-Z'_u} \le \abs{x-y}+\int_0^u \norm{\cF(\LL_{s,t})^\phi}_{Lip, U}\abs{Z_r-Z'_r} dr\,.
    \end{equation*}
    So by Gronwall's inequality, we conclude
    \begin{equation*}
        \abs{Z_1-Z'_1} \le e^{\norm{\cF(\LL_{s,t})^\phi}_{Lip, U}} \abs{x-y}\,.
    \end{equation*}
    It follows that $\mu$ expressed in the coordinate chart $\phi$, given by $\mu_{s,t}^\phi := \phi\circ\mu_{s,t}\circ\phi^{-1}$, is Lipschitz with Lipschitz-constant $L(s,t) = e^{\norm{\cF(\LL_{s,t})^\phi}_{Lip, U}}$ which fulfills $L(K,s,s) = 1 = \lim_{\abs{t-s}\to 0}L(K,s,t)$. For the Hölder continuity, it suffices to recall that $\norm{\cF(\LL_{s,t})^\phi}_{\infty, U} \le C\abs{t-s}^\alpha$ to calculate
    \begin{equation*}
        \abs{\mu^\phi_{s,t} x - x} = \abs{\int_0^1 \cF(\LL_{s,t})^\phi(Z_r) dr}\le C\abs{t-s}^\alpha\,.
    \end{equation*}
    For the almost-flow property, let $m = \phi^{-1}(x)\in M$. We write
    \begin{align*}
        \mu^\phi_{u,t}\circ\mu^\phi_{s,u} x &= \cF(\XX_{u,t})\phi(\mu_{s,u} m) +\epsilon_{u,t}(\mu_{s,u} m)\\
        &= \cF(\XX_{s,u})\circ\cF(\XX_{u,t})\phi(m) + \epsilon_{u,t}(\mu_{s,u} m)+\epsilon_{s,u}(m)
    \end{align*}
    by the Taylor-decomposition, where $\abs{\epsilon_{s,u}(m)}\le C\abs{s-u}^{1+\epsilon}$ and $\abs{\epsilon_{u,t}(\mu_{s,u} m)}\le C\abs{u-t}^{1+\epsilon}$ holds, for $\epsilon = (N+1)\alpha$-1. Lemma \ref{lem:technical4} together with the Taylor-decomposition of $\mu_{s,t}^\phi$ gives us
    \begin{equation*}
        \abs{\mu_{u,t}^\phi\circ\mu_{s,u}^\phi x- \mu_{s,t}^\phi x} \le C\abs{t-s}^{1+\epsilon}\,.
    \end{equation*}
    For the joined Lipschitz-almost-flow property, we use the second part of the Taylor decomposition to calculate
    \begin{align*}
        \mu^\phi_{u,t}\circ\mu^\phi_{s,u} (x - y) = \cF(\XX_{s,u})\circ\cF(\XX_{u,t})(\phi(m) - \phi(n)) + \epsilon_{u,t}(\mu_{s,u} m,\mu_{s,u} n)+\epsilon_{s,u}(m,n)\,,
    \end{align*}
    where $n = \phi^{-1}(y)$. We get that 
    \[
        \abs{\epsilon_{u,t}(\mu_{s,u} m,\mu_{s,u} n)}\le C \abs{t-u}^{1+\epsilon}\abs{\mu^\phi_{s,u} x-\mu^\phi_{s,u} y} \le C \abs{t-u}^{1+\epsilon}\abs{x-y}\,,
    \]
    and $\abs{\epsilon_{s,u}(m,n)}\le C\abs{u-s}^{1+\epsilon}\abs{x-y}$. The joined Lipschitz-almost-flow property follows again from Lemma \ref{lem:technical4}.
\end{proof}

\noindent Thanks to Proposition \ref{prop:sewingOnM}, we see that $\mu$ generates a flow on $M$, which we denote by $\eta$. Our Taylor decomposition of $\mu$ immediately gives us a Davie's formula for $\eta$:

\begin{cor}[Davie's formula]
    Let $K\subset M$ be compact and $\phi\in C^\infty(M)$. It holds that for each $m\in M$ and $\abs{t-s}\le T(K)$:
    \begin{equation}\label{ineq:Davie}
        \abs{(\eta_{s,t*}-\cF(\XX_{s,t}))\phi(m)} \le C \abs{t-s}^{(N+1)\alpha}\,,
    \end{equation}
    where $C$ depends on $K, \phi, \cF$ and $\norm{\XX}$. Furthermore, $\eta_{s,t} x$ is the unique $\alpha$-Hölder continuous flow in $M$ (up to time $\abs{s-t} \le T(K)$) with this property.
\end{cor}

\begin{proof}
    The uniqueness immediately follows from the fact that every over $\alpha$-Hölder continuous flow $\tilde\eta$ in $M$ fulfilling \eqref{ineq:Davie} would be a flow with
    \begin{equation*}
        \abs{(\mu_{s,t*}-\tilde\eta_{s,t*})\phi(m)}\le C\abs{t-s}^{(N+1)\alpha}\,,
    \end{equation*}
    which contradicts the sewing lemma by Remark \ref{rem:UNiquenessAddition}. To see that the above inequality holds, observe that
    \begin{align*}
        \abs{(\eta_{s,t*}-\cF(\XX_{s,t}))\phi(m)} &\le \abs{(\eta_{s,t*}- \mu_{s,t*})\phi(m)}+\abs{(\mu_{s,t*} - \cF(\XX_{s,t}))\phi(m)}\\ 
        &\le C \abs{t-s}^{(N+1)\alpha}\,.
    \end{align*}
\end{proof}

\section{Constructing elementary differentials}\label{sec:elementaryDiff}

\noindent The goal of this section is to provide an overview of the classical ways to construct elementary differentials and show that they indeed generate pseudo bialgebra maps. We also present a concrete form of $\cF$ on the Hopf algebras over trees $\cH_{GL}$ and $\cH^g_{MKW}$. This will especially allow us to generalize the construction of $\cF$ on $\cH_{MKW}^g$ to any manifold equipped with a connection, whereas the standard construction requires a flat connection with constant torsion.

The main idea behind the constructions on $T(\RR^n), \cH_{GL}$ and $\cH_{MKW}^*$ is that all of these algebras (or the Lie algebras of the primitive elements) are free in some sense. More precisely, 

\begin{itemize}
    \item The tensor algebra $T(V)$ over $V=\myspan\{e_1,\dots, e_n\}$ is the free associative algebra generated by $e_1,\dots, e_n$ for linear independent $e_1,\dots,e_n$.
    \item The primitive elements of the Grossman-Larson algebra form the free pre-Lie algebra.
    \item The primitive elements of $\cH_{MKW}^g$ form the free post-Lie algebra.
\end{itemize}

\noindent Thus, the general strategy to generate $\cF$ is as follows: For the generator set $\{e_1,\dots,e_n\}$ of $T(\RR^n)$ or $\{\tomato1,\dots,\tomato n\}$ in the branched case, we set $\cF(e_i) = V_i = \cF(\tomato i)$, where $V_i$ is given by \eqref{eq:RDE}. Then, the universal properties of the respective algebras will generate maps $\cF$ on $T(\RR^n)$ as well as $\cP(\cH_{GL})$ and $\cP(\cH_{MKW}^*)$.

\subsection{Tensor algebra}\label{subsec:elementaryDiffTensor}

The elementary differentials of the tensor algebra are well known, but we still want to spend a page to show that it does fit into our theory: Considering the RDE \eqref{eq:RDE} for a smooth path $X$ lifted to a rough path $\XX$, a simple Taylor expansion of $\phi(Y_t)$ for some $\phi\in C^\infty(M)$ suggest (\cite{FV10}, page 128):
\begin{equation*}
    \phi(Y_t) \approx \sum_{\abs{w}\le N} V_{w_1}\circ\dots\circ V_{w_n}\phi(Y_s)\XX^w_{s,t}\,,
\end{equation*}
implying that we should expect $\cF(w) = V_{w_1}\circ\dots\circ V_{w_n}$ for any word $w = (w_1,\dots, w_n)$. This is precisely the map one gets from the freeness property of $T(\RR^n)$: $T(\RR^n)$ is the free associative algebra (\cite{sweedler69}, \cite{freeLieAlgebras}), meaning that it has the following universal property: Every linear map $F:\RR^n\to \cA$ for some associative algebra $(\cA,\circ,\one)$ uniquely extends to an associative algebra map $\cF:T(\RR^n)\to \cA$. Note that $(\cD,\circ,\Id)$ is an associative algebra, and we assumed that $\cF(e_i) = V_i$ for $i=1,\dots,n$, which gives us a linear map $\cF:\RR^n\to\cV\subset\cD$. Thus, we can extend it to an associative algebra map $\cF:T(\RR^n)\to \cD$. Furthermore, since every word $w = (w_1\dots w_n) = w_1\cdot w_2 \cdot\ldots\cdot w_n$, it immediately follows that
\begin{equation*}
    \cF(w) = \cF(w_1)\circ\dots\circ\cF(w_n) = V_{w_1}\circ\dots\circ V_{w_n} 
\end{equation*}

\begin{thm}
    $\cF$ is a pseudo bialgebra map.
\end{thm}

\begin{proof}
    By its definition, $\cF:T(\RR^n)\to\cD$ is an algebra map. Thus, we only need to check that
    \begin{equation*}
        \cF(\Delta x)(\phi\otimes\psi) = \cF(x)(\phi\cdot\psi)\,.
    \end{equation*}
    Note that for letters $i=1\dots, n$, the Leipniz rule together with the fact that $i$ is primitive gives us
    \begin{align*}
        \cF(i)(\phi\cdot\psi) &= V_i(\phi\cdot\psi) \\
        &= (V_i\otimes \Id + \Id\otimes V_i)(\phi\otimes \psi) \\
        &= \cF(\Delta i)(\phi\otimes\psi)\,.
    \end{align*}
    For general words $x= (x_1\dots x_m)$, we set $\bar x = (x_2\dots x_m)$ and inductively conclude that
    \begin{align*}
        \cF(\Delta x)(\phi\otimes\psi) &= \cF((x_1\otimes \one+ \one \otimes x_1)\Delta\bar x)(\phi\otimes\psi)\\
        &= \sum_{(\bar x)}\cF(x_1\bar x^{(1)}\otimes \bar x^{(2)} + \bar x^{(1)}\otimes x_1\bar x^{(2)})(\phi\otimes \psi) \\
        &= \sum_{(\bar x)} \left[\cF(x_1)\circ\cF(\bar x^{(1)})\phi\cdot\cF(\bar x^{(2)})\psi + \cF(\bar x^{(1)})\phi \cdot\cF(x_1)\circ\cF(\bar x^{(2)}) \psi\right]\\
        &= \cF(x_1)\left[\cF(\Delta\bar x)(\phi\otimes\psi)\right]\\
        &= \cF(x_1)\circ\cF(\bar x)(\phi\cdot\psi)\\
        &= \cF(x)(\phi\cdot\psi)\,.
    \end{align*}
\end{proof}

\begin{rem}
    The set of primitive elements of $T(V)$ is given by the free Lie algebra of $V$ (\cite{FV10},\cite{freeLieAlgebras}). Thus, one can also construct $\cF$ as the unique Lie algebra map mapping $\cP\to\cV$, extending $\cF(i) = V_i$. One can easily check that the pseudo bialgebra map generated by this Lie map is the same as we described above.
\end{rem}

\subsection{Grossmann-Larson algebra}

Before we start with the proper construction of $\cF$ over $\cH_{GL}$, let us recall the elementary differentials for the Grossman-Larson algebra: Taking again a smooth path $X$ in \eqref{eq:RDE} and Taylor expending $\phi(Y)$, we see that (\cite{HK15} calculation (1.6))
\begin{equation*}
    \phi(Y_t) \approx \phi(Y_s) + V_i \phi(Y_s) X^i_{s,t} + \frac 12V_i^\alpha V_j^\beta \partial_{\alpha,\beta} \phi(Y_s) X^i_{s,t} X^j_{s,t} + V_i^\alpha\partial_\alpha V_j^\beta \partial_\beta\phi(Y_s) \int_{s}^t X^i_{s,r} dX^j_{r}+\dots
\end{equation*}
for any smooth function $\phi$. If we compare that to our expected series $\sum_{\abs{\tau}\le N} \frac 1{sg(\tau)}\cF(\tau)\phi(Y_s) \XX^\tau_{s,t}$ (recall that $\tau$ is not the dual element of $\tau\in\cH_{CK}$, but it is connected via $\scalar{\tau,\tau} = sg(\tau)$. Thus, we need to add a factor $\frac 1{sg(\tau)}$ to the formula.), we see a pattern arise for $\cF$: For $n$ vector fields $U_i^\alpha\partial_\alpha, i=1\dots,n$, we say that we apply all $U_i$ to $\phi$ by applying the differential operator $U_1^{\alpha_1}\dots U_n^{\alpha_n}\partial_{\alpha_1,\dots\alpha_n}$ to $\phi$. For a vector field $U = U^\beta\partial_\beta$, we say that $U_1,\dots, U_n$ applied to U is the vector field $(U_1^{\alpha_1}\dots U_n^{\alpha_n})\partial_{\alpha_1,\dots\alpha_n} U^\beta\partial_\beta$. Then for any tree $\tau$, $\cF(\tau)$ is the vector field one gets by setting $\cF(\tomato i) = V_i^\alpha\partial_\alpha$ and inductively applying all children to their parent. For example, 
\begin{equation*}
    \cF(\righttree 1234) = (V_4^\alpha\partial_\alpha V_3^\beta) V_2^\gamma\partial_{\beta\gamma}V_1^\delta\partial_\delta\,.
\end{equation*}
Finally, one gets $\cF(\tau_1\dots\tau_n)\phi$ for any forest $\tau_1\dots\tau_n$ by applying all vector fields $\cF(\tau_i),i=1\dots,n$ to $\phi$.

If our vector fields do not live in flat space but on the manifold $M$, we do not have partial derivatives but only covariant derivatives. Replacing the partial derivatives with covariant derivatives gives us the following recursive formula for $\cF$: As before, we set $\cF(\tomato i) = V_i$. Inductively, we then assign to the tree $[\tau_1,\dots,\tau_n]_i$ the $n$-th covariant derivative of $V_i$ in the directions $\cF(\tau_1),\dots,\cF(\tau_n)$:
\begin{equation*}
    \cF([\tau_1,\dots,\tau_n]_i) = \nabla^n V_i(\cF(\tau_1),\dots,\cF(\tau_n))\,.
\end{equation*}
Finally, we set $\cF(\tau_1,\dots,\tau_n)\phi = \nabla^n\phi(\cF(\tau_1),\dots,\cF(\tau_n))$.

\begin{rem}
    Note that $\cF$ is in general not well-defined (see \cite{geometricViewRP}, Remark 4.35): In general, $\nabla^2\phi(V_1,V_2)\neq \nabla^2\phi(V_2,V_1)$. However, the Grossman-Larson algebra does not differentiate between the order of trees in a forest or the order of children of a vertex in a tree. Thus, we will need an additional assumption to construct $\cF$, namely that $\nabla$ is flat and torsion-free. In this case, $\nabla^n\phi(V_1,\dots, V_n)$ does not depend on the order of $V_1,\dots, V_n$. We discuss this map in more detail in Section \ref{subsec:MyMap}.
\end{rem}

\noindent In the flat, torsion-free case it actually suffices to require
\begin{itemize}
    \item $\cF(\tomato i) = V_i$  for $i=1,\dots, n$ and
    \item $\cF(\tau\curvearrowright\sigma) = \nabla_{\cF(\tau)}\cF(\sigma)$ for all trees $\tau,\sigma\in\cUF$
\end{itemize}
to construct $\cF$ on all trees (see Thm. \ref{thm:flatCase}), where $\curvearrowright$ is the crafting product defined in Section \ref{sec:HopfAlgebrasTrees}. This approach has a deep connection to the algebraic structure of $\myspan(\cUT)\subset\cH_{GL}$: As shown in for example \cite{ayadi} or \cite{chapoton}, $(\myspan(\cUT),\curvearrowright)$ is the \emph{free pre-Lie algebra} generated by $\{\tomato 1,\dots\tomato n\}$. Pre Lie algebras are an important concept for the analysis of Butcher series \cite{MKL13}, and are defined as follows:

\begin{defn}
    A pre-Lie algebra is a vector field $\cA$ equipped with a (not necessarily associative) product $\triangleright$, which fulfills the pre-Lie identity:
    \begin{equation*}
        (x\triangleright y)\triangleright z - x\triangleright(y\triangleright z) = (y\triangleright x)\triangleright z - y\triangleright(x\triangleright z)
    \end{equation*}
    for all $x,y,z\in \cA$.
\end{defn}

\noindent Pre-Lie algebras are \emph{Lie admissible}, meaning that their commutator $[x,y] = x\triangleright y- y\triangleright x$ forms a Lie bracket on $\cA$. As mentioned above, the free pre-Lie algebra generated by the set $\{\tomato1,\dots,\tomato n\}$ is well known to be the span of (unordered) trees $\myspan(\cUT)$ equipped with the grafting product 
\begin{equation*}
    \tau\triangleright\sigma = \tau \curvearrowright \sigma\,,
\end{equation*}
where $\curvearrowright$ is defined in Section \ref{sec:HopfAlgebrasTrees}. Recall that $\myspan(\cUT) = \cP(\cH_{GL})$, which immediately gives us the free structure of the primitives of $\cH_{GL}$.

We want to connect $\curvearrowright$ with the the non-associative product $V\triangleright U = \nabla_V U$ on $\cV$. Assume for the moment, that $(\cV,\triangleright)$ forms a pre-Lie algebra, such that the classical Lie bracket on $\cV$ is equal to the commutator $[U,V] = U\triangleright V-V\triangleright U$. By assumption, we have a linear map $\cF:\myspan\{\tomato 1\dots\tomato n\}\to\cV$, $\cF(\tomato i) = V_i$. The free property of $\cP(\cH_{GL})$ then automatically extends this to a unique pre-Lie map $\cF:\cP(\cH_{GL})\to\cV$. Note that for all trees $\tau,\sigma\in\cUT$, we have
\begin{equation*}
    \tau\curvearrowright\sigma = \tau\star\sigma-\tau\sigma\,.
\end{equation*}
since $\tau\sigma = \sigma\tau$, we get that $[\tau,\sigma] = \tau\star\sigma-\sigma\star\tau = \tau\curvearrowright\sigma-\sigma\curvearrowright\tau$. Thus, the pre-Lie property of $\cF$ then immediately gives us
\begin{equation*}
    \cF([\tau,\sigma]) = \cF(\tau)\triangleright\cF(\sigma)- \cF(\sigma)\triangleright\cF(\tau) = [\cF(\tau),\cF(\sigma)]\,,
\end{equation*}
turning $\cF$ into a Lie algebra map. We can summarize this result in the following proposition:

\begin{prop}\label{prop:PreLieMap}
    Let $\nabla$ be flat and torsion-free. Then the pre-Lie map $\cF:\cP(\cH_{GL})\to\cV$ generated by $\cF(\tomato i) = V_i$, $i=1\dots, n$, is a Lie-algebra map. Thus, it generates a pseudo bialgebra map on $\cH_{GL}$.
\end{prop}

\begin{proof}
    The only thing left to show is that $(\cV,\triangleright)$ is a pre-Lie algebra with $[V,U] = V\triangleright U - U \triangleright V$ for all $U,V\in\cV$. The fact that $(\cV,\triangleright)$ is pre-Lie if and only if $(M,\nabla)$ is flat and torsion-free has been proven in Thm 2.23 in \cite{MKSV20}. Writing down the torsion tensor gives us for all $V,U\in\cV$
    \begin{equation*}
        0 = T(V,U) = V\triangleright U - U\triangleright V - [V,U]\,,
    \end{equation*}
    which gives us the second property.
\end{proof}

\subsection{The Munthe-Kaas Wright algebra}\label{subsec:elementarryDiffMKW}

Note that the map $\cF$, given by $\cF(\tomato i) = V_i, \cF([\tau_1,\dots,\tau_n]_i) = \nabla^n V_i(\cF(\tau_1),\dots,\cF(\tau_n))$ is well-defined as a map from $\cF:\cH^g_{MKW}\to\cD$, as the forests are now ordered. This map naturally arises on $\cH_{MKW}^g$, if we loosen the assumptions on $(M,\nabla)$, which gives rise to the structure of \emph{post-Lie algebras}. Post-Lie algebras were independently introduced in \cite{MKL13} and \cite{VALLETTE2007699} (for an introduction, we recommend \cite{CKMK19}) and are based on the observation, that without the assumption of flatness and torsion-freeness, a single operation $\triangleright$ no longer suffices to generate all the relevant vector fields from $V_1,\dots, V_n$. The solution to this is to introduce a second operation, which is given by another Lie bracket $\llbracket\cdot,\cdot\rrbracket$. In this Section, we will show the construction of $\cF$ with the use of the post-Lie structure of the primitive elements $\cP(\cH^g_{MKW})$ (mainly following \cite{curry20}), which will still require rather strong assumptions on $(M,\nabla)$. In Section \ref{subsec:MyMap}, we show that this leads to the same map as directly constructing $\cF$ on each forest as above and that the direct approach allows us to loosen the assumptions on $(M,\nabla)$ even further.

Let us first recall the definition of a post-Lie algebra:

\begin{defn}
    A post Lie algebra $(\mathfrak g,\brackets{\cdot,\cdot},\triangleright)$ is a Lie algebra $(\mathfrak g,\brackets{\cdot,\cdot})$ equipped with a product $\triangleright:\mathfrak g\otimes\mathfrak g\to \mathfrak g$, such that for all $x,y,z\in\mathfrak g$
    \begin{align*}
        x\triangleright\brackets{y,z} &= \brackets{x\triangleright y, z} + \brackets{y,x\triangleright z}\,. \\
        \brackets{x,y}\triangleright z &= a_\triangleright(x,y,z)-a_\triangleright(y,x,z)
    \end{align*}
    where $a_\triangleright(x,y,z) := x\triangleright (y\triangleright z) - (x\triangleright y)\triangleright z$ is the associator.
\end{defn}

\noindent As shown in \cite{MKW08}, \cite{MKL13} the free post-Lie algebra is given as the free Lie algebra over the space of trees $\cOT$, equipped with the left crafting product $\curvearrowright_l$ and the Lie bracket $\brackets{\tau,\sigma} = \tau\sigma-\sigma\tau$ being the commutator with respect to the tensor product. A thorough discussion of the free post-Lie algebra can be found here: \cite{postLie}. Note that $\cH^g_{MKW}$ as a space is given by the tensor algebra $T(\cOT)$ with the coproduct being the coshuffle, which immediately gives us that the primitive elements of $\cH_{MKW}^g$ are given by the free Lie algebra over $\cOT$ with respect to the Lie bracket $\brackets{\cdot,\cdot}$. Thus, $(\cP(\cH_{MKW}^g),\brackets{\cdot,\cdot},\curvearrowright_l)$ is the free post-Lie algebra.

It remains to find the post-Lie structure on $\cV$: We equip $\cV$ with the product $V\triangleright U = \nabla_V U$ and the operation $\brackets{U,V} := -T(U,V)$, where $T$ is the torsion (note that this is in general not a Lie bracket). Thm 2.23 from \cite{MKSV20} then states that $(\cV,\brackets{\cdot,\cdot},\triangleright)$ is a post-Lie algebra if and only if $(M,\nabla)$ is flat and $\brackets{\cdot,\cdot}$ is a Lie bracket, which Munthe-Kaas et al call ``$(M,\nabla)$ is parallel''. \cite{MKL13} shows that a sufficient condition such that $\brackets{\cdot,\cdot}$ is a Lie bracket is, that $(M,\nabla)$ is flat and has \emph{constant torsion}, that is, its covariant derivative vanishes: $\nabla T = 0$. This is the condition we will use in the following to ensure that $(\cV,\brackets{\cdot,\cdot},\triangleright)$ forms a post-Lie algebra.

Since $\cV$ forms a post-Lie algebra, the universal property immediately gives us a unique post-Lie map $\cF:\cP(\cH_{MKW}^g)\to\cV$ with $\cF(\tomato i) = V_i$ for $i=1,\dots, n$. And while this immediately implies that $\cF$ is a Lie homomorphism with respect to the Lie brackets $\brackets{\cdot,\cdot}$ on $\cP(\cH_{MKW}^g)$ and $\cV$ respectively, we will only get a pseudo bialgebra map if it is a Lie homomorphism with respect to the Lie bracket $[\tau,\sigma] = \tau\star\sigma-\sigma\star\tau$ on $\cP(\cH_{MKW}^g)$ and $[U,V] = U\circ V-V\circ U$ on $\cV$. This follows easily, as the main result of this section shows:

\begin{prop}\label{prop:postLie}
    Let $(M,\nabla)$ be flat with constant torsion. Then $\cF:(\cP(\cH_{MKW}^g),[\cdot,\cdot])\to(\cV,[\cdot,\cdot])$ is a Lie map and thus generates a pseudo bialgebra map $\cF:\cH_{MKW}^g\to \cD$.
\end{prop}

\noindent Before we show this proposition, let us present a helpful Lemma:

\begin{lem}\label{lem:postLieTechnical}
    Let $f\in\cP(\cH_{MKW}^g)$ and $g\in\cOF$. Then
    \begin{equation}\label{eq:StarDecomp}
        f\star g = fg + f\curvearrowright_l g\,.
    \end{equation}
\end{lem}

\begin{proof}
    We show the claim by induction, where we use that $\cP(\cH_{MKW}^g)$ gets generated by $\cOT$ and the Lie bracket $\brackets{\cdot,\cdot}$. It is clear by the definition of $\star$ that \eqref{eq:StarDecomp} holds, whenever $f\in\cOT$ is a tree. Thus, let us assume that the claim holds for two $f,g\in\cP(\cH_{MKW}^g)$, i.e. $f\star h = fh +f\curvearrowright_l h$ and $g\star h = gh + g\curvearrowright_l h$ for all $h\in\cOF$. We only need to show that it holds for $\brackets{f,g} = fg-gf$. That is, we show that for any $h\in\cOF$, we have $\brackets{f,g}\star h = \brackets{f,g}h + \brackets{f,g}\curvearrowright_l h$. To do so, we introduce a bit more notation: Let $g\in\cOF$ with a subforest $h\subset g$. When we denote by $f\curvearrowright_l^h g$ the sum over all possibilities, to let the trees of $f$ grow out of nodes of $h$ as the left-most child of said node. For example, for $g = \righttree 1234$ and $h = \ladderTwo 13$, we have
    \begin{equation*}
        \tomato{5}\curvearrowright_l^h g = \begin{tree}
            \node[label=right:{\tiny 1}]{}[grow'=up, level distance = 2ex, sibling distance = 2.5ex]
            child{node [label=right:{\tiny 5}] {}}
            child{node [label=right:{\tiny 2}] {}}
            child{node [label=right:{\tiny 3}] {}
                child{node [label=right:{\tiny 4}] {}}
                };
        \end{tree} + \begin{tree}
            \node[label=right:{\tiny 1}]{}[grow'=up, level distance = 2ex, sibling distance = 2.5ex]
            child{node [label=right:{\tiny 2}] {}}
            child{node [label=right:{\tiny 3}] {}
                child{node [label=right:{\tiny 5}] {}}
                child{node [label=right:{\tiny 4}] {}}
                };
            \end{tree}
    \end{equation*}
    Analogously, we introduce the product $f\otimes^h g$ which puts the forest $f$ left to the leftmost tree which contains at least one node from $h$. For example, for $g = \tomato 5\cherry123$ and $h= \ladderTwo13$, we have
    \begin{equation*}
        f\otimes^h g = \tomato 5 f \righttree1234\,.
    \end{equation*}
    The last notation we need for the proof is the following: For a forest $f=f_1\dots f_n$ where $f_1,\dots,f_n\in\cOT$ and $J\subset\{1,\dots,n\}$, we denote with $f_J := f_{j_1}\dots f_{j_m}$, where $J= \{j_1<\dots< j_m\}$ holds. One can now check, that for all forests $f,g,h$, with $f=f_1\dots f_n$ and $g = g_1\dots g_m$ we have that
    \begin{align*}
        (fg) \star h &= \sum_{\substack{J\subset\{1,\dots,n\} \\ G \subset\{1,\dots,m\}}} f_J g_G (f_{J^c}\curvearrowright_l^h (g_{G^c}\curvearrowright_l h))\\
        &=\sum_{J\subset\{1,\dots,n\}} f_J(f_{J^c}\curvearrowright_l^h (g\star h))\,.
    \end{align*}
    Assume that \eqref{eq:StarDecomp} holds for $f,g\in \cP(\cH_{MKW}^g)$ and $h\in\cOF$. Then the above becomes
    \begin{align*}
        (fg)\star h &= \sum_{J\subset\{1,\dots,n\}} f_J(f_{J^c}\curvearrowright_l^h (gh + g\curvearrowright_l h)) \\
        &= g\otimes^h(f\star h) + g\curvearrowright_l^h (f\star h)\\
        &= fg h + g(f\curvearrowright_l h) + f(g\curvearrowright_l h) + (fg)\curvearrowright_l h\,.
    \end{align*}
    We conclude that
    \begin{align*}
        \brackets{f,g}\star h &= fgh-gfh +(fg)\curvearrowright h - (gf)\curvearrowright h \\
        &= \brackets{f,g} h + \brackets{f,g}\curvearrowright h\,,
    \end{align*}
    which shows \eqref{eq:StarDecomp} for $\brackets{f,g}$.
\end{proof}

\begin{proof}[Proof of Proposition \ref{prop:postLie}]
    By the definition of the torsion, we have
    \begin{equation*}
        [U,V] = U\triangleright V - V\triangleright U -T(U,V) = U\triangleright V - V\triangleright U +\brackets{U,V}
    \end{equation*}
    for all $U,V\in\cV$. On the other side, Lemma \ref{lem:postLieTechnical}gives us for all $f,g\in \cP(\cH_{MKW}^g)$
    \begin{equation*}
        [f,g] = f\star g-g\star f = g\curvearrowright_l f-f\curvearrowright_l g + \brackets{f,g}\,.
    \end{equation*}
    It now follows from the post-Lie properties of $\cF$ that
    \begin{align*}
        \cF([f,g]) &= \cF(f)\triangleright\cF(g)-\cF(g)\triangleright\cF(f)+\brackets{\cF(f),\cF(g)}\\
        &=[\cF(f),\cF(g)]\,,
    \end{align*}
    showing that $\cF$ is a Lie map with respect to the Lie brackets $[\cdot,\cdot]$ on both spaces.
\end{proof}

\subsection{A pseudo bialgebra map on $\cH_{MKW}^*$}\label{subsec:MyMap}

In this section, we take a closer look at the linear map $\cF:\cH_{MKW}^g\to\cD$ constructed in the previous section via
\begin{align}
\begin{split}\label{myMap}
    \cF(\tomato i)\psi &= V_i\psi \\
    \cF(\tau_1,\dots,\tau_k)\psi &= \nabla^k \psi (\cF(\tau_1),\dots,\cF(\tau_k)) \\
    \cF([\tau_1,\dots,\tau_k]_i) &= \nabla^k V_i (\cF(\tau_1),\dots,\cF(\tau_k))\,,
\end{split}
\end{align}
for $i=1\dots, n$ and $\tau_1,\dots,\tau_k\in\cOT$ being ordered trees. To our knowledge, this map was first explicitly mentioned in \cite{geometricViewRP}, where it is only briefly discussed for non-planarly branched rough paths. We show that as long as $(M,\nabla)$ is a manifold with a smooth connection, $\cF$ is a pseudo bialgebra map, thus getting rid of the assumption that $\nabla$ is flat and has constant torsion. We further show that if $\nabla$ is flat and has constant torsion, $\cF$ restricted to $\cP(\cH^g_{MKW})$ is a post-Lie algebra map, making it a generalization to the previously discussed post-Lie map.

While $\cF$ is in general ill-defined on $\cH_{GL}$, it is well-defined for a flat and torsion-free connection $\nabla$. In this case, we also show that $\cF$ restricted to $\cP(\cH_{GL})$ is a pre-Lie map, and thus becomes the map discussed in Proposition \ref{prop:PreLieMap}. Finally, we compare it to the map constructed in \cite{emilio22}, where the authors managed to solve non-geometric RDEs on manifolds for level $N=2$ for any connection $\nabla$.

This leads us to the main result of this section:

\begin{thm}\label{theo:main2}
Assume that $(M,\nabla)$ is a manifold equipped with a smooth connection. Then the map $\cF:\cH_{MKW}^g\to\cD$ constructed above is a pseudo bialgebra map. Furthermore, given two non-empty trees $\tau,\sigma\in\cOT$, $\cF$ fulfills:

\begin{equation}\label{graftCov}
    \cF(\tau\curvearrowright_l \sigma) = \cF(\tau)\triangleright\cF(\sigma)\,.
\end{equation}
\end{thm}

\noindent To ease notation in the following proofs, we use the notation $V_f := \cF(f)$ for all forests $f\in\cOF$.

\begin{proof}
We start by showing the relations of $V_{\tau\star\sigma} = V_\tau\circ V_\sigma$ and $V_{\tau\curvearrowright_l\sigma} = \nabla_{V_\tau} V_\sigma$ for two trees $\tau,\sigma\in\cOT$. Since the identity operator is not a vector field, $\nabla_{V_{\tau}} V_{\sigma}$ is not a well-defined object for empty trees. On the other hand, $V_\tau\circ V_\sigma = V_{\tau\star\sigma}$ is obvious, if $\tau$ or $\sigma$ is empty.

We show the claim by induction over the number of vertices in $\sigma$. For the induction start, let $\sigma = \tomato i$ for some $i\in\{1,\dots, n\}$. Then we immediately have
\begin{equation*}
    \nabla_{V_\tau} V_i = V_{[\tau]_i} = V_{\tau\curvearrowright_l\sigma}.
\end{equation*}
On the other hand, \eqref{covariant_der_general} gives us for any vector fields $U,W$ the formula $U\circ W = \nabla_U W + \nabla^2(\cdot)(U,W)$. This leads to:
\begin{equation*}
V_\tau\circ V_i = \nabla_{V_\tau} V_i + \nabla^2(\cdot)(V_\tau,V_i) = V_{[\tau]_i} + V_{\tau\tomato i} = V_{\tau\star\tomato i}\,,
\end{equation*}
showing the induction start. Now let $\sigma = [\sigma_1,\dots,\sigma_k]_i$ have $k\ge1$ children. We consider the tree $[\tau,\sigma_1,\dots,\sigma_k]_i$ and calculate
\begin{align*}
    V_{[\tau,\sigma_1,\dots,\sigma_k]_i} &= \nabla_{V_\tau}(\nabla^k V_i)(V_{\sigma_1},\dots, V_{\sigma_k}) \\
    &= V_\tau\circ V_\sigma - \nabla^2(\cdot)(V_\tau,V_\sigma)-\sum_{j=1}^k \nabla^k V_i(V_{\sigma_1},\dots,\nabla_{V_\tau}V_{\sigma_j},\dots,V_{\sigma_k})\,.
\end{align*}
Rearranging this expression and using the induction hypothesis on $\tau\curvearrowright_l \sigma_j$ leads to:
\begin{equation*}
    V_\tau\circ V_\sigma = V_{[\tau,\sigma_1,\dots,\sigma_k]_i+\tau\sigma+\sum_{j=1}^k [\sigma_1,\dots,\tau\curvearrowright_l\sigma_j,\dots,\sigma_k]_i} = V_{\tau\star\sigma}\,.
\end{equation*}
It now easily follows that
\begin{equation*}
    \nabla_{V_\tau}V_\sigma = V_\tau\circ V_\sigma -\nabla^2(\cdot)(V_\tau,V_\sigma) = V_{\tau\star\sigma-\tau\sigma} = V_{\tau\curvearrowright_l\sigma}\,.
\end{equation*}
We continue to proof that $V_{f_1}\circ V_{f_2} = V_{f_1\star f_2}$ holds for all forests $f_1,f_2\in \cOF$. To get started let $\tau$ be a tree and $f_2 = \sigma_1\dots\sigma_k$ be a forest. In that case, we get that for any smooth $\phi$:
\begin{align*}
    V_{\tau f_2}\phi &= \nabla_{V_\tau}(\nabla^n\phi)(V_{\sigma_1},\dots,V_{\sigma_k})\\
    &= V_\tau\circ V_{f_2} \phi - \sum_{i=1}^k\nabla^{k}\phi(V_{\sigma_1},\dots,\nabla_{V_\tau}V_{\sigma_i},\dots, V_{\sigma_k}) \\
    &= V_\tau\circ V_{f_2}\phi - V_{\sigma_1\dots\tau\curvearrowright_l\sigma_i\dots\sigma_k}\,.
\end{align*}
Rearranging the terms again leads to $V_\tau\circ V_{f_2} = V_{\tau\star f_2}$. Finally, let $f_1=\tau_1\dots\tau_m$ also be a forest. In that case,
\begin{align*}
    V_{f_1}\circ V_{f_2} \phi &= \nabla^m (V_{f_2}\phi)(V_{\tau_1},\dots, V_{\tau_m}) \\
    &=V_{\tau_1}\circ(\nabla^{m-1} V_{f_2} \phi)(V_{\tau_2},\dots,V_{\tau_m})-\sum_{i=2}^m \nabla^{m-1}(V_{f_2}\phi)(V_{\tau_2},\dots,\nabla_{V_{\tau_1}}V_{\tau_i},\dots,V_{\tau_m})\\
    &= V_{\tau_1\star(\tau_2\dots\tau_m)\star f_2 - \sum_{i=2}^m (\tau_2\dots(\tau_1\curvearrowright_l\tau_i)\dots\tau_m)\star f_2}\phi = V_{f_1\star f_2}\phi\,,
\end{align*}
where we again use an inductive argument over the number of trees in $f_1$. \\[.5\baselineskip]

\noindent All that remains to prove, is that for two smooth functions $\phi,\psi$ and any forest $f\in\cOF$:
\begin{equation*}
    V_{\Delta f}(\phi\otimes\psi) = V_{f}(\phi\cdot\psi)\,,
\end{equation*}
where $\Delta$ is the coshuffle, which we are proving by yet another induction over the number m of trees in $f=\tau_1\dots\tau_m$. For $m=1$, $V_f= V_{\tau_1}$ is just a vector field, so the result follows from $V_f(\psi\phi) = (Vf\psi)\phi+(\psi)V_f\psi = (V_f\otimes 1+ 1\otimes V_f)(\psi\otimes\phi)$.

Now consider $m\ge 2$. We use Sweedler's notation: $\Delta(\tau_2,\dots,\tau_m) = \sum_{(f)}f^{(1)}\otimes f^{(2)}$. Observe that $\sum_{i=2}^m\Delta(\tau_2\dots(\tau_1\curvearrowright_l\tau_i)\dots\tau_m) = \sum_{(\bar f)}(\tau_1\curvearrowright_l \bar f^{(1)}\otimes \bar f^{(2)}+\bar f^{(1)}\otimes\tau_1\curvearrowright_l \bar f^{(2)})$ holds, where $\bar f = \tau_2\dots\tau_m$. With this in mind, we see that the following holds:

\begin{align*}
    V_f(\phi\psi) &= V_{\tau_1}(\nabla^{m-1}(\phi\psi)(V_{\tau_2},\dots,V_{\tau_m}))-\sum_{i=2}^m\nabla^{m-1}(\phi\psi)(V_{\tau_2},\dots,\nabla_{V_{\tau_1}}V_{\tau_i},\dots,V_{\tau_m}) \\
    &= V_{\tau_1}(V_{\Delta(\tau_2\dots\tau_m)}(\phi\otimes\psi))-V_{\sum_{(f)}(\tau\curvearrowright_l f^{(1)}\otimes f^{(2)}+f^{(1)}\otimes\tau\curvearrowright_l f^{(2)})}(\phi\otimes\psi)\\
    &= V_{\sum_{(f)}(\tau\star f^{(1)}\otimes f^{(2)}+f^{(1)}\otimes \tau\star f^{(2)} - \tau\curvearrowright_l f^{(1)}\otimes f^{(2)}+f^{(1)}\otimes \tau\curvearrowright_l f^{(2)})} (\phi\otimes\psi)\\
    &=V_{\tau f^{(1)}\otimes f^{(2)}+f^{(1)}\otimes \tau f^{(2)}}(\phi\otimes\psi)\\
    &=V_{\Delta f}(\phi\otimes\psi).
\end{align*}

\end{proof}

\subsubsection{Consistency} We still need to check that $\cF$ coincides with the elementary differentials constructed in Sections \ref{subsec:elementaryDiffTensor}-\ref{subsec:elementarryDiffMKW}. To do so, we show that
\begin{itemize}
    \item If $(M,\nabla)$ is flat and torsion-free, $\cF$ is well-defined on the Grossman-Larson algebra. In this case, $\cF\vert_{\cP(\cH_{GL})}$ is equal to the unique pre-Lie map $\tilde\cF$ generated by $\tilde\cF(\tomato i) = V_i$.
    \item If $(M,\nabla)$ is flat and has constant torsion, $\cF\vert_{\cP(\cH^g_{MKW})}$ is equal to the free post-Lie algebra map $\hat\cF$ generated by $\hat\cF(\tomato i) = V_i$.
    \item In \cite{emilio22}, Armstrong et al solve \eqref{eq:RDE} for level $N=2$ in the so far most general setting. We show that even in this case, their solution agrees with our approach with the above-constructed map $\cF$.
\end{itemize}
Let us start with the flat and torsion-free case, so that $(\cF,\triangleright)$ is a pre-Lie algebra. By standard results, we know that there is a set of coordinates, such that the Christoffel symbols vanish locally. This immediately implies that for all $k\ge 1$ and a permutation $\sigma:\{1,\dots,k\}\to\{1,\dots,k\}$, we have
\begin{equation}\label{eq:commuting}
\nabla^k \cdot (U_1,\dots, U_k) = \nabla^k \cdot (U_{\sigma_1},\dots, U_{\sigma_k})
\end{equation}
for any vector fields $U_1,\dots, U_k\in\cV$. Thus, $\cF:\cH_{GL}\to\cD$ is well-defined. We can now show the following:

\begin{thm}\label{thm:flatCase}
    Assume $(M,\nabla)$ is flat and torsion-free. Then $\cF\vert_{\cP(\cH_{GL})}:(\cP(\cH_{GL}),\curvearrowright)\to(\cV,\triangleright)$ is a pre-Lie algebra map. Thus, $\cF = \tilde\cF$ on $\cP(\cH_{GL})$.
\end{thm}

\begin{proof}
    We only need to show for any trees $\tau,\sigma\in \cP(\cH_{GL}) = \myspan(\cUT)$, $\cF(\tau\curvearrowright\sigma) = \cF(\tau)\triangleright\cF(\sigma)$. Note that \eqref{graftCov} holds for unordered trees (if one replaces $\curvearrowright_l$ with $\curvearrowright$) by the same proof as for ordered trees, showing the claim.
\end{proof}

\noindent Let us now consider the case in that $(M,\nabla)$ is flat and has constant torsion, i.e. $\nabla_Z T = 0$ for all $Z\in\cV$. Then $(\cV,\brackets{\cdot,\cdot},\triangleright)$, where $\brackets{U,V} = -T(U,V)$, forms a post-Lie algebra. We claim that in this case, $\cF$ restricted to the primitive elements of $\cH_{MKW}^g$ is a post-Lie algebra map and thus equal to $\hat\cF$. To do so, we need to show that it preserves $\brackets{\cdot,\cdot}$ as well as the product $\curvearrowright_l$. Using Lemma \ref{lem:postLieTechnical}, we see that for $f,g\in\cP(\cH_{MKW}^g)$, we have
\[
    [f,g] = f\curvearrowright_l g- g\curvearrowright_l l + \brackets{f,g}\,,
\]
while for two vector fields $U,V\in\cV$, we have
\[
    [U,V] = U\triangleright V - V\triangleright U + \brackets{U,V}\,.
\]
Thus, showing that $\cF\vert_{\cP(\cH_{MKW}^g)}$ is a post-Lie algebra map is equivalent to showing that for all $\tau,\sigma\in \cP(\cH_{MKW}^g)$, we have
\begin{align*}
    \cF(\tau\curvearrowright_l \sigma) &= \cF(\tau)\triangleright\cF(\sigma)\\
    \cF([\tau,\sigma]) &= [\cF(\tau),\cF(\sigma)]\,.
\end{align*}
We show this in the following theorem:

\begin{thm}
Let $\nabla$ be flat with constant torsion. Then $\cF\vert_{\cP(\cH_{MKW}^g)} :\cP(\cH_{MKW}^g)\to\cV$ is a post-Lie algebra map. It follows that $\cF = \hat\cF$.
\end{thm}

\begin{proof}
    Theorem \ref{theo:main2} already gives us that $\cF$ is a Lie algebra map with respect to the Lie brackets $[\cdot,\cdot]$ on $\cP(\cH_{MKW}^g)$ and $\cV$, so we only need to show that $\cF(\tau\curvearrowright_l \sigma) = \cF(\tau)\triangleright\cF(\sigma)$. Furthermore, the Theorem already shows it for trees $\tau,\sigma\in\cOT$. Since any $f\in\cP(\cH_{MKW}^g)$ is a finite sum over forests, we can set $gr(f)$ to be the highest number of trees in any summand of $f$ and use induction over this number. So assume that the claim has already been shown for any $f,g$ such that $gr(f),gr(g)\le n$ for some $n\in\NN$. Since $\cP(\cH_{MKW}^g)$ is the free Lie algebra with respect to the Lie bracket $\brackets{\cdot,\cdot}$, it suffices to show for any $f,g,h$ with $gr(f),gr(g),gr(h)\le n$ that
    \begin{align}
        \cF(\brackets{f,g}\curvearrowright_l h) &= \cF(\brackets{f,g})\triangleright\cF(h) \label{eq:toShow1}\\
        \cF(h\curvearrowright_l\brackets{f,g}) &= \cF(h)\triangleright\cF(\brackets{f,g})\,.\label{eq:toShow2}
    \end{align}
    Furthermore, since $\brackets{f,g} = [f,g]-f\curvearrowright_l g- g\curvearrowright_l f$, we can switch the Lie brackets and show the following instead of \eqref{eq:toShow1}:
    \begin{equation}\label{eq:toShow3}
        \cF([f,g]\curvearrowright_l h) = \cF([f,g])\triangleright\cF(h).
    \end{equation}
    To do so, we recall the definition of flat and torsion-free:

    \begin{itemize}
        \item $(M,\nabla)$ is flat, if and only if
        \begin{equation*}
            [X,Y]\triangleright Z = X\triangleright(Y\triangleright Z) - Y\triangleright(X\triangleright Z)
        \end{equation*}
        holds for all $X,Y,Z\in\cV$\,.

        \item $(M,\nabla)$ has constant torsion, if and only if for all $\nabla_ZT = 0$. One easily calculates that this is equivalent to
        \begin{equation*}
            Z\triangleright T(X,Y) = T(Z\triangleright X,Y)+T(X, Z\triangleright Y),
        \end{equation*}
        for all $X,Y,Z\in\cV$, where $T(X,Y) = \nabla_X Y-\nabla_Y X-[X,Y]$.
    \end{itemize}

    \noindent Let us start by showing \eqref{eq:toShow3}. The flatness gives us:
    \begin{align*}
        [\cF(f),\cF(g)]\triangleright\cF(h) &= \cF(f)\triangleright(\cF(g)\triangleright\cF(h))- \cF(g)\triangleright(\cF(f)\triangleright\cF(h)) \\
        &= \cF(f\curvearrowright_l(g\curvearrowright_l h) - g\curvearrowright_l(f\curvearrowright_l h))\\
        &= \cF((f\star g -g\star f) h) \\
        &= \cF([f,g] h)\,,
    \end{align*}
    where we have used the known connection between grafting and Grossman-Larson products $f\curvearrowright_l(g\curvearrowright_l h) = (f\star g)\curvearrowright_l h$ in $\cH_{MKW}^g$, see \cite{curry20} for reference. To show \eqref{eq:toShow2}, we use the constant torsion as well as $\cF(\brackets{f,g}) = \brackets{\cF(f),\cF(g)}$ by induction hypothesis to see:
    \begin{align*}
        \cF(h)\triangleright\cF(\brackets{f,g}) &= \cF(h)\triangleright\brackets{\cF(f)\cF(g)}\\
        &= \brackets{\cF(h)\triangleright\cF(f),\cF(g)} + \brackets{\cF(f),\cF(h)\triangleright\cF(g)}\\
        &= \cF(\brackets{h\curvearrowright_l f,g} + \brackets{f,\curvearrowright_l g}) \\
        &= \cF(h\curvearrowright_l\brackets{f,g})\,,
    \end{align*}
    which finishes the proof. 
\end{proof}

\noindent Finally, let us discuss the relation to the results from \cite{emilio22}: In the paper, the authors solve the RDE \eqref{eq:RDE} for non-geometric rough paths on manifolds in the case $N=2$, by constructing rough integrals in each coordinate chart in a coordinate independent way. They manage to solve the equation in this case for the $\cH_{GL}$ algebra and any manifold $M$ for any connection $\nabla$. Note that they take their analysis a step further and discuss the case, in which $X$ lives on a manifold, as well. Since we restrict ourselves to the case in which $X$ lives in flat space, we only show that our approach gives the same solution in that case, in which their formula (3.29) simplifies to
\begin{align*}
    Y_{s,t}^k &\approx V_i^k(Y_s) \XX^i_{s,t} + V_i^\alpha\partial_\alpha V_j^k(Y_s) \XX^{ij}_{s,t} -\frac 12 V^i_\alpha V^j_\beta\Gamma_{ij}^k(Y_s) (\XX^{\alpha}_{s,t}\XX^\beta_{s,t} - \XX^{\alpha\beta}_{s,t}-\XX^{\beta\alpha}_{s,t})\\
    &= V_i^k(Y_s) \XX^{\tomato i}_{s,t} + (\tilde\nabla_{V_i} V_j)^k(Y_s) \XX^{\ladderTwo ji}_{s,t} +\frac 12\tilde\nabla^2 \phi^k(V_i,V_j)\XX^{\tomato i\tomato j} \,,
\end{align*}
where $\phi^k$ is the $k-th$ coordinate function and $\tilde\nabla$ is the torsion-free version of $\nabla$ (i.e. $\tilde\Gamma_{ij}^k = \frac 12(\Gamma_{ij}^k+\Gamma_{ji}^k)$). The $\approx$ means that the two sides agree up to an error of order $o(\abs{t-s})$. It follows that our Davie's formula \eqref{ineq:Davie}
\begin{align*}
    \phi(Y_t) &\approx \sum_{\abs\tau\le 2} \frac1{sg(\tau)}\cF(\tau)\phi(Y_s)\XX^\tau_{s,t} = \cF(\XX_{s,t})\phi(Y_s)
\end{align*}
holds, as long as we replace the connection $\nabla$ with $\tilde\nabla$. (Recall that the symbols in $\cH_{GL}$ where chosen in such a way, that $\scalar{\tau,\sigma} = sg(\tau)\delta_{\tau,\sigma}$. Thus, one needs to divide by $sg(\tau)$ in the sum.) This mirrors our analysis of $\cF$: While it always works on $\cH_{MKW}^*$, we can only solve the RDE on a manifold with a rough path in $\cH_{GL}$, if we have the commuting property \eqref{eq:commuting} holds. Having a torsion-free connection $\tilde\nabla$ ensures this property for $k\le 2$, whereas more general $k$ also require a flat connection. 

\begin{rem}
    The most interesting observation in the level $N=2$ case is that one can start with a general connection and construct a pseudo bialgebra map $\cF$ by replacing $\nabla$ with $\tilde\nabla$. This raises the question if it is possible to construct a pseudo bialgebra map for higher levels with weaker conditions than flatness and torsion-freeness. It should be noted that such a map could no longer fulfill the condition $\cF(\tau\curvearrowright\sigma) = \cF(\tau)\triangleright\cF(\sigma)$ for any trees $\tau,\sigma\in\cUT$, since $\cP(\cH_{GL})$ is the free pre-Lie algebra, so the above condition forces $\cF$ to be as in \eqref{myMap}. However, as long as one does not mind losing that condition, it is an open problem whether one can construct a pseudo bialgebra map from a general connection $\nabla$. In fact, \cite{ferrucci22} succeded at constructing rough integrals on manifolds, giving us high hope that it is possible.
\end{rem}

\section{Discussion of rough paths on manifolds}\label{sec:RPonManifolds}

\noindent One of the critical aspects of geometric rough paths theory is, that the solution to an RDE can be seen as a rough path itself. For an RDE on a manifold $M$, it follows that the solution to \eqref{eq:RDE} should give us a rough path on a manifold, a concept which was introduced in \cite{RPonManifolds} and refined in \cite{RPonManifoldsNewDef}. In this section, we want to analyze whether there is a canonical rough path living over our solution $Y_t$.

For a non-geometric rough path, we can almost immediately answer this in the negative: At the current state, it is unclear how to define a rough path over a general Hopf algebra on a manifold. Even in the branched case, one needs an additional structure called a bracket extension \cite{bracketExtension}. 

\begin{rem}\label{rem:branched}
    It should be noted that \cite{ferrucci22} successfully constructed branched rough paths on manifolds with said bracket extension. Furthermore, \cite{branchedIto} Remark 2.15 manages to construct a canonical bracket extension, thus making sense of branched rough paths on manifolds in a canonical way. It would be an interesting follow-up project to see how the results of this section can be transferred to branched rough paths, and how this connects to pseudo bialgebra maps.
\end{rem}

\noindent However, in this paper, we will restrict ourselves to geometric rough paths on manifolds. Let us start by recalling the notion of geometric rough paths on manifolds, before showing that there is indeed a rough path on $\YY$ $M$ such that its trace fulfills $\YY^i_t = Y^i_t \approx \cF(\XX_{s,t})\phi^i(Y_t)$, where $\phi^i$ is the $i$-th component of some coordinate function $\phi$.

\begin{rem}
    It should be noted that geometric rough path theory on manifolds is well understood. It is further clear, that the classical solution to an RDE on a manifold should have a $Y_t$ as a trace, which fulfills our Davie's formula $\phi(Y_t)\approx \cF(\XX_{s,t})\phi(Y_s)$ for any smooth function $\phi\in C^\infty(M)$ since our solution agrees with the classical solution in the geometric rough path case. The only real new proof in this section is a new proof of Corollary \ref{cor:RPonManifolds}, using mainly algebraic arguments.

    However, we still want to present this section as a possible starting point towards understanding general rough path solutions on manifolds as rough paths themselves, especially considering Remark \ref{rem:branched}.
\end{rem}

\subsection{Recall: Geometric rough paths on manifolds}

We use the view-point used in \cite{RPonManifoldsNewDef} (also used in \cite{emilio22} and\cite{ferrucci22} for non-geometric rough paths), in which rough paths on manifolds are a collection of rough paths in coordinate sheets $(\XX_i,\phi_i)$, which is consistent under the push-forward of the coordinate change functions: $(\phi_j\circ\phi_i^{-1})_*\XX_j = \XX_i$ up to some restriction on the times $s,t$. To define this properly, we need to introduce the notion of the push-forward, which requires the integration of a rough path against a one-form, which requires the notion of half shuffles. Let us start with introducing half shuffles:

Given two words $w,u$, the ordered shuffle $u\tilde\shuffle w$ is given as the sum over all words with the same letters as $uw$, such that the order inside $u$ and $w$ is preserved and the order of the last letters $u_{\abs u}$ and $w_{\abs w}$ is preserved. Formally, that is given by
\begin{equation*}
    u\tilde\shuffle w = (u\shuffle \bar w)w_{\abs w}\,,
\end{equation*}
where $w= (\bar w,w_{\abs w})$ for a word $\bar w$ and a letter $w_{\abs w}$. Conversely, we define the set of ordered deshuffles as follows: $\tilde\Delta^n(w)$ is the set of all splittings of $w$ into $n$ many non-empty words $(u_1,\dots,u_n)$, such that
\begin{itemize}
    \item $w\in Sh(u_1,\dots, u_n)$, where $Sh(u_1,\dots,u_n)$ is the set of all words with the same letters as $u_1\dots u_n$, such that the order of letters in all $u_i, i=1,\dots,n$ is preserved. It especially holds that $u_1\shuffle \dots\shuffle u_n = \sum_{w\in Sh(u_1,\dots, u_n)} w$, as long as no letter in $u_1,\dots, u_n$ appears more than once.
    \item The last letters of $u_1,\dots,u_n$ is ordered as it is in the word $w$.
\end{itemize}
For example, we have that
\begin{equation*}
    \tilde\Delta^2(123) = \{(1,23), (2,13), (12,3)\}\,.   
\end{equation*}
Using this notation, we can now introduce the integral against a one-form $\nu$: In the original paper \cite{originalArticleRP}, Terry Lyons shows that the integral of a one-form against a rough path can be seen as another rough path. To do so, he constructs so-called \emph{almost multiplicative functionals} out of $\nu$ and $\XX$. Almost multiplicative functionals uniquely give rise to rough paths, which are characterized by the fact that $\abs{\XX^w-\YY^w}\in o(\abs{t-s})$. We write $\XX_{s,t}^w\approx \YY_{s,t}^w$.

Using the ordered shuffle, we split Definition 3.2.2 from \cite{originalArticleRP} into each word to get the following definition:

\begin{defn}
    Let $\nu:\RR^n\to L(\RR^n,\RR^d)$ be a smooth one-form and let $\XX$ be a geometric rough path over $\RR^n$. The integral $\int\nu(d\XX)$ is given by the unique rough path $\YY$, such that for each word $w$, we have
    \begin{equation}\label{eq:pushForward}
        \YY^w_{s,t} \approx \sum_{\abs u\ge\abs w} \left(\sum_{(s_1,\dots,s_{\abs w})\in\tilde\Delta^{\abs w} (u)}\nu_{s_1}^{w_1}\dots\nu_{s_{\abs w}}^{w_{\abs w}} (X_s)\right)\XX_{s,t}^u\,.
    \end{equation}
    We denote $\XX^\nu := \YY$.
\end{defn}

\noindent Here $\nu_w$ is defined as follows: Note that $\nu$ is an $\RR^d$ valued one form on $\RR^n$ (i.e. $\nu:\RR^n\to L(\RR^n,\RR^d)$). Thus, it can be written as $\nu(x) = \sum_{i=1}^n \nu_i(x) d^i$, where $\nu_i :\RR^n\to\RR^d$ are smooth functions. We set $\nu_w := \partial_{w_1}\dots\partial_{w_{\abs w-1}} \nu_{w_{\abs w}}$.

Note that any smooth function $\phi:\RR^n\to\RR^d$ immediately generates a one-form $d\phi = \sum_{i=1}^n\partial_i\phi d^i$ by differentiating $\phi$. With this, one can set the push-forward of $\phi$ to be:

\begin{defn}
    The push-forward of $\XX$ under $\phi$ is the rough path given by $\XX^{d\phi}$.
\end{defn}

\noindent This leads to the following definition of a geometric rough path on a manifold:

\begin{defn}\cite{RPonManifoldsNewDef}
    A (geometric) rough path on a manifold over the interval $J$ is a finite collection $(x_i,\XX_i, J_i, (\phi_i, U_i))$, such that
    \begin{itemize}
        \item $U_i \subset \RR^d$ are open sets and $\phi_i: M\supset V_i\to U_i$ are diffeomorphisms.
        \item $(J_i)_i$ are intervalls and a compact cover of the interval $J$.
        \item $x_i$ is a path in $U_i$ with $x_i^j(t)-x_i^j(s) = \XX_{i;s,t}^j$ for each letter $j=1,\dots, n$. (That is, $x_i$ is a trace of $\XX_i$).
        \item $\XX_i$ is a geometric rough path over the interval $J_i$ over $\RR^n$.
        \item {\bf Consistency condition:} It holds that $\XX_i = (\phi_j\circ\phi_i^{-1})_* \XX_j$ on $J_i\cup J_j$, as long as this interval is not empty.
    \end{itemize}
\end{defn}

\subsection{Solutions to RDEs are rough paths}

In this section, we recall the construction of rough path solutions on manifolds and present a short, new proof that the solution to \eqref{eq:RDE} is indeed a rough path on $M$. We furthermore briefly discuss the connection to the pseudo bialgebra map $\cF$. The main idea towards solving an $RDE$ on $M$ is to express the vector fields $V_i$ in some coordinate chart $V_i^\phi = V_i \phi^k\partial_k$ for each $i=1,\dots, n$. One can then solve the RDE
\begin{equation*}
    d(\phi(Y_t)) = V^\phi_i(Y_t) d\XX^{i}_{t}
\end{equation*}
as a classical RDE over flat space for each coordinate function $\phi$. Doing so leads to the following terms for each word $w\in T(\RR^d), s\le t$ and coordinate function $\phi$: 
\begin{equation}\label{Y_phi}
    \YY_{s,t}^{\phi,w} \approx \sum_{\abs u\ge\abs w}\sum_{(s_1,\dots,s_{\abs w})\in\tilde\Delta^{\abs w}(u)}V_{s_1}\phi^{w_1}(Y_s)\dots V_{s_{\abs w}}\phi^{w_{\abs w}}(Y_s) \XX_{s,t}^u\,,
\end{equation}
where we sum over words $u$ of length less or equal to $N$ as well as all ordered splittings $(s_1,\dots,s_{\abs w})$ of $u$ into $\abs{w}$ many non-empty words. This gives a unique rough path in each coordinate chart, as long as we also choose a starting point $y_0\in M$ (becoming $\phi(y_0)$ in the coordinate chart). See \cite{originalArticleRP} for reference. Note that, as we expected, we have for all 1-letter words $i=1,\dots, d$:
\begin{align*}
    \phi(Y_t) &\approx \sum_{0\le \abs w\le N} V_{w_1}\circ\dots\circ V_{w_{\abs w}}\phi(Y_s) \XX^w_{s,t} \\
    &= \cF(\XX_{s,t})\phi(Y_s)\,,
\end{align*}
showing that the trace on $M$ is indeed the same path we get with our notion of solution. We also see that $\YY$ does indeed seem to fulfill some form of Davie's formula:
\begin{equation*}
    \YY^{\phi,w}_{s,t} \approx \tilde\cF(\XX_{s,t})\phi(Y_s)\,,
\end{equation*}
where for each word $u$, $\tilde\cF(u)$ maps functions $\phi\in C^\infty(M,\RR^d)$ onto functions in $C^\infty(M,T(\RR^d))$ via
\begin{equation*}
    \scalar{\tilde\cF(u)\phi,w} = \sum_{(s_1,\dots,s_{\abs w})\in\tilde\Delta^{\abs w}(u)}V_{s_1}\phi^{w_1}(Y_s)\dots V_{s_{\abs w}}\phi^{w_{\abs w}}(Y_s)\,.
\end{equation*}
While this indicates that for more general algebra, one should try to find maps $\tilde\cF:C^\infty(M,\RR^d)\to C^\infty(M,\cH)$ for some Hopf algebra $\cH$, such that the above Davie's formula has the correct trace and is ``coordinate independent''. However, as stated above, it is at the moment not clear how to push forward a group-like element for a general Hopf algebra from one coordinate chart to another, so ``coordinate independent'' is ambiguous here. For the geometric case, everything is well-defined and we will spend the rest of this section presenting a new proof, showing the coordinate independence: To do so, we show the following preliminary result:

\begin{lem}\label{lem:algebraicTechnical}
    For two given words $u, w$ with $\abs w\le \abs u$, consider the two sets
    \begin{align*}
        A &:= \{(t,s,z) ~\vert~ (t_1,\dots,t_{\abs{w}})\in\tilde\Delta^{\abs w}(u), 1\le\abs{s_i}\le\abs{t_i}, (z^i_1,\dots,z^i_{\abs{s_i}})\in\tilde\Delta^{\abs{s_i}}(t_i) \text{ for }i=1,\dots,\abs w\}\\
        B &:= \{(v,s,z)~\vert~ \abs u\ge\abs v\ge\abs w, (s_1,\dots,s_{\abs w})\in\tilde\Delta^{\abs w}(v), (z_1,\dots, z_{\abs v})\in\tilde\Delta^{\abs v}(u)\}\,.
    \end{align*}
    Then there are maps $v(t,s,z)$ and $t(v,s,z)$, such that $(t,s,z)\mapsto (v(t,s,z),s,z)$ as well as $(v,s,z)\mapsto (t(v,s,z),s,z)$ are inverse to each other and thus bijections.
\end{lem}

\begin{proof}
    We start by constructing $v(t,s,z)$: Let $(t,s,z)\in A$ and denote the index set of the $z^i_k$ by $I := \{(i,k_i)~\vert~ i=1,\dots,\abs{w}, k_i = 1,\dots,\abs{s_i}\}$. Let $\sigma: \{1,\dots, m\}\to I$ be the unique, bijective map such that $z_{\sigma(1)},\dots,z_{\sigma(m)}$ are ordered in such a way, that their last letters have the same order as in $u$. (Here we assume that the letters of u are colored in the sense that we can differentiate all letters in $u$, even if $u$ contains the same letter several times.)

    We then define $v := s_{\sigma(1)}\dots s_{\sigma(m)}$, where $s_{i,k_i}$ is the $k_i$-th letter of the word $s_i$. It follows that $(s_1,\dots, s_{\abs{w}})$ is a splitting of $v$. Furthermore, note that the last letters of $z^i_{\abs s_i}$ are just the last letters of $t_i$ for $i=1,\dots,\abs w$, which are ordered by the order of $u$. Thus, $\sigma$ does not change the ordering of $(z^i_{\abs s_i})_{i=1,\dots, \abs w}$, which means that the last letters of $s_1,\dots,s_{\abs w}$ are ordered as in $v$. It follows that $(s_1,\dots,s_{\abs w})\in\tilde\Delta^{\abs w}(v)$. Furthermore, one easily sees that $\abs u\ge\abs v\ge\abs w$, since $1\le\abs {s_i}\le\abs{t_i}$. It follows, that $\zeta(t,s,z) := (v(t,s,z),s,z)\in B$.

    We show that it is bijective by constructing $t(v,s,z)$ such that $\zeta^{-1}(v,s,z) = (t(v,s,z),s,z)$. Given $(v,s,z)$, let $\sigma: I\to \{1,\dots,\abs{v}\}$ be the unique bijection, such that $s(i,k_i) = v_{\sigma(i,k_i)}$ for all $(i,k_i)\in I$. We then set $z^i_{k_i} := z_{\sigma(i,k_i)}$. For each $i$, we then set $t_i$ to be the word containing the same letters as $z^i_1,\dots, z^i_{\abs{s_i}}$, where we reorder the letters such that they have the same order as in $u$. It is straight-forward to see that $v(t(v,s,z),s,z) = v$, so $\zeta$ is bijective given that $(t(v,s,z),s,z)\in A$. To show this, note that $(t_1,\dots, t_{\abs w})$ is a splitting of $u$ since $(z_1,\dots,z_{\abs v})$ is one. Further, since $(s_1,\dots, s_{\abs w})$ respected the order of last letters in $v$, $\sigma$ does not change the order of $z^i_{\abs {s_i}}$, which thus still have last letters respecting the order of $u$. It follows that $(t_1,\dots, t_{\abs w})$ respects this order and is thus in $\tilde\Delta(u)$. Furthermore, $1\le \abs{s_i}\le \abs{t_i}$ and since $(z_1,\dots,z_{\abs v})\in\tilde\Delta^{\abs v}(u)$ respect the ordering of $u$, we get that they respect the ordering of the subwords $t_i$ ($(z^i_1,\dots z^i_{\abs {s_i}})$ is obviously a spitting of $t_i$.) Thus, $(t(v,s,z),s,z)\in A$, which finishes the proof.
\end{proof}

\noindent A simple corollary of this is, that a rough path in a single coordinate sheet gives rise to a rough path on $M$: Given an atlas $(\phi_i, U_i)$ and a rough path $\XX$ in $U_1$, we set $\XX^{\phi_i} := (\phi_i\circ \phi_1^{-1})_*\XX$ on the respective interval. The following holds:

\begin{cor}\label{cor:RPonManifolds}
    $(\XX^{\phi_i},\phi_i,U_i)$ is a rough path on $M$.
\end{cor}

\begin{proof}
    The only thing one needs to check is the consistency property. To do so, it suffices to check that the push-forward factorizes. Since this result is not the main focus of this section, but simply a nice observation, we move the proof of the factorization of the push-forward in the appendix and show it in Proposition \ref{prop:appendix}.
\end{proof}

\noindent To show that $\YY$ is indeed a rough path on a manifold, we need one more technical result: We will need to calculate the general derivative of compositions $\psi\circ\phi$ of smooth functions $\phi,\psi$.

\begin{lem}\label{lem:chainRuleWords}
    Let $w$ be any word and $\phi,\psi$ be smooth functions. It holds that
    \begin{equation*}
        \partial_w (\psi\circ\phi)(x) = \sum_{1\le\abs{v}\le \abs w} (\partial_v \psi)(\phi(x))\left(\sum_{(s_1,\dots, s_n)\in \tilde\Delta^{\abs v}(w)}\partial_{s_1}\phi^{v_1}(x)\dots\partial_{s_{\abs v}} \phi^{v_{\abs v}}(x)\right)\,,
    \end{equation*}
    where $\partial_w = \partial_{w_1}\dots\partial_{w_{\abs w}}$.
\end{lem}

\begin{proof}
    For one-letter words $w=1,\dots, n$, this is just the chain-formula. For longer words, we do inductively get for $i=1,\dots n$ and any word $w$:
    \begin{align*}
        \partial_{iw}(\psi\circ\phi)(x) &= \sum_{1\le\abs{v}\le \abs w}\sum_{j=1}^n (\partial_{jv} \psi)(\phi(x))\left(\sum_{(s_1,\dots, s_n)\in \tilde\Delta^{\abs v}(w)}\partial_i \phi^j(x)\partial_{s_1}\phi^{v_1}(x)\dots\partial_{s_{\abs v}} \phi^{v_{\abs v}}(x)\right)\\
        &\quad + \sum_{1\le \abs{v}\le \abs w} (\partial_v \psi)(\phi(x))\left(\sum_{(s_1,\dots, s_n)\in \tilde\Delta^{\abs v}(w)}\sum_{j=1}^{\abs v}\partial_{s_1}\phi^{v_1}(x)\dots\partial_{is_j}\phi^{v_j}(x)\dots\partial_{s_{\abs v}} \phi^{v_{\abs v}}(x)\right)\\
        &= \sum_{1\le\abs{v}\le \abs {iw}} (\partial_v \psi)(\phi(x))\left(\sum_{(s_1,\dots, s_n)\in \tilde\Delta^{\abs v}(iw)}\partial_{s_1}\phi^{v_1}(x)\dots\partial_{s_{\abs v}} \phi^{v_{\abs v}}(x)\right)\,.
    \end{align*}
\end{proof}

\noindent We can now show our main result: That $\YY$ is indeed a rough path on $M$. The only non-trivial part of that is the consistency condition: Let $\phi,\psi$ be two coordinate charts such that their supports on $M$ overlap and let $\nu:\psi\circ\phi^{-1}$ be the coordinate change map. We claim that $\YY_{s,t}^\psi = \nu_*\YY^\phi_{s,t}$ holds, which shows the following theorem:

\begin{thm}
    Given an atlas $\phi\in\Phi$, $(\YY^\phi)_{\phi\in\Phi}$ from \eqref{Y_phi} is a geometric rough path on $M$.
\end{thm}

\begin{proof}
    Let us first calculate the push-forward $\nu_*\YY^\phi:$
    \begin{align}
    \begin{split}\label{nu*Y}
        \nu_*\YY^{\phi,w} &\approx \sum_{\abs u\ge \abs w}\sum_{(s_1,\dots,s_{\abs w})\in\tilde\Delta^{\abs w}(u)}\partial_{s_1}\nu^{w_1} \dots \partial_{s_{\abs w}}\nu^{w_{\abs w}} \YY^{\phi,u} \\
        &\approx\sum_{\abs v\ge\abs u\ge\abs w} \left(\sum_{(s_1,\dots,s_{\abs w})\in\tilde\Delta^{\abs w}(u)}\partial_{s_1}\nu^{w_1} \dots \partial_{s_{\abs w}}\nu^{w_{\abs w}}\right)\\
        &\qquad \left(\sum_{(z_1,\dots z_{\abs u})\in\tilde\Delta^{\abs u}(v)} V_{z_1}\phi^{u_1}\dots V_{z_{\abs u}} V\phi^{u_{\abs u}}\right)\XX^v 
    \end{split}
    \end{align}
    On the flip side, we can express $\psi$ as $\psi = \nu\circ\phi$. Together with Lemma \ref{lem:chainRuleWords}, this gives us
    \begin{align}
    \begin{split}\label{Yphi}
        \YY^{\psi,w} &\approx \sum_{\abs v\ge\abs w} \sum_{(t_1,\dots,t_{\abs w})\in\tilde\Delta^{\abs w}(v)}V_{t_1}(\nu\circ\phi)^{w_1}\dots V_{t_{\abs w}}(\nu\circ\phi)^{w_{\abs w}} \XX^v\\
        &\approx\sum_{\abs v\ge\abs w} \sum_{(t_1,\dots,t_{\abs w})\in\tilde\Delta^{\abs w}(v)}\left(\sum_{\abs{s_1}\le \abs{t_1}\dots \abs{s_{\abs w}}\le\abs{t_{\abs w}}} \partial_{s_1}\nu^{w_1}\dots\partial_{s_{\abs w}}\nu^{w_{\abs w}}\right)\\
        &\qquad\left(\sum_{(z^1_1,\dots,z^1_{\abs {s_1}})\in\tilde\Delta^{\abs {s_1}}(t_1),\dots,(z^{\abs w}_1,\dots,z^{\abs w}_{\abs{s_{\abs w}}})\in\tilde\Delta^{\abs{s_{\abs w}}}(t_{\abs w})} V_{z^1_1}\phi^{s_{1,1}}\dots V_{z^{\abs w}_{\abs {s_{\abs w}}}} \phi^{s_{\abs w,\abs{s_{\abs w}}}}\right)\XX^v\,.
    \end{split}
    \end{align}
    For fixed $w,v$, \eqref{nu*Y} can be written as
    \[
        \YY^{\psi,w} = \sum_{\abs{v}\ge \abs w} \left(\sum_{(u,s,z)\in B} \partial_{s_1}\nu^{w_1} \dots \partial_{s_{\abs w}}\nu^{w_{\abs w}} V_{z_1}\phi^{u_1}\dots V_{z_{\abs u}} V\phi^{u_{\abs u}}\right)\XX^v
    \]
    and \eqref{Yphi} can be written as
    \[
        \YY^{\psi,w} = \sum_{\abs{v}\ge \abs w} \left(\sum_{(t,s,z)\in A} \partial_{s_1}\nu^{w_1} \dots \partial_{s_{\abs w}}\nu^{w_{\abs w}} V_{z^1_1}\phi^{s_{1,1}}\dots V_{z^{\abs w}_{\abs {s_{\abs w}}}} \phi^{s_{\abs w,\abs{s_{\abs w}}}}\right)\XX^v
    \]
    Thus, Lemma \ref{lem:algebraicTechnical} gives us that \eqref{Yphi} and \eqref{nu*Y} are the same, showing the claim.
\end{proof}

\subsection{Discussion: General rough paths are geometric rough paths}

Something an attentive reader might have noticed is that any rough path can be lifted to a geometric rough path over a larger index set: The Theorem of Milnor-Moore gives us, that any rough path lives in a universal enveloping algebra, which is given by $U(P) = T(P)/I$ by Proposition \ref{prop:UisShuffle}. By the extension theorem \cite{extensionTheorem}, we can thus find a rough path lift $\tilde\XX$ into $T(P)$, making it a non-homogeneous geometric rough path over $P$. 

\begin{rem}
    It should be noted that the extension theorem was only proven in \cite{extensionTheorem} for the homogeneous case. However, the proof that any inhomogeneous rough path in $U(P)$ can be extended to a rough path in $T(P)$ works exactly the same.
\end{rem}

\noindent Thus, using the last section, we can see the solution to any rough differential equation as a geometric rough path $\YY$ simply by replacing
\begin{equation*}
    dY_s = V_i(Y_s)d\XX_s^i
\end{equation*}
with
\begin{equation*}
    dY_s = \sum_{p\in \tilde P} \cF(p)(Y_s)d\tilde\XX_s^p\,,
\end{equation*}
where $\tilde P$ is a basis of $P^N$. However, the problem with this construction is that in general, $\tilde\XX$ is not unique, and $\YY$ fundamentally relies on the lift. This can easily be seen by using a geometric rough path: In this case, $\XX$ should live in $T(\RR^n)$, a much smaller space than $T(P)$.

\begin{rem}
    We should note that in the special case of branched rough paths, this approach should be quite successful: For non-planarly branched rough paths \cite{IsomorphismUnordered}, it is shown that $\cH_{GL}$ is isomorphic to $T(B)$, where $B$ is a subspace of $\myspan(\cUT) = \cP(\cH_{GL})$. This shows that any non-planarly branched rough path can be seen as a geometric rough path over $B$. And \cite{rahm2023} shows in Section 6 that any planarly branched rough path is given by a geometric rough path over $\myspan(\cOT)$. Thus, one can use the above results to make sense of solutions to branched RDEs as (geometric) rough paths on manifolds.
\end{rem}

\noindent However, for general $\cH$ it does not seem likely that this approach should lead to a unique rough path $\YY$ over $Y$. Indeed, if $\XX$ is not unique, one easily sees that $\YY$ can not be unique unless the vector fields $V_i, i=1,\dots, n$ interact with $\XX$ in such a way that all the additional information in $\tilde\XX$ vanishes. Indeed, if we have that $\tilde\XX_{s,t}^u$ is not unique for some word with grade $\abs{u} = N$ equal to the level of the rough path, \eqref{Y_phi} becomes for any word $w$ with grade $\abs{w} = N$:
\begin{equation*}
    \YY^{\phi,w}_{s,t} = \sum_{\abs{u} = \abs{w}} V_{u_1}\phi^{w_1}\dots V_{u_{\abs{u}}} \phi^{w_{\abs w}} (Y_s) \tilde \XX^u_{s,t}\,.
\end{equation*}
By choosing appropriate vector fields $V_i$, we can always ensure that $\YY^w$ depends on the non-unique $\tilde\XX^u$.

\appendix

\section{The push-forward factorizes}

\noindent The goal of this section is to show that given a geometric rough path $\XX$ over $\RR^d$ and two smooth functions $\phi:\RR^d\to\RR^n$, $\psi:\RR^n\to\RR^m$, we want to show that
\begin{equation*}
    \psi_*(\phi_*\XX) = (\psi\circ\phi)_*\XX\,.
\end{equation*}
Before we start the proof, let us expand both sides of the equation: Let $\YY = \psi_*(\phi_*) \XX$. By applying \eqref{eq:pushForward} twice, we get that
\begin{align}
\begin{split}\label{eq:Y}
    \YY^w_{s,t} &\approx \sum_{\abs v\ge \abs w} \left(\sum_{(s_1,\dots,s_{\abs w})\in\tilde\Delta^{\abs w}(v)} \partial_{s_1}\psi^{w_1}\dots\partial_{s_{\abs w}}\psi^{w_{\abs{w}}}(\phi(X_s))\right) \phi^*\XX^v_{s,t} \\
    &\approx \sum_{\abs u \ge \abs v\ge \abs w} \left(\sum_{(s_1,\dots,s_{\abs w})\in\tilde\Delta^{\abs w}(v)} \partial_{s_1}\psi^{w_1}\dots\partial_{s_{\abs w}}\psi^{w_{\abs{w}}}(\phi(X_s))\right)\\ &\qquad \left(\sum_{(z_1,\dots,z_{\abs v})\in\tilde\Delta^{\abs v}(u)} \partial_{z_1}\phi^{v_1}\dots\partial_{z_{\abs v}}\phi^{v_{\abs{v}}}(X_s)\right) \XX^u
\end{split}
\end{align}
holds for all words $w$ and $s\le t$. On the other side, let $\ZZ := (\psi\circ\phi)_*\XX$. One can then calculate that
\begin{align}
\begin{split}\label{eq:Z}
    \ZZ^w &\approx \sum_{\abs u\ge\abs w}\left(\sum_{(s_1,\dots,s_{\abs w})\in\tilde\Delta^{\abs w}(u)} \partial_{s_1}(\psi\circ\phi)^{w_1}\dots\partial_{s_{\abs w}}(\psi\circ\phi)^{w_{\abs{w}}}(X_s)\right)\XX^u\\
    &= \sum_{\abs u\ge \abs w} \bigg(\sum_{(t_1,\dots,t_{\abs w})\in\tilde\Delta^{\abs w}(u)}\sum_{\abs{s_1}\le\abs{t_1},\dots,\abs{s_{\abs w}}\le\abs{t_{\abs w}}} \partial_{s_1}\psi^{w_1}\dots\partial_{s_{\abs w}}\psi^{w_{\abs w}}(\phi(X_s)) \\
    &\qquad \sum_{(z^1_1,\dots,z^1_{\abs{s_1}})\in\tilde\Delta^{\abs {s_1}}(t_1),\dots,(z^{\abs w}_1,\dots z^{\abs w}_{\abs{s_{\abs w}}})\in\tilde\Delta^{\abs{s_{\abs w}}}(t_{\abs{w}})}\partial_{z^1_1} \phi^{s_{1,1}}\dots\partial_{z^{\abs w}_{\abs{s_{\abs w}}}}\phi^{s_{\abs w, \abs{s_{\abs w}}}}(X_s)\bigg)\XX^u\,.
\end{split}
\end{align}
where we used Lemma \ref{lem:chainRuleWords}. We claim that both are the same:

\begin{prop}\label{prop:appendix}
    The push-forward factorizes. That is, for any geometric rough path and $\psi,\phi$ as above, we have that $\psi_*(\phi_*\XX) = (\psi\circ\phi)_*\XX$.
\end{prop}

\begin{proof}
    We need to show, that $\ZZ = \YY$. To do so, it suffices to show that the RHS of \eqref{eq:Z} and \eqref{eq:Y} coincide. Thus, we need to show that for any words $w,u$, we have that
    \begin{align*}
        &\sum_{(t_1,\dots,t_{\abs w})\in\tilde\Delta^{\abs w}(u)}\sum_{\abs{s_1}\le\abs{t_1},\dots,\abs{s_{\abs w}}\le\abs{t_{\abs w}}} \partial_{s_1}\psi^{w_1}\dots\partial_{s_{\abs w}}\psi^{w_{\abs w}}(\phi(X_s)) \\
    &\qquad \sum_{(z^1_1,\dots,z^1_{\abs{s_1}})\in\tilde\Delta^{\abs {s_1}}(t_1),\dots,(z^{\abs w}_1,\dots z^{\abs w}_{\abs{s_{\abs w}}})\in\tilde\Delta^{\abs{s_{\abs w}}}(t_{\abs{w}})}\partial_{z^1_1} \phi^{s_{1,1}}\dots\partial_{z^{\abs w}_{\abs{s_{\abs w}}}}\phi^{s_{\abs w, \abs{s_{\abs w}}}}(X_s)\\
    &= \sum_{\abs u \ge \abs v\ge \abs w} \left(\sum_{(s_1,\dots,s_{\abs w})\in\tilde\Delta^{\abs w}(v)} \partial_{s_1}\psi^{w_1}\dots\partial_{s_{\abs w}}\psi^{w_{\abs{w}}}(\phi(X_s))\right)\\ &\qquad \left(\sum_{(z_1,\dots,z_{\abs v})\in\tilde\Delta^{\abs v}(u)} \partial_{z_1}\phi^{v_1}\dots\partial_{z_{\abs v}}\phi^{v_{\abs{v}}}(X_s)\right)
    \end{align*}
    By noting that both expressions are sums over
    \begin{equation*}
        \partial_{s_1}\psi^{w_1}\dots\partial_{s_{\abs w}}\psi^{w_{\abs{w}}}(\phi(X_s))\partial_{z_1}\phi^{v_1}\dots\partial_{z_{\abs v}}\phi^{v_{\abs{v}}}(X_s)
    \end{equation*}
    Using that the map $\zeta:A\to B$ from Lemma \ref{lem:algebraicTechnical} is a bijection, we see that both formulas are the same, showing the claim.
\end{proof}

\bibliographystyle{alpha}
\bibliography{ref}

\end{document}